\documentclass[ a4paper, onecolumn, superscriptaddress,10pt, accepted=2022-02-14, issue=3, volume=4, shorttitle=papers/compositionality-4-3 ]{compositionalityarticle}
\pdfoutput=1

\usepackage{amssymb,amsmath,stmaryrd,mathrsfs,amsthm}
\usepackage{comment}
\usepackage[mathscr]{euscript}
\usepackage[neveradjust]{paralist}
\usepackage{mathtools}
\usepackage{multirow}
\usepackage[outline]{contour}

\contourlength{1.2pt}
\usepackage{tikz}
\usetikzlibrary{intersections,arrows.meta,calc,quotes,math,decorations.pathreplacing,decorations.markings,cd,arrows,positioning,fit,matrix,
shapes.geometric,external,decorations.pathmorphing,backgrounds,circuits,circuits.ee.IEC,shapes}
\usepackage[a4paper,top=3cm,bottom=3cm,inner=3cm,outer=3cm]{geometry}
\usepackage{graphicx}
\usepackage{stackrel}

\usepackage{xcolor}
\usepackage{framed,color}

\usepackage{hyperref}

\usepackage[draft]{fixme}
\usepackage[capitalize]{cleveref}
\crefname{equation}{}{}
\crefname{defn}{Definition}{Definitions}
\crefname{thm}{Theorem}{Theorems}
\crefname{lem}{Lemma}{Lemmas}

\definecolor{rewritecolor}{rgb}{0,.9,1}
\tikzset{rewritenode/.style={shape=circle,fill=rewritecolor,scale=0.25,font=\Huge}}
\tikzset{RWopen/.style={shape=circle,draw=black,fill=white,scale=0.5,font=\Huge}}
\tikzset{RWclosed/.style={shape=circle,fill=black,scale=0.5,font=\Huge}}
\tikzset{CDnode/.style={shape=circle,fill=white,scale=.5}}
\makeatletter
\let\ea\expandafter

\pgfdeclarelayer{edgelayer}
\pgfdeclarelayer{nodelayer}
\pgfsetlayers{edgelayer,nodelayer,main}

\definecolor{lblue}{rgb}{0,250,255}
\tikzstyle{species}=[circle,fill=yellow,draw=black,scale=1.15]
\tikzstyle{transition}=[rectangle,fill=lblue,draw=black,scale=1.15]
\tikzstyle{inarrow}=[->, >=stealth, shorten >=.03cm,line width=1.5]
\tikzstyle{empty}=[circle,fill=none, draw=none]
\tikzstyle{inputdot}=[circle,fill=purple,draw=purple, scale=.25]
\tikzstyle{inputarrow}=[->,draw=purple, shorten >=.05cm]
\tikzstyle{simple}=[-,draw=purple,line width=1.000]
\tikzstyle{none}=[inner sep=0pt]

\tikzset{->-/.style={decoration={
  markings,
  mark=at position .5 with {\arrow{>}}},postaction={decorate}}}

\def\mdef#1#2{\ea\ea\ea\gdef\ea\ea\noexpand#1\ea{\ea\ensuremath\ea{#2}}}
\def\alwaysmath#1{\ea\ea\ea\global\ea\ea\ea\let\ea\ea\csname your@#1\endcsname\csname #1\endcsname
  \ea\def\csname #1\endcsname{\ensuremath{\csname your@#1\endcsname}}}

\mdef\fahat{\hat{\fa}}

\newcommand{\op}{^{\mathrm{op}}}

\alwaysmath{ell}
\alwaysmath{infty}
\alwaysmath{odot}
\def\frc#1/#2.{\frac{#1}{#2}}   
\mdef\ten{\mathrel{\otimes}}
\newcommand{\simrightarrow}{\xrightarrow{\raisebox{-3pt}[0pt][0pt]{\ensuremath{\sim}}}}

\newcommand{\N}{\mathbb{N}}
\newcommand{\R}{\mathbb{R}}

\newcommand*{\graysquare}{\textcolor{lightgray}{\blacksquare}}

\newcommand{\To}{\Rightarrow}

\let\maps\colon

\def\defthm#1#2{%
  \newtheorem{#1}{#2}[section]%
  \expandafter\def\csname #1autorefname\endcsname{#2}%
  \expandafter\let\csname c@#1\endcsname\c@thm}
\newtheorem{thm}{Theorem}[section]

\defthm{cor}{Corollary}
\defthm{prop}{Proposition}
\defthm{lem}{Lemma}
\defthm{conj}{Conjecture}
\defthm{hyp}{Hypothesis}
\defthm{fact}{Fact}
\defthm{defn}{Definition}
\defthm{notn}{Notation}
\defthm{rmk}{Remark}
\defthm{eg}{Example}

\newcommand{\fhat}{\ensuremath{\hat{f}}}

\mdef\fchk{\check{f}}

\newcommand{\upback}{\mathbin{\rotatebox[origin=c]{135}{$\ulcorner$}}}

\usepackage[all,2cell,cmtip]{xy}
\UseAllTwocells

\newcommand{\La}{\mathcal{L}}

\newcommand{\Set}{\mathsf{Set}}
\newcommand{\Graph}{\mathsf{Graph}}

\newcommand{\Petri}{\mathsf{Petri}}

\newcommand{\A}{\mathsf{A}}
\newcommand{\C}{\mathsf{C}}
\newcommand{\J}{\mathsf{J}}

\newcommand{\D}{\mathsf{D}}
\newcommand{\X}{\mathsf{X}}

\newcommand{\Fin}{\mathsf{Fin}}

\newcommand{\Csp}{\mathsf{Csp}}

\newcommand{\one}{\mathsf{1}}

\newcommand{\bicat}{\mathbf}
\newcommand{\Dbl}{\bicat{Dbl}}
\newcommand{\bCsp}{\bicat{Csp}}
\newcommand{\bA}{\bicat{A}}
\newcommand{\bB}{\bicat{B}}

\newcommand{\bD}{\bicat{D}}
\newcommand{\Cat}{\bicat{Cat}}
\newcommand{\MonCat}{\bicat{MonCat}}
\newcommand{\Rex}{\bicat{Rex}}
\newcommand{\SMC}{\bicat{SymMonCat}}
\newcommand{\OpFib}{\bicat{OpFib}}

\newcommand{\double}[1]{\mathbf{\mathbb #1}}
\newcommand{\lCsp}{\double{Csp}}

\newcommand{\Fr}{\double{Fr}}

\newcommand{\lD}{\double{D}}
\newcommand{\lE}{\double{E}}
\newcommand{\lF}{\double{F}}
\newcommand{\lG}{\double{G}}
\newcommand{\lH}{\double{H}}

\newcommand{\define}[1]{{\rm \textbf{#1}}}

\tikzset{tick/.style={postaction={decorate,decoration={markings,
mark=at position 0.4 with {\draw[-] (0,.4ex) -- (0,-.4ex);}}}}}

\newcommand{\inta}{\raisebox{.3\depth}{$\smallint\hspace{-.01in}$}}
\newcommand{\pse}{\mathrm{ps}}

\newcommand{\ot}{\otimes}

\usepackage{adjustbox}
\newenvironment{adju}[1][0.925]{%
\begin{center}\begin{adjustbox}{max height=0.5\textheight, max width=#1\textwidth}}{\end{adjustbox}\end{center}}

\title{Structured Versus Decorated Cospans}

\author{John C. Baez}
\affiliation{Department of Mathematics, University of California, Riverside CA, USA 92521}
\affiliation{Centre for Quantum Technologies, National University of Singapore, Singapore 117543}
\email{baez@math.ucr.edu}

\author{Kenny Courser}
\affiliation{Department of Mathematics, University of California, Riverside CA, USA 92521}
\email{kcour001@ucr.edu}

\author{Christina Vasilakopoulou}
\affiliation{Department of Mathematics, University of Patras, Greece 265 04}
\email{cvasilak@math.upatras.gr}

\begin{document}

\begin{abstract}
One goal of applied category theory is to understand open systems.  We compare two ways of describing open systems as cospans equipped with extra data.    First, given a functor $L \maps \A \to \X$, a `structured cospan' is a diagram in $\X$ of the form $L(a) \rightarrow x \leftarrow L(b)$.  We give a new proof that if $\A$ and $\X$ have finite colimits and $L$ preserves them, there is a symmetric monoidal double category whose objects are those of $\A$ and whose horizontal 1-cells are structured cospans. Second, given a pseudofunctor $F \maps \A \to \Cat$, a `decorated cospan' is a diagram in $\A$ of the form $a \rightarrow m \leftarrow b$ together with an object of $F(m)$. Generalizing the work of Fong, we show that if $\A$ has finite colimits and $F \maps (\A,+) \to (\Cat,\times)$ is symmetric lax monoidal, there is a symmetric monoidal double category whose objects are those of $\A$ and whose horizontal 1-cells are decorated cospans.  We prove that under certain conditions, these two constructions become isomorphic when we take $\X = \inta F$ to be the Grothendieck category of $F$.  We illustrate these ideas with applications to electrical circuits, Petri nets, dynamical systems and epidemiological modeling.
\end{abstract}

\maketitle

\setcounter{tocdepth}{1}
\tableofcontents

\section{Introduction}

An `open system' is any sort of system that can interact with the outside world.  Experience has shown that open systems are nicely modeled using cospans \cite{CourserThesis, FongThesis, PollardThesis}. A cospan in some category $\A$ is a diagram of this form:
\[
\begin{tikzpicture}[scale=1.5]
\node (A) at (0,0) {$a$};
\node (B) at (1,1) {$m$};
\node (C) at (2,0) {$b$};
\path[->,font=\scriptsize,>=angle 90]
(A) edge node[above]{$i$} (B)
(C) edge node[above]{$o$} (B);
\end{tikzpicture}
\]
We call $m$ the \define{apex}, $a$ and $b$ the \define{feet}, and $i$ and $o$ the \define{legs} of the cospan.   The apex describes the system itself.  The feet describe `interfaces'  through which the system can interact with the outside world.  The legs describe how the interfaces are included in the system.   If the category $\A$ has finite colimits, we can compose cospans using pushouts: this describes the operation of attaching two open systems together in series by identifying one interface of the first with one of the second.  We can also `tensor' cospans using coproducts: this describes setting open systems side by side, in parallel.  Via these operations we obtain a symmetric monoidal double category
with cospans in $\A$ as its horizontal 1-cells \cite{Courser,Niefield}.

However, we often want the system itself to have more structure than its interfaces.   This led Fong to develop a theory of `decorated' cospans \cite{Fong}.  Given a category $\A$ with finite colimits, a symmetric lax monoidal functor $F \maps (\A,+) \to (\textsf{Set},\times)$ can be used to equip the apex $m$ of a cospan in $\A$ with some extra data: an element $s \in F(m)$, which we call a \textbf{decoration}.  Thus a \define{decorated cospan} is a pair:
\[
\begin{tikzpicture}[scale=1.5]
\node (A) at (0,0) {$a$};
\node (B) at (1,0) {$m$};
\node (C) at (2,0) {$b$,};
\node (E) at (4,0) {$s \in F(m)$.};
\path[->,font=\scriptsize,>=angle 90]
(A) edge node[above]{$i$} (B)
(C) edge node[above]{$o$} (B);
\end{tikzpicture}
\]
Fong proved that there is a symmetric monoidal category with objects
of $\A$ as its objects and equivalence classes of decorated cospans as its morphisms.  Such categories were used to describe a variety of open systems: electrical circuits, Markov processes, chemical reaction networks and dynamical systems \cite{BF,BFP,BP}.

Unfortunately, many applications of decorated cospans were flawed.  The problem is that while Fong's decorated cospans are good for decorating the apex $m$ with an element of a set $F(m)$, they are unable to decorate it with an object of a category.   An example would be equipping a finite set $m$ with edges making its elements into the nodes of a graph.    We would like the following `open graph' to be a decorated cospan where the apex is the finite set $m = \{n_1, n_2, n_3, n_4\}$:
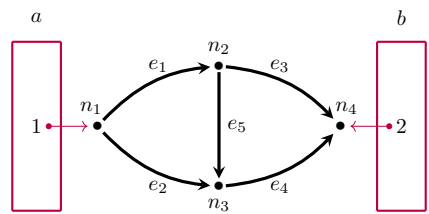
\[
\scalebox{0.8}{
\begin{tikzpicture}
	\begin{pgfonlayer}{nodelayer}
		\node [contact] (n1) at (-2,0) {$\bullet$};
		\node [style = none] at (-2.1,0.3) {$n_1$};
		\node [contact] (n2) at (0,1) {$\bullet$};
		\node [style = none] at (0,1.3) {$n_2$};
		\node [contact] (n3) at (0,-1) {$\bullet$};
		\node [style = none] at (0,-1.3) {$n_3$};
		\node [contact] (n4) at (2,0) {$\bullet$};
		\node [style = none] at (2.1,0.3) {$n_4$};

		\node [style = none] at (-1,1) {$e_1$};
		\node [style = none] at (-1,-1) {$e_2$};
		\node [style = none] at (1,1) {$e_3$};
		\node [style = none] at (1,-1) {$e_4$};
	    \node [style = none] at (0.3,0) {$e_5$};

		\node [style=none] (1) at (-3,0) {1};
		\node [style=none] (4) at (3,0) {2};

		\node [style=none] (ATL) at (-3.4,1.4) {};
		\node [style=none] (ATR) at (-2.6,1.4) {};
		\node [style=none] (ABR) at (-2.6,-1.4) {};
		\node [style=none] (ABL) at (-3.4,-1.4) {};

		\node [style=none] (X) at (-3,1.8) {$a$};
		\node [style=inputdot] (inI) at (-2.8,0) {};

		\node [style=none] (Z) at (3,1.8) {$b$};
	 \node [style=inputdot] (outI') at (2.8,0) {};

		\node [style=none] (MTL) at (2.6,1.4) {};
		\node [style=none] (MTR) at (3.4,1.4) {};
		\node [style=none] (MBR) at (3.4,-1.4) {};
		\node [style=none] (MBL) at (2.6,-1.4) {};

	\end{pgfonlayer}
	\begin{pgfonlayer}{edgelayer}
		\draw [style=inarrow, bend left=20, looseness=1.00] (n1) to (n2);
		\draw [style=inarrow, bend right=20, looseness=1.00] (n1) to (n3);
		\draw [style=inarrow, bend left=20, looseness=1.00] (n2) to (n4);
		\draw [style=inarrow, bend right=20, looseness=1.00] (n3) to (n4);
		\draw [style=inarrow] (n2) to (n3);
		\draw [style=simple] (ATL.center) to (ATR.center);
		\draw [style=simple] (ATR.center) to (ABR.center);
		\draw [style=simple] (ABR.center) to (ABL.center);
		\draw [style=simple] (ABL.center) to (ATL.center);
		\draw [style=simple] (MTL.center) to (MTR.center);
		\draw [style=simple] (MTR.center) to (MBR.center);
		\draw [style=simple] (MBR.center) to (MBL.center);
		\draw [style=simple] (MBL.center) to (MTL.center);
		\draw [style=inputarrow] (inI) to (n1);
		\draw [style=inputarrow] (outI') to (n4);
	\end{pgfonlayer}
\end{tikzpicture}
}
\]
We might hope to do this using a symmetric lax monoidal functor $F \maps (\Fin\Set, +) \to (\Set, \times)$ assigning to each finite set $m$ the set of all graphs with $m$ as their set of nodes. But this hope is doomed, for reasons painstakingly explained in \cite[Section 5]{BC}.   The key problem is that two graphs with $m$ as their set of nodes but different sets of edges give distinct elements of $F(m)$, and this prevents $F$ from being symmetric lax monoidal.   To solve this problem, we need to bring isomorphisms of graphs into the framework---so we need $F(m)$ to be a \emph{category} of graphs with $m$ as their set of nodes.

Here we implement this solution.  Instead of basing the theory of decorated cospans on a symmetric lax monoidal functor $F \maps (\A, +) \to (\Set, \times)$, we use a symmetric lax monoidal pseudofunctor $F \maps (\A, +) \to (\Cat, \times)$.  In \cref{thm:decorated_cospans_1,thm:decorated_cospans_2}, we use this
data to construct a symmetric monoidal double category $F\lCsp$ in which:
\begin{itemize}
\item an object is an object of $\A$,
\item a vertical 1-morphism is a morphism of $\A$,
\item a horizontal 1-cell from $a$ to $b$ is a decorated cospan:
\[
\begin{tikzpicture}[scale=1.5]
\node (A) at (0,0) {$a$};
\node (B) at (1,0) {$m$};
\node (C) at (2,0) {$b,$};
\node (D) at (3,0) {$s \in F(m)$,};
\path[->,font=\scriptsize,>=angle 90]
(A) edge node[above]{$i$} (B)
(C) edge node[above]{$o$} (B);
\end{tikzpicture}
\]
\item a 2-morphism is a \define{map of decorated cospans}: that is, a commutative
diagram
\[
\begin{tikzpicture}[scale=1.5]
\node (A) at (0,0.5) {$a$};
\node (A') at (0,-0.5) {$a'$};
\node (B) at (1,0.5) {$m$};
\node (C) at (2,0.5) {$b$};
\node (C') at (2,-0.5) {$b'$};
\node (D) at (1,-0.5) {$m'$};
\node (E) at (3,0.5) {$s \in F(m)$};
\node (F) at (3,-0.5) {$s' \in F(m')$};
\path[->,font=\scriptsize,>=angle 90]
(A) edge node[above]{$i$} (B)
(C) edge node[above]{$o$} (B)
(A) edge node[left]{$f$} (A')
(C) edge node[right]{$g$} (C')
(A') edge node[above] {$i'$} (D)
(C') edge node[above] {$o'$} (D)
(B) edge node [left] {$h$} (D);
\end{tikzpicture}
\]
together with a morphism $\tau \maps F(h)(s) \to s'$ in $F(m')$.
\end{itemize}

In fact another solution to the problem is already known: the theory of structured cospans \cite{BC,CourserThesis}.  Given a functor $L \maps \A \to \X$, a \define{structured cospan} is a cospan in $\X$ whose feet come from a pair of objects in $\A$:
\[
\begin{tikzpicture}[scale=1.2]
\node (A) at (0,0) {$L(a)$};
\node (B) at (1,1) {$x$};
\node (C) at (2,0) {$L(b).$};
\path[->,font=\scriptsize,>=angle 90]
(A) edge node[above]{$$} (B)
(C)edge node[above]{$$}(B);
\end{tikzpicture}
\]
This is another way of letting the apex have more structure than the feet.   When $\A$ and $\X$ have finite colimits and $L$ preserves them, there is a symmetric monoidal double category ${}_L \lCsp(\X)$ where:
\begin{itemize}
\item an object is an object of $\A$,
\item a vertical 1-morphism is a morphism of $\A$,
\item a horizontal 1-cell from $a$ to $b$ is a diagram in $\X$ of this form:
\[
\begin{tikzpicture}[scale=1.5]
\node (A) at (0,0) {$L(a)$};
\node (B) at (1,0) {$x$};
\node (C) at (2,0) {$L(b)$};
\path[->,font=\scriptsize,>=angle 90]
(A) edge node[above]{$i$} (B)
(C)edge node[above]{$o$}(B);
\end{tikzpicture}
\]
\item a 2-morphism is a commutative diagram in $\X$ of this form:
\[
\begin{tikzpicture}[scale=1.5]
\node (E) at (3,0) {$L(a)$};
\node (F) at (5,0) {$L(b)$};
\node (G) at (4,0) {$x$};
\node (E') at (3,-1) {$L(a')$};
\node (F') at (5,-1) {$L(b')$};
\node (G') at (4,-1) {$x'$};
\path[->,font=\scriptsize,>=angle 90]
(F) edge node[above]{$o$} (G)
(E) edge node[left]{$L(f)$} (E')
(F) edge node[right]{$L(g)$} (F')
(G) edge node[left]{$\alpha$} (G')
(E) edge node[above]{$i$} (G)
(E') edge node[below]{$i'$} (G')
(F') edge node[below]{$o'$} (G');
\end{tikzpicture}
\]
\end{itemize}

Many of the flawed applications of decorated cospans have been fixed using structured cospans \cite[Section 6]{BC}, but not every decorated cospan double category is equivalent to a structured cospan double category.   Here we give sufficient conditions for a decorated cospan double category to be equivalent---and in fact, isomorphic---to a structured cospan double category.

Suppose $\A$ has finite colimits and $F \maps (\A , +) \to (\Cat, \times)$ is a symmetric lax monoidal pseudofunctor.  Then each category $F(a)$ for $a \in A$ becomes symmetric monoidal, and $F$ becomes a pseudofunctor $F \maps \A \to \SMC$.    Using the Grothendieck construction, $F$ also gives an opfibration $U \maps \X \to \A$ where $\X = \inta F$.   Let $\Rex$ be the 2-category of categories with finite colimits, functors preserving finite colimits, and natural transformations.  We show that if $F \maps \A \to \SMC$ factors through $\Rex$ as a pseudofunctor, the opfibration $U \maps \X \to \A$ is also a right adjoint.   From the accompanying left adjoint $L \maps \A \to \X$, we  construct a symmetric monoidal double category ${}_L \lCsp(\X)$ of structured cospans.  In \cref{thm:equiv} we prove that this structured cospan double category ${}_L \lCsp(\X)$ is isomorphic to the decorated cospan double category $F \lCsp$.   In fact, they are isomorphic as symmetric monoidal double categories.

This result shows that under certain conditions, structured and decorated cospans provide equivalent ways of describing open systems.  We illustrate this in Section \ref{Applications} with applications to graphs, electrical circuits, Markov processes, Petri nets, Petri nets with rates, and dynamical systems.   This is meant to be a fairly thorough review of the existing literature.   It becomes clear that when either structured or decorated cospans can be used, structured cospans are simpler.  However, in some cases we need decorated cospans, for reasons we explain.

\subsection*{Outline}

In \cref{DecCospansDoubleCat} we construct the double category of decorated cospans, $F\lCsp$, and show how to construct maps between decorated cospan double categories. In \cref{Structured} we give a new construction of the double category of structured cospans, ${}_L \lCsp(\X)$.  In \cref{EquivDoubleCats} we prove that the double categories of decorated cospans and structured cospans are isomorphic under suitable conditions.  In \cref{spinoffs} we establish the isomorphism between structured and decorated cospans at the level of bicategories and categories (via decategorification).  In \cref{Applications} we describe applications.

\subsection*{Conventions}

In this paper, we use a sans-serif font like $\C$ for categories, boldface like $\mathbf{B}$ for bicategories or 2-categories, and blackboard bold like $\lD$ for double categories. For double categories with names having more than one letter, like $\lCsp(\X)$, only the first letter is in
blackboard bold. In this paper, `double category' means `pseudo double category', as in \cref{defn:double_category}. A double category $\lD$ has a category of objects and a category of arrows, and we call these $\lD_0$ and $\lD_1$ despite the fact that they are categories. Vertical composition in our double categories is strictly associative, while horizontal composition need not be.  We use $(\C,\otimes)$ to stand for a monoidal or perhaps symmetric monoidal category with $\otimes$ as its tensor product.  For composition or morphisms we use concatenation or occasionally $\circ$.

\subsection*{Acknowledgements}

We thank Daniel Cicala, Brendan Fong, Sophie Libkind, Joe Moeller and Morgan Rogers for helpful conversations, and the referees for their many useful suggestions.  We especially thank Michael Shulman for suggesting a proof strategy that greatly improved this paper.   The third author would like to thank the General Secretariat for Research  and Technology (GSRT) and the Hellenic Foundation for Research and Innovation (HFRI).

\section{Decorated cospans}\label{DecCospansDoubleCat}

In this section we build symmetric monoidal double categories of decorated cospans, and then study the functoriality of this construction.
\cref{thm:decorated_cospans_1} explains how to construct a double category of decorated cospans from a lax monoidal pseudofunctor $F \maps (\A,+) \to
(\Cat, \times)$ whenever $\A$ has finite colimits.  \cref{thm:decorated_cospans_2} gives conditions under which this double category is symmetric
monoidal.   These results build on earlier work of Fong \cite{Fong}, and we recommend his thesis for intuitive explanations of the fundamental ideas
\cite{FongThesis}. For concrete examples of the structures that follow, we refer the reader to \cref{Applications}.

In all that follows, when we say a category `has finite colimits' we mean it is equipped with a choice of colimit for every finite diagram.   Thus, if
$\A$ has finite colimits it gives a cocartesian monoidal category $(\A,+)$: that is, a symmetric monoidal category where the monoidal structure is
given by the chosen binary coproducts and initial object.   However, when we say a functor `preserves finite colimits', it need only do this up to
canonical isomorphism, unless otherwise specified.   We recall the concept of lax monoidal pseudofunctor in \cref{subsec:bicats}.

\begin{thm}\label{thm:decorated_cospans_1}
Let $\A$ be a category with finite colimits and $(F,\phi,\phi_0) \maps (\A,+) \to (\Cat,\times)$ a lax monoidal pseudofunctor. Then there exists a double category $F\lCsp$ in which
\begin{itemize}
\item an object is an object of $\A$,
\item a vertical 1-morphism is a morphism of $\A$,
\item a horizontal 1-cell is an $F$\define{-decorated cospan}, that is,
a diagram in $\A$ of the form
\[
\begin{tikzpicture}[scale=1.5]
\node (A) at (0,0) {$a$};
\node (B) at (1,0) {$m$};
\node (C) at (2,0) {$b,$};
\path[->,font=\scriptsize,>=angle 90]
(A) edge node[above]{$i$} (B)
(C) edge node[above]{$o$} (B);
\end{tikzpicture}
\]
together with a \define{decoration} $s \in F(m)$,
\item a 2-morphism is a \define{map of} $F$\define{-decorated cospans}, that is,
a commutative diagram in $\A$ of the form
\[
\begin{tikzpicture}[scale=1.5,baseline=(current bounding box.center)]
\node (A) at (0,0.5) {$a$};
\node (A') at (0,-0.5) {$a'$};
\node (B) at (1,0.5) {$m$};
\node (C) at (2,0.5) {$b$};
\node (C') at (2,-0.5) {$b'$};
\node (D) at (1,-0.5) {$m'$};
\node (E) at (3,0.5) {$s \in F(m)$};
\node (F) at (3,-0.5) {$s' \in F(m')$};
\path[->,font=\scriptsize,>=angle 90]
(A) edge node[above]{$i$} (B)
(C) edge node[above]{$o$} (B)
(A) edge node[left]{$f$} (A')
(C) edge node[right]{$g$} (C')
(A') edge node[below] {$i'$} (D)
(C') edge node[below] {$o'$} (D)
(B) edge node [left] {$h$} (D);
\end{tikzpicture}
\]
together with a \define{decoration morphism} $\tau \maps F(h)(s) \to s'$ in $F(m')$, which can be thought of as a natural transformation
\begin{equation}\label{eq:taunattransf}
\begin{tikzpicture}[baseline=(current  bounding  box.center),scale=1.5]
\node (G) at (-3.5,-0.5) {$1$};
\node (H) at (-2,0) {$F(m)$};
\node (I) at (-2,-1) {$F(n)$};
\node (J) at (-2.5,-0.5) {$\scriptstyle\Downarrow\tau$};
\path[->,font=\scriptsize,>=angle 90]
(G) edge node [above] {$s$} (H)
(G) edge node [below] {$s'$} (I)
(H) edge node [right] {$F(h)$} (I);
\end{tikzpicture}
\end{equation}
\end{itemize}
\end{thm}

Note that the decoration $s \in F(m)$ is now an object in the category $F(m)$, not an element of a set as in Fong's original approach.  Vertical composition in $F\lCsp$ is done using composition in $\A$.  The horizontal composite of $F$-decorated cospans
\[ \left(a\to m\leftarrow b,s\in F(m)\right), \qquad \left(b\to n\leftarrow c,t\in F(n)\right)\]
is the usual composite via pushout of their underlying cospans, shown in dashed arrows here:
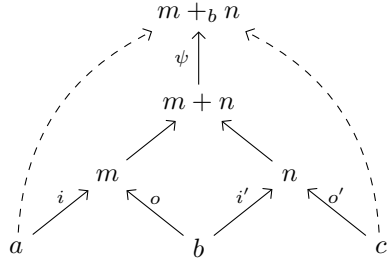
\[
\begin{tikzpicture}[scale=1.2,baseline=(current bounding box.center)]
\node (A) at (0,0) {$a$};
\node (B) at (1,0.8) {$m$};
\node (C) at (2,0) {$b$};
\node (D) at (3,0.8) {$n$};
\node (E) at (4,0) {$c$};
\node (F) at (2,1.6) {$m+n$};
\node (G) at (2,2.6) {$m+_{b} n$};
\path[->,font=\scriptsize,>=angle 90]
(A) edge node[above]{$i$} (B)
(C) edge node[above]{$o$} (B)
(C) edge node [above] {$i'$} (D)
(E) edge node [above] {$o'$} (D)
(B) edge node [below] {$$} (F)
(D) edge node [below] {$$} (F)
(F) edge node [left] {$\psi$} (G)
(A) edge[dashed,bend left] node [left] {} (G)
(E) edge[dashed,bend right] node [right] {} (G);
\end{tikzpicture}
\]
together with the decoration $t \odot s \in F(m +_b n)$ specified by this functor:
\[  \one\cong\one \times \one \xrightarrow{s \, \times \, t} F(m) \times F(n) \xrightarrow{\phi_{m,n}} F(m+n) \xrightarrow{F(\psi)} F(m +_{b}n). \]
Given two horizontally composable maps of $F$-decorated cospans $\alpha$ and $\beta$:
\[
\begin{tikzpicture}[scale=1.5]
\node (A) at (0,0.5) {$a$};
\node (A') at (0,-0.5) {$a'$};
\node (B) at (1,0.5) {$m$};
\node (C) at (2,0.5) {$b$};
\node (C') at (2,-0.5) {$b'$};
\node (D) at (1,-0.5) {$m'$};
\node (E) at (3,0.5) {$s \in F(m)$};
\node (F) at (3,-0.5) {$s' \in F(m')$};
\node (G) at (4,0.5) {$b$};
\node (H) at (5,0.5) {$n$};
\node (I) at (6,0.5) {$c$};
\node (G') at (4,-0.5) {$b'$};
\node (H') at (5,-0.5) {$n'$};
\node (I') at (6,-0.5) {$c'$};
\node (J) at (7,0.5) {$t \in F(n)$};
\node (K) at (7,-0.5) {$t' \in F(n')$};
\node (L) at (1,-1) {$\tau_\alpha \maps F(h_1)(s) \to s'$};
\node (M) at (5,-1) {$\tau_\beta \maps F(h_2)(t) \to t'$};
\path[->,font=\scriptsize,>=angle 90]
(A) edge node[above]{$i_1$} (B)
(C) edge node[above]{$o_1$} (B)
(A) edge node[left]{$f$} (A')
(C) edge node[right]{$g$} (C')
(A') edge node[above] {$i_1'$} (D)
(C') edge node[above] {$o_1'$} (D)
(B) edge node [left] {$h_1$} (D)
(G) edge node [above] {$i_2$} (H)
(G) edge node [left] {$g$} (G')
(H) edge node [left] {$h_2$} (H')
(G') edge node [above] {$i_2'$} (H')
(I) edge node [above] {$o_2$} (H)
(I) edge node [right] {$k$} (I')
(I') edge node [above] {$o_2'$} (H');
\end{tikzpicture}
\]
their composite $\beta\odot \alpha$ is the horizontal composite of the two maps of cospans in $\A$:
\[
\begin{tikzpicture}[scale=1.5,baseline=(current bounding box.center)]
\node (A) at (0,0.5) {$a$};
\node (A') at (0,-0.5) {$a'$};
\node (B) at (1.5,0.5) {$m+_{b} n$};
\node (C) at (3,0.5) {$c$};
\node (C') at (3,-0.5) {$c'$};
\node (D) at (1.5,-0.5) {$m'+_{b'}n'$};
\node (E) at (4.5,0.5) {$t \odot s \in F(m+_b n)$};
\node (F) at (4.5,-0.5) {$t' \odot s' \in F(m' +_{b'} n')$};
\path[->,font=\scriptsize,>=angle 90]
(A) edge node[above]{} (B)
(C) edge node[above]{} (B)
(A) edge node[left]{$f$} (A')
(C) edge node[right]{$k$} (C')
(A') edge node [above]{} (D)
(C') edge node [above]{} (D)
(B) edge node [left] {$h_1 +_g h_2$} (D);
\end{tikzpicture}
\]
together with the decoration morphism $\tau_{\beta \odot \alpha} \maps F(h_1 +_g h_2)(t \odot s) \to (t' \odot s')$ specified by this natural transformation:
\[
\begin{tikzpicture}[scale=1.7,baseline=(current bounding box.center)]
\node (A) at (4.75,0) {$\scriptstyle\Downarrow \tau_\alpha \times \tau_\beta$};
\node (D) at (3,0) {$\one \cong \one \times \one$};
\node (E) at (5.5,0.5) {$F(m) \times F(n)$};
\node () at (6.5,0) {$\stackrel{\phi_{h_1,h_2}}{\cong}$};
\node () at (8.5,0) {$\cong$};
\node (E') at (5.5,-0.5) {$F(m') \times F(n')$};
\node (B) at (7.5,0.5) {$F(m+n)$};
\node (B') at (7.5,-0.5) {$F(m' + n')$};
\node (C) at (9.5,0.5) {$F(m+_{b} n)$};
\node (C') at (9.5,-0.5) {$F(m' +_{b'} n')$};
\path[->,font=\scriptsize,>=angle 90]
(E) edge node [above] {$\phi_{m,n}$} (B)
(E') edge node [above] {$\phi_{m',n'}$} (B')
(B) edge node [above] {$F(\psi)$} (C)
(B') edge node [above] {$F(\psi)$} (C')
(C) edge node  [fill=white] {$F(h_1 +_g h_2)$} (C')
(B) edge node [fill=white] {$F(h_1 + h_2)$} (B')
(D) edge node [above] {$s \times t$} (E)
(D) edge node [below] {$s' \times t'$} (E')
(E) edge node [fill=white] {$F(h_1) \times F(h_2)$} (E');
\end{tikzpicture}
\]
where the middle isomorphism is \cref{eq:pseudonaturality} from the pseudonaturality of $\phi$ and the right-hand isomorphism comes from the pseudofunctoriality of $F$.

\begin{thm}\label{thm:decorated_cospans_2}
Let $\A$ be a category with finite colimits and let $(F,\phi,\phi_0)\maps (\A,+) \to (\Cat,\times)$ be a symmetric lax monoidal pseudofunctor. Then the double category $F\lCsp$ of \cref{thm:decorated_cospans_1} is symmetric monoidal, where the tensor product
\begin{itemize}
\item of two objects $a$ and $b$ is their coproduct $a+b$ in $\A$,
\item of two vertical 1-morphisms $f \maps a \to b$ and $f' \maps a' \to b'$ is $f+f' \maps a+a' \to b+b'$ in $\A$,
\item of two horizontal 1-cells $(a \xrightarrow{i_1} m \xleftarrow{o_1} b, s \in F(m))$ and $(c \xrightarrow{i_2} n \xleftarrow{o_2} d, t \in F(n))$ is
\[
\begin{tikzcd}
& m+n & \\
a+c\ar[ur,"i_1+i_2"] && b+d,\ar[ul,"o_1+o_2"']
\end{tikzcd}\qquad s \otimes t:=  \phi_{m,n}(s,t) \in F(m + n)
\]
\item of two 2-morphisms $\alpha$ and $\beta$ is:
\begin{adju}
\begin{tikzpicture}[scale=1.2]
\node (A) at (0,0.5) {$a$};
\node (A') at (0,-0.5) {$a'$};
\node (B) at (1,0.5) {$m$};
\node (C) at (2,0.5) {$b$};
\node (C') at (2,-0.5) {$b'$};
\node (D) at (1,-0.5) {$m'$};
\node () at (2.7,0) {$\otimes$};
\node () at (6.2,0) {$=$};
\node (G) at (3.5,0.5) {$c$};
\node (H) at (4.5,0.5) {$n$};
\node (I) at (5.5,0.5) {$d$};
\node (G') at (3.5,-0.5) {$c'$};
\node (H') at (4.5,-0.5) {$n'$};
\node (I') at (5.5,-0.5) {$d'$};
\node (L) at (1,-1.2) {$\scriptstyle\tau_\alpha \maps F(h)(s) \to s' \textrm{ in } F(m')$};
\node (M) at (4.5,-1.2) {$\scriptstyle\tau_\beta \maps F(h')(t) \to t'\textrm{ in } F(n')$};
\path[->,font=\scriptsize,>=angle 90]
(A) edge node[above]{$i_1$} (B)
(C) edge node[above]{$o_1$} (B)
(A) edge node[left]{$f$} (A')
(C) edge node[right]{$g$} (C')
(A') edge node[below] {$i_1'$} (D)
(C') edge node[below] {$o_1'$} (D)
(B) edge node [left] {$h$} (D)
(G) edge node [above] {$i_2$} (H)
(G) edge node [left] {$f'$} (G')
(H) edge node [left] {$h'$} (H')
(G') edge node [below] {$i_2'$} (H')
(I) edge node [above] {$o_2$} (H)
(I) edge node [right] {$g'$} (I')
(I') edge node [below] {$o_2'$} (H');
\end{tikzpicture}
\begin{tikzpicture}[scale=1.1]
\node (A) at (0,0.5) {$a+c$};
\node (A') at (0,-0.5) {$a'+c'$};
\node (B) at (2,0.5) {$m+n$};
\node (C) at (4,0.5) {$b+d$};
\node (C') at (4,-0.5) {$b'+d'$};
\node (D) at (2,-0.5) {$m'+n'$};
\node () at (2,-1.3) {$\scriptstyle\tau_{\alpha\ot\beta}\maps F(h+h')(\phi_{m,n}(s,t))\to \phi_{m',n'}(s',t')\textrm{ in }F(m'+n')$};
\path[->,font=\scriptsize,>=angle 90]
(A) edge node[above]{$i_1+i_2$} (B)
(C) edge node[above]{$o_1+o_2$} (B)
(A) edge node[left]{$f+f'$} (A')
(C) edge node[right]{$g+g'$} (C')
(A') edge node [below]{$i_1'+i_2'$} (D)
(C') edge node [below]{$o_1'+o_2'$} (D)
(B) edge node [left] {$h+h'$} (D);
\end{tikzpicture}
\end{adju}
with decoration morphism $\tau_{\alpha\otimes\beta}$ given by the following diagram:
\begin{displaymath}
 \begin{tikzcd}[column sep=.5in, row sep=.1in]
& F(m)\times F(n)\ar[r,"\phi_{m,n}"]\ar[dd,"F(h)\times F(h')"description]\ar[ddr,phantom,"{\stackrel[\cref{eq:pseudonaturality}]{\phi_{h,h'}}{\cong}}"] & F(m+n)\ar[dd,"F(h+h')"] \\
\one \ar[ur,"s\,\times \, t"]\ar[dr,"s' \,\times \, t' "']\ar[r,phantom,"\scriptstyle\Downarrow\tau_\alpha\times\tau_\beta"] & \phantom{A} && \\
& F(m')\times F(n')\ar[r,"\phi_{m',n'}"'] & F(m'+n')
 \end{tikzcd}
\end{displaymath}
\end{itemize}
\end{thm}

We prove both these theorems using the work of Shulman \cite{Shulman2008}, who gives a general way to construct a double category from a `Beck--Chevalley monoidal bifibration'---a concept recalled in \cref{subsec:fibrations}.

\begin{lem}[\textbf{Shulman}] \label{lem:shulman}
Suppose $(\A,+)$ is cocartesian monoidal and $\Phi \maps (\C,\otimes) \to (\A,+)$ is a Beck--Chevalley monoidal bifibration.  Then there is a double category $\Fr(\Phi)$ in which

\begin{itemize}
\item an object is an object of $\A$,
\item a vertical 1-morphism is a morphism of $\A$,
\item a horizontal 1-cell is a pair of objects $a,b \in \A$ together with an
object $c \in \C$ with $\Phi(c) = a + b$,
\item a 2-morphism is a pair of morphisms $f \maps a \to a'$, $g \maps b \to b'$
in $\A$ together with a morphism $h \maps c \to c'$ in $\C$ with $\Phi(h) = f + g$.
\end{itemize}
If $\Phi$ is a Beck--Chevalley symmetric monoidal bifibration then $\Fr(\Phi)$ becomes a symmetric
monoidal double category.
\end{lem}

\begin{proof} This is \cite[Theorem 14.9]{Shulman2008}; Shulman proves a dual version in more detail in his Theorem 14.4.
\end{proof}

Note that when $\C$ is the arrow category of $\A$ and $\Phi$ maps any arrow in $\C$ to its
domain, a horizontal 1-cell in $\Fr(\Phi)$ simply amounts to a cospan in $\A$. In this case $\Fr(\Phi)$ is the double category of cospans in $\A$.    We shall use variations on this idea to construct double categories of decorated and structured cospans, and to prove that under certain conditions these double categories are equivalent.    We begin with decorated cospans, proving Theorem \ref{thm:decorated_cospans_1} and \ref{thm:decorated_cospans_2} by applying Shulman's result to a particular monoidal bifibration built from a lax monoidal pseudofunctor $F \maps (\A,+) \to (\Cat,\times)$.

The Grothendieck construction gives a bijection between pseudofunctors $F \maps \A \to \Cat$ and opfibrations $U \maps \inta F \to \A$.  We need two refinements of this construction: one that gives monoidal opfibrations, and one that gives symmetric monoidal opfibrations.  We recall pseudofunctors in \cref{subsec:bicats}, and opfibrations in \cref{subsec:fibrations}; here we simply state the result we need.

\begin{lem}
\label{lem:MonGroth}
For any monoidal category $(\A,\otimes)$, there is a 2-equivalence between the 2-categories of monoidal opfibrations $U\maps (\X, \otimes) \to (\A,
\otimes)$ and lax monoidal pseudofunctors $F \maps (\A,\otimes) \to (\mathsf{Cat}, \times)$, and if $\A$ cocartesian monoidal, there is a
2-equivalence between these and pseudofunctors from $\A$ into $\MonCat$.  If $\A$ is symmetric monoidal, there is also a 2-equivalence between the
2-categories of symmetric monoidal opfibrations $U \maps (\X, \otimes) \to (\A, \otimes)$ and symmetric lax monoidal pseudofunctors $F \maps
(\A,\otimes) \to (\mathsf{Cat}, \times)$ , and if $\A$ cocartesian monoidal, there is a 2-equivalence between these and pseudofunctors from $\A$ into
$\SMC$.
\end{lem}

\begin{proof}
This was shown by Moeller and the third author \cite[Theorems~3.13 \& 4.2]{MV}. In summary, for a cocartesian base $\A$ we have correspondences
\begin{gather}
\textrm{lax monoidal pseudofunctors }F\maps(\A,+)\to(\Cat,\times) \notag\\
\Updownarrow \notag\\
\textrm{monoidal opfibrations }U\maps(\X,\otimes_\X)\to(\A,+) \notag\\
\Updownarrow \notag\\
\textrm{pseudofunctors } F \maps \A\to \MonCat \notag
\end{gather}
The second equivalence was observed earlier by Shulman \cite{Shulman2008}. Moreover, symmetric lax monoidal pseudofunctors correspond to symmetric monoidal opfibrations, and those to pseudofunctors into $\SMC$.

In more detail, if $(\phi,\phi_0)$ is the lax monoidal structure of the pseudofunctor $F$ as recalled in \cref{subsec:bicats}, the induced monoidal structure on the Grothendieck category $\X=\inta F$ (\cref{def:GrothCat}) is given by
\begin{equation}\label{eq:explicitstructure2}
\Big(a,s\in F(a)\Big)\ot_{\X}\Big(b,t\in F(b)\Big)=\Big(a+b,\phi_{a,b}(s,t)\in F(a+b)\Big), \qquad I_{\X}=\Big(0_A,\phi_0\Big)
\end{equation}
If $F$ is a symmetric lax monoidal pseudofunctor, the induced monoidal structure in $\inta F$ is symmetric via
$$\left(\beta_{a,b},(u_{a,b})_{s,t}\right)\maps \left(a+b,\phi_{a,b}(s,t)\right)\simrightarrow\left(b+a,\phi_{b,a}(t,s)\right)$$ where $\beta$ is the canonical symmetry for $\A$ and $u$ is the natural isomorphism of \cref{eq:braidedpseudofun}.

Moreover, each fiber $\X_a=F(a)$ obtains a monoidal structure via
\begin{equation} \label{eq:explicitstructure1}
\otimes_a\maps F(a)\times F(a)\xrightarrow{\phi_{a,a}}F(a+a)\xrightarrow{F(\nabla)}F(a),\quad
I_a\maps\one\xrightarrow{\phi_0}F(0)\xrightarrow{F(!_a)}F(a)
\end{equation}
where $\nabla$ is the fold map, which is symmetric when $F$ is, again via the components of $u_{a,a}$. Also, each reindexing functor $f_!=F(f)$ obtains a strong monoidal structure with
\begin{displaymath}
\begin{tikzcd}[row sep=.35in]
 F(a)\times F(a)\ar[r,"\phi_{a,a}"]\ar[d,"F(f)\times F(f)"']\ar[dr,phantom,"{\stackrel[\cref{eq:pseudonaturality}]{\phi_{f,f}}{\cong}}"] &
 F(a+a)\ar[r,"F(\nabla)"]\ar[d,"F(f+f)"description]\ar[dr,phantom,"\cong"] & F(a)\ar[d,"F(f)"] \\
 F(b)\times F(b)\ar[r,"\phi_{b,b}"'] & F(b+b)\ar[r,"F(\nabla)"'] & F(b)
\end{tikzcd}
\end{displaymath}
\end{proof}

These 2-equivalences further restrict to the case when the Grothendieck category $(X,\otimes_\X)$ is specifically cocartesian monoidal itself, with coproducts built up from \cref{eq:explicitstructure2}. In that case, opfibrations $(\X,+) \to (\A,+)$ that strictly preserve coproducts and initial object bijectively correspond to pseudofunctors into the 2-category of cocartesian categories.  For more details, see \cite[Corollary 4.7]{MV} and the related discussion.

Now we are ready to use Shulman's result to prove Theorems \ref{thm:decorated_cospans_1} and  \ref{thm:decorated_cospans_2}.  An $F$-decorated cospan from $a$ to $b$
is a cospan $a \to m \leftarrow b$ in $\A$ together with an object $s \in F(m)$.   This can be seen as a pair $(e \to m, s \in F(m))$ such that $e$ equals $a + b$.   But $(m,s \in F(m))$ is precisely an object of the Grothendieck category $\inta F$, and the opfibration $U \maps \inta F \to \A$ sends this object to $m$.  Thus, $(e \to m, s \in F(m))$ can be seen as an object of the comma category $\A/U$.   Let $\Phi \maps \A/U \to \A$ be the functor with
\[    \Phi(e \to m,s \in F(m)) = e .\]
Then an $F$-decorated cospan from $a$ to $b$ amounts to an object $c \in \A/U$ such that $\Phi(c) = a + b$.    With this choice of $\Phi$, Shulman's \ref{lem:shulman} gives the double category of $F$-decorated cospans.

\begin{proof}[Proof of \cref{thm:decorated_cospans_1}]
Let $U \maps \inta F \to \A$ be the monoidal opfibration associated to the lax monoidal pseudofunctor $F \maps (\A,+) \to (\Cat, \times)$ via \cref{lem:MonGroth}.  To obtain the desired double category we apply \cref{lem:shulman} to the functor $\Phi \maps \A/U \to \A$, which we now describe in more detail.  In the comma category $\A/U$:
\begin{itemize}
\item an object is a pair $(a \xrightarrow{i} b, s)$ consisting of a morphism in $\A$ and an
object $s \in F(b)$;
\item a morphism from $(a \xrightarrow{i} b, s)$ to $(a' \xrightarrow{i'} b', s')$ is a triple $(f,g,h)$ where $f \maps a \to a'$ and $g \maps b \to b'$ are morphisms in $\A$ such that this square commutes:
\[
\begin{tikzpicture}[scale=1.5]
\node (A) at (0,0) {$a$};
\node (B) at (1.5,0) {$a'$};
\node (C) at (0,-1) {$b$};
\node (D) at (1.5,-1) {$b'$};
\path[->,font=\scriptsize,>=angle 90]
(A) edge node[above]{$f$} (B)
(A) edge node[left]{$i$} (C)
(B) edge node[right]{$i'$} (D)
(C) edge node[below]{$g$} (D);
\end{tikzpicture}
\]
and $h \maps F(g)(s) \to s'$ is a morphism in $F(b')$.
\end{itemize}
As in any comma category we have a functor $\Phi \maps \A/U \to \A$ given on objects by $\Phi(a \xrightarrow{i} b, s)=a$ and on morphisms by $\Phi(f,g,h)=f$.

To apply \cref{lem:shulman} it suffices to prove that $\Phi$ is a Beck--Chevalley monoidal opfibration.   First, there is a monoidal structure on $\A/U$ coming from the monoidality of $U$.  On objects this is given by
\[  (a \xrightarrow{i} b, s) \otimes (a' \xrightarrow{i'} b', s') = (a+a' \xrightarrow{i+i'} b+b',
\phi_{b,b'}(s,s') )\]
where $\phi$ is the laxator for $F$.  The monoidal unit is $(0_\A \xrightarrow{!} 0_\A, 0)$
where $0 \in F(0_\A)$ is the object given by $\phi_0 \maps \one \to F(0_\A)$.   With this monoidal structure on $\A/U$, it is easy to see that $\Phi$ is a strict monoidal functor.

Next, one can check that $\Phi$ is a fibration.  Given a morphism $f \maps a \to a'$ in $\A$ and an object $(a' \xrightarrow{i'} b', s') \in \A/U$ over $a'$, one can show that a cartesian lifting of $f$ to this object is given by
\[ (f,1,1) \maps (a \xrightarrow{i' f} b' , s')\; \to \; (a' \xrightarrow{i'} b', s'). \]
Denoting the cartesian lifting of $f$ to an object $z \in \A/U$ by $\mathrm{Cart}(f,z)$, the tensor product of $\A/U$ preserves such cartesian liftings, in fact strictly, in the sense that $\mathrm{Cart}(f_1,z_1) \otimes \mathrm{Cart}(f_2,z_2)=\mathrm{Cart}(f_1 + f_2, z_1 \otimes z_2)$.

One can also check that $\Phi$ is an opfibration. Given a morphism $f \maps a \to a'$ in $\A$ and an object $(a \xrightarrow{i}b, s) \in \A/U$ over $a$, one can show that a cocartesian lifting of $f$ is given by
\[ (f,g,1) \maps (a \xrightarrow{i} b, s) \, \to \, (a' \xrightarrow{i'} b+_a a', F(g)(s)). \]
where $g, i'$ and $b +_a a'$ arise from this pushout square in $\A$:
\[
\begin{tikzpicture}[scale=1.5]
\node (A) at (0,0) {$a$};
\node (B) at (1.5,0) {$a'$};
\node (C) at (0,-1) {$b$};
\node (D) at (1.5,-1) {$b+_a a'$};
\node at (1.2,-0.8) {\Large{$\ulcorner$}};
\path[->,font=\scriptsize,>=angle 90]
(A) edge node[above]{$f$} (B)
(A) edge node[left]{$i$} (C)
(B) edge node[right]{$i'$} (D)
(C) edge node[below]{$g$} (D);
\end{tikzpicture}
\]
To show that $(f,g,1)$ is a cocartesian lifting of $f$, suppose we are given a morphism
\[    (f'',g'',h'') \maps (a \xrightarrow{i} b, s) \; \to \; (a'' \xrightarrow{i''} b'', s'')   \]
in $\A/U$ and a morphism $k \maps a' \to a''$ in $\A$ such that $\Phi(f'',g'',h'') = k f$.   We need to show there exists a unique morphism
\[     (\tilde{f}, \tilde{g}, \tilde{h}) \maps (a' \xrightarrow{i'} b +_a a', F(g)(s)) \; \to \; (a'' \xrightarrow{i''} b'', s'') \]
such that $\Phi(\tilde{f}, \tilde{g}, \tilde{h}) = k$ and
\[       (\tilde{f}, \tilde{g}, \tilde{h}) \circ (f,g,1) = (f'',g'',h'') .\]
To achieve this we choose $\tilde{f} = k$, $\tilde{h} = h''$, and define $\tilde{g}$ using
the universal property of the pushout:
\[
\begin{tikzpicture}[scale=1.5]
\node (A) at (0,0) {$a$};
\node (B) at (1.5,0) {$a'$};
\node (C) at (0,-1) {$b$};
\node (D) at (1.5,-1) {$b+_a a'$};
\node (E) at (2.5,-1.5) {$b''$};
\node at (1.2,-0.8) {\Large{$\ulcorner$}};
\path[->,font=\scriptsize,>=angle 90]
(A) edge node[above]{$f$} (B)
(A) edge node[left]{$i$} (C)
(B) edge node[right]{$i'$} (D)
(C) edge node[below]{$g$} (D)
(C) edge [bend right] node[below]{$g''$} (E)
(B) edge [bend left] node[right]{$i'' k$} (E)
(D) edge [dashed] node[below left]{$\exists! \, \tilde{g} \phantom{a}$} (E);
\end{tikzpicture}
\]
after noting that $g'' i = i' f'' = i'' k f $ by the commuting square condition obeyed by morphisms
in $\A/U$.   One can check that with these choices, $(\tilde{f},\tilde{g}, \tilde{h})$ obeys the desired conditions and is the unique morphism to do so.

It follows that $\Phi$ is a bifibration.  Denoting the cocartesian lifting of a morphism $f$ to an object $z \in \A/U$ by $\mathrm{Cocart}(f,z)$, we have that $\mathrm{Cocart}(f_1,z_1) \otimes \mathrm{Cocart}(f_2,z_2) \cong \mathrm{Cocart}(f_1 + f_2,z_1 \otimes z_2)$ by the universal map between two colimits of the same diagram.    Thus, $\Phi$ is a monoidal bifibration.

Lastly, we need to check that $\Phi$ satisfies the Beck--Chevalley condition.  Merely from the fact that $\Phi$ is a bifibration, for any commutative square in $\A$:
\[
\begin{tikzpicture}[scale=1.5]
\node (A) at (0,0) {$a$};
\node (B) at (1.5,0) {$b$};
\node (C) at (0,-1) {$c$};
\node (D) at (1.5,-1) {$d$};
\label{bc-square}
\path[->,font=\scriptsize,>=angle 90]
(A) edge node[above]{$h$} (B)
(A) edge node[left]{$k$} (C)
(B) edge node[right]{$g$} (D)
(C) edge node[below]{$f$} (D);
\end{tikzpicture}
\]
there is a natural transformation
\[
\begin{tikzpicture}[scale=1.5]
\node (A) at (0,0) {$(\A/U)_a$};
\node (B) at (1.5,0) {$(\A/U)_b$};
\node (C) at (0,-1) {$(\A/U)_c$};
\node (D) at (1.5,-1) {$(\A/U)_d$};
\node (E) at (0.75,-0.5) {$\scriptstyle\Downarrow \theta$};
\path[->,font=\scriptsize,>=angle 90]
(B) edge node[above]{$h^*$} (A)
(A) edge node[left]{$k_!$} (C)
(B) edge node[right]{$g_!$} (D)
(D) edge node[below]{$f^*$} (C);
\end{tikzpicture}
\]
defined in \eqref{theta}. We need to prove that when the square in $\A$ is
a pushout, $\theta$ is a natural isomorphism.

We start by describing $\theta$ more concretely.  Let $(b \xrightarrow{i} e, s)$ be an object of $(\A/U)_b$. Going left, we precompose with $h \maps a \to b$ to obtain the object $(a \xrightarrow{i h} e, s)$ of $(\A/U)_a$. Then going down, we push forward along $k$:
\[
\begin{tikzpicture}[scale=1.5]
\node (A) at (0,0) {$a$};
\node (B) at (1.5,0) {$e$};
\node (C) at (0,-1) {$c$};
\node (D) at (1.5,-1) {$c+_{a}e$};
\node at (1.2,-0.8) {\Large{$\ulcorner$}};
\path[->,font=\scriptsize,>=angle 90]
(A) edge node[above]{$i h$} (B)
(A) edge node[left]{$k$} (C)
(B) edge node[right]{$$} (D)
(C) edge node[below]{$\psi$} (D);
\end{tikzpicture}
\]
to obtain the object $(c \xrightarrow{\psi} c+_a e, F(\psi)(s))$ of $(\A/U)_c$.  For the other route, first going down, we push the object $(b \xrightarrow{i} e,s)$ forward along $g \maps b \to d$ by taking the following pushout:
\[
\begin{tikzpicture}[scale=1.5]
\node (A) at (0,0) {$b$};
\node (B) at (1.5,0) {$e$};
\node (C) at (0,-1) {$d$};
\node (D) at (1.5,-1) {$d+_b e$};
\node at (1.2,-0.8) {\Large{$\ulcorner$}};
\path[->,font=\scriptsize,>=angle 90]
(A) edge node[above]{$i$} (B)
(A) edge node[left]{$g$} (C)
(B) edge node[right]{$$} (D)
(C) edge node[below]{$\phi$} (D);
\end{tikzpicture}
\]
which yields the object $(d \xrightarrow{\phi} d+_b e, F(\phi)(s))$ of $(\A/U)_d$.  Then going left, we precompose with $f \maps c \to d$ to obtain the object $(c \xrightarrow{\phi f} d+_b e, F(\phi)(s))$ of $(\A/U)_c$.

The natural transformation $\theta$ gives a morphism
\[    \theta \maps   (c \xrightarrow{\psi} c+_a e, F(\psi)(s)) \; \to \; (c \xrightarrow{\phi f} d+_b e, F(\phi)(s))  \]
in $(\A/U)_c$.   Concretely, this arises from the universal property of the pushout:
\[
\begin{tikzpicture}[scale=1.5]
\node (A) at (0,0) {$b$};
\node (B) at (1.5,0) {$e$};
\node (C) at (0,-1) {$d$};
\node (D) at (1.5,-1) {$d+_{b}e$};
\node (E) at (-1.5,0) {$a$};
\node (F) at (-1.5,-1) {$c$};
\node (G) at (2.5,-2) {$c+_{a}e$};
\node at (1.2,-0.8) {\Large{$\ulcorner$}};
\path[->,font=\scriptsize,>=angle 90]
(E) edge node [left] {$k$} (F)
(E) edge node [above] {$h$} (A)
(F) edge node [below] {$f$} (C)
(A) edge node[above]{$i$} (B)
(A) edge node[left]{$g$} (C)
(B) edge node[right]{$$} (D)
(C) edge node[below]{$\phi$} (D)
(F) edge [bend right] node[below]{$\psi$} (G)
(G) edge [dashed] node[below left]{$\exists ! \;\zeta $} (D)
(B) edge [bend left] (G);
\end{tikzpicture}
\]
Namely, we have $\theta = (1, \zeta, 1)$.   But when the original square in $\A$ (at upper left
above) is a pushout, its pasting with the other pushout gives a pushout, so $\zeta$ and hence $\theta$ is an isomorphism.

This shows that $\Phi$ obeys the Beck--Chevalley condition.  Having checked all the hypotheses of \cref{lem:shulman}, we can define $F\lCsp$ to be $\Fr(\Phi)$ and conclude the proof of \cref{thm:decorated_cospans_1}.
\end{proof}

\begin{proof}[Proof of \cref{thm:decorated_cospans_2}]
When the lax monoidal pseudofunctor $F \maps (\A,+) \to (\Cat, \times)$ of \cref{thm:decorated_cospans_1} is symmetric, the monoidal bifibration $\Phi \maps \A/U \to A$
is symmetric, so \cref{lem:shulman} implies that the double category $F\lCsp = \Fr(\Phi)$ is symmetric monoidal as described.
\end{proof}

The decorated cospan formalism gives not only double categories, but also maps between these.  Suppose we have two categories $\A,\A'$ with finite colimits and two lax monoidal pseudofunctors $F \maps \A \to \Cat$ and $F' \maps \A' \to \Cat$.   Then we can obtain a map between their decorated cospan double categories, namely a double functor $\lH \maps F\lCsp \to F' \lCsp$, from:
\begin{itemize}
\item a functor $H \maps \A \to \A'$ that preserves finite colimits,
\item a lax monoidal pseudofunctor $(E,\phi,\phi_0) \maps \Cat \to \Cat$,
\item a natural transformation $\theta \maps EF \Rightarrow F'H$:
\[
\begin{tikzpicture}[scale=1.5]
\node (A) at (0,0) {$\A$};
\node (B) at (1,0) {$\Cat$};
\node (C) at (0,-1) {$\A'$};
\node (D) at (1,-1) {$\Cat$.};
\node (E) at (0.5,-0.5) {$\scriptstyle\Downarrow\theta$};
\path[->,font=\scriptsize,>=angle 90]
(A) edge node[above]{$F$} (B)
(A) edge node[left]{$H$} (C)
(B) edge node[right]{$E$} (D)
(C) edge node[below]{$F'$} (D);
\end{tikzpicture}
\]
\end{itemize}
The intuition behind this square is that while the pseudofunctors $F$ and $F'$ serve to assign categories of possible decorations to the objects of $\A$ and $\A'$, respectively, the functor $H$ lets us turn objects of $\A$ into objects of $\A'$, and the pseudofunctor $E$ lets us change decorations as prescribed by $F$ into those as prescribed by $F'$, up to the transformation $\theta$.  In applications $E$ and $H$ are often identities.

The double functor $\lH \maps F\lCsp \to F'\lCsp$ is defined as follows:
\begin{itemize}
\item The image of an object $a \in F\lCsp_0=\A$ is the object $H(a) \in F'\lCsp_0=\A'$.
\item The image of a vertical 1-morphism $f \maps a \to b$ is the vertical 1-morphism $H(f) \maps H(a) \to H(b)$.
\end{itemize}
In other words, the object component $\lH_0$ of the double functor $\lH$ is the functor $H$.
\begin{itemize}
\item The image of an $F$-decorated cospan
\[ M= (a\xrightarrow{i}m\xleftarrow{o}b, \; s \in F(m)) \]
is the following $F'$-decorated cospan:
\[ \lH(M)= (H(a)\xrightarrow{H(i)}H(m)\xleftarrow{H(o)}H(b), \;\bar{s}\in F'(H(m))) \]
 where
\[
\bar{s}:=\one \xrightarrow{\phi_0} E(\one) \xrightarrow{E(s)} E(F(m)) \xrightarrow{\theta_m} F'(H(m)).\]
\item The image of a map of decorated cospans in $F\lCsp$
\[
\begin{tikzpicture}[scale=1.5]
\node (A) at (0,0.5) {$a$};
\node (A') at (0,-0.5) {$a'$};
\node (B) at (1,0.5) {$m$};
\node (C) at (2,0.5) {$b$};
\node (C') at (2,-0.5) {$b'$};
\node (D) at (1,-0.5) {$n$};
\node (E) at (3,0.5) {$s \in F(m)$};
\node (F) at (3,-0.5) {$s' \in F(n)$};
\node (G) at (1,-1) {$\tau \maps F(h)(s) \to s'$};
\path[->,font=\scriptsize,>=angle 90]
(A) edge node[above]{$i$} (B)
(C) edge node[above]{$o$} (B)
(A) edge node[left]{$f$} (A')
(C) edge node[right]{$g$} (C')
(A') edge node[below] {$i'$} (D)
(C') edge node[below] {$o'$} (D)
(B) edge node [left] {$h$} (D);
\end{tikzpicture}
\]
is the following map of $F'$-decorated cospans in $F'\lCsp$:
\[
\begin{tikzpicture}[scale=1.5]
\node (A'') at (4.75,0.5) {$H(a)$};
\node (A''') at (4.75,-0.5) {$H(a')$};
\node (B'') at (5.75,0.5) {$H(m)$};
\node (C'') at (6.75,0.5) {$H(b)$};
\node (C''') at (6.75,-0.5) {$H(b')$};
\node (D'') at (5.75,-0.5) {$H(n)$};
\node (E'') at (8.5,0.5) {$\overline{s}\in F'(H(m))$};
\node (F'') at (8.5,-0.5) {$\overline{s'} \in F'(H(n))$};
\node (G) at (5.75,-1) {$\lH(\tau) \maps F'(H(h))(\overline{s}) \to \overline{s'}$};
\path[->,font=\scriptsize,>=angle 90]
(A'') edge node[above]{$H(i)$} (B'')
(C'') edge node[above]{$H(o)$} (B'')
(A'') edge node[left]{$H(f)$} (A''')
(C'') edge node[right]{$H(g)$} (C''')
(A''') edge node[below] {$H(i')$} (D'')
(C''') edge node[below] {$H(o')$} (D'')
(B'') edge node [left] {$H(h)$} (D'');
\end{tikzpicture}
\]
where the decoration morphism $\lH(\tau)$ is defined as follows.  Treating $\tau$ as a natural transformation as in \cref{eq:taunattransf}, $\lH(\tau)$ is given by
\[
\begin{tikzpicture}[scale=1.5]
\node (A) at (-0.25,-0.5) {$1$};
\node (L) at (0.5,-0.5) {$E(1)$};
\node (B) at (2,0) {$E(F(m))$};
\node (D) at (2,-1) {$E(F(n))$};
\node (C) at (1.5,-0.5) {$\scriptstyle\Downarrow E(\tau)$};
\node (E) at (4,0) {$F^\prime(H(m))$};
\node (F) at (4,-1) {$F^\prime(H(n))$.};
\path[->,font=\scriptsize,>=angle 90]
(A) edge node [above] {$\phi_0$} (L)
(B) edge node [above] {$\theta_m$} (E)
(D) edge node [below] {$\theta_n$} (F)
(E) edge node [right] {$F'(H(h))$} (F)
(L) edge node[above]{$E(s)$} (B)
(L) edge node[below]{$E(s')$} (D)
(B) edge node[right]{$E(F((h))$} (D);
\end{tikzpicture}
\]
The square at right commutes strictly because $\theta$ is a natural transformation.
\end{itemize}

\begin{thm}
\label{thm:functoriality}
Given categories $\A$ and $\A'$ with finite colimits, lax monoidal pseudofunctors $F \maps (\A,+) \to (\Cat,\times)$ and $F' \maps (\A',+) \to
(\Cat,\times)$, a finite colimit preserving functor $H \maps \A \to \A'$, a lax monoidal pseudofunctor $E \maps (\Cat,\times) \to (\Cat,\times)$ and
a monoidal natural transformation $\theta$ as in the following diagram:
\[
\begin{tikzpicture}[scale=1.5]
\node (A) at (0,0) {$\A$};
\node (B) at (1,0) {$\Cat$};
\node (C) at (0,-1) {$\A'$};
\node (D) at (1,-1) {$\Cat$};
\node (E) at (0.5,-0.5) {$\scriptstyle\Downarrow\theta$};
\path[->,font=\scriptsize,>=angle 90]
(A) edge node[above]{$F$} (B)
(A) edge node[left]{$H$} (C)
(B) edge node[right]{$E$} (D)
(C) edge node[below]{$F'$} (D);
\end{tikzpicture}
\]
we obtain a double functor $\lH \maps F\lCsp \to F'\lCsp$ defined as above.  If $F, F'$ and $E$ are symmetric, then $\lH \maps F\lCsp \to F'\lCsp$ is a symmetric monoidal double functor.
\end{thm}

\begin{proof}
From the lax monoidal pseudofunctors $F$ and $F'$ we obtain Beck--Chevalley monoidal bifibrations $\Phi \maps \A/U \to \A$ and $\Phi' \maps \A'/U' \to \A'$ as in \cref{thm:decorated_cospans_1}.  In what follows we construct a strong monoidal functor $G \maps \A/U \to \A'/U'$ that makes this square commute:
\[
\begin{tikzpicture}[scale=1.5]
\node (A) at (0,0) {$\A/U$};
\node (B) at (1,0) {$\A'/U'$};
\node (C) at (0,-1) {$\A$};
\node (D) at (1,-1) {$\A'$};
\path[->,font=\scriptsize,>=angle 90]
(A) edge node[above]{$G$} (B)
(A) edge node[left]{$\Phi$} (C)
(B) edge node[right]{$\Phi'$} (D)
(C) edge node[below]{$H$} (D);
\end{tikzpicture}
\]
Such a commuting square is a morphism in Shulman's 2-category of Beck--Chevalley monoidal bifibrations.  Applying Shulman's 2-functor $\mathbb{F}\mathbf{r}$ to this morphism \cite[Theorem 14.11]{Shulman2008}, we obtain the desired double functor
\[ \mathbb{H} \maps  F\lCsp \to F'\lCsp .\]

We described $\A/U$ and its monoidal structure in the proof of \cref{thm:decorated_cospans_1},
and of course the case of $\A'/U'$ is analogous.   We define $G \maps \A/U \to \A'/U'$ as follows.  On objects $G$ is given by
\[  G(a \xrightarrow{i} b, s) \; = \; (H(a) \xrightarrow{H(i)} H(b), \underline{s}), \]
where to define $\underline{s}$ we treat $s \in F(b)$ as a functor
$s \maps 1 \to F(b)$ and let $\underline{s} \in F'(H(b))$ be the composite functor
\[  1 \xrightarrow{\phi_0} E(1) \xrightarrow{E(s)} E(F(b)) \xrightarrow{\;\theta_b\;} F'(H(b)) .\]
On morphisms $(f,g,h) \maps (a \xrightarrow{i} b, s) \; \to \;
(a' \xrightarrow{i'} b', s')$, $G$ is given by
\[   G(f,g,h) = (H(f),H(g),\underline{h}) \]
where to define $\underline{h}$ we treat $h \maps F(g)(s) \to s'$ as a natural transformation
\[
\begin{tikzpicture}[scale=1.5]
\node (G) at (-3.5,-0.5) {$1$};
\node (H) at (-2,0) {$F(b)$};
\node (I) at (-2,-1) {$F(b')$};
\node (J) at (-2.5,-0.5) {$\scriptstyle\Downarrow h$};
\path[->,font=\scriptsize,>=angle 90]
(G) edge node [above] {$s$} (H)
(G) edge node [below] {$s'$} (I)
(H) edge node [right] {$F(g)$} (I);
\end{tikzpicture}
\]
and let $\underline{h} \maps F'(H(g))(\underline{s}) \to \underline{s}'$ be this
natural transformation:
\[
\begin{tikzpicture}[scale=1.5]
\node (A) at (-0.25,-0.5) {$1$};
\node (L) at (0.5,-0.5) {$E(1)$};
\node (B) at (2,0) {$E(F(b))$};
\node (D) at (2,-1) {$E(F(b'))$};
\node (C) at (1.5,-0.5) {$\scriptstyle\Downarrow E(h)$};
\node (E) at (4,0) {$F^\prime(H(b))$};
\node (F) at (4,-1) {$F^\prime(H(b'))$.};
\path[->,font=\scriptsize,>=angle 90]
(A) edge node [above] {$\phi_0$} (L)
(B) edge node [above] {$\theta_b$} (E)
(D) edge node [above] {$\theta_{b'}$} (F)
(E) edge node [right] {$F'(H(g))$} (F)
(L) edge node[above]{$E(s)$} (B)
(L) edge node[below]{$E(s')$} (D)
(B) edge node[right]{$E(F(g))$} (D);
\end{tikzpicture}
\]
One can check that with these definitions $G$ is a functor.

To make $G$ into a monoidal functor, we need to equip it with a laxator
\[  \gamma \maps G(a \xrightarrow{i} b, s) \otimes G(a' \xrightarrow{i'} b', s') \;
\to \; G((a \xrightarrow{i} b, s) \otimes (a' \xrightarrow{i'} b', s')).\]
Note that
\[   G(a \xrightarrow{i} b, s) \otimes G(a' \xrightarrow{i'} b', s') =
\left(H(a) + H(a') \xrightarrow{H(i) + H(i')} H(b) + H(b') , \; \phi'_{H(b), H(b')}(\underline{s},
\underline{s}')\right) \]
where $\phi'$ is the laxator for $F'$, while
\[    G((a \xrightarrow{i} b, s) \otimes (a' \xrightarrow{i'} b', s')) =
\left(H(a+a') \xrightarrow{H(i+i')} H(b+b'), \; \underline{\phi_{b,b'}(s,s')} \right)
\]
where $\phi$ is the laxator for $F$.  Since $H$ preserves finite colimits, the laxator
$\gamma$ is obvious except for the morphism from $\phi'_{H(b), H(b')}(\underline{s},
\underline{s}')$ to $\underline{\phi_{b,b'}(s,s')}$.   This is given by the following natural
transformation:
\[
\begin{tikzpicture}[scale=1.5]
\node () at (3,-0.5) {$\scriptscriptstyle{\stackrel{E_{s,s'}} \cong}$};
\node () at (1.5,-0.5) {$\scriptscriptstyle{\stackrel{\lambda} \cong}$};
\node (L1) at (1,0) {$\scriptscriptstyle{1 \times E(1)}$};
\node (L2) at (1,-1) {$\scriptscriptstyle{E(1)}$};
\node (K) at (0,0) {$\scriptscriptstyle{1 \times 1}$};
\node (K2) at (0,-1) {$\scriptscriptstyle{1}$};
\node (B) at (2,0) {$\scriptscriptstyle{E(1) \times E(1)}$};
\node (D) at (2,-1) {$\scriptscriptstyle{E(1 \times 1)}$};
\node (E) at (4,0) {$\scriptscriptstyle{E(F(b)) \times E(F(b'))}$};
\node (F) at (4,-1) {$\scriptscriptstyle{E(F(b) \times F(b'))}$};
\node (G) at (6,0) {$\scriptscriptstyle{F'(H(b)) \times F'(H(b'))}$};
\node (H) at (8,0) {$\scriptscriptstyle{F'(H(b)+H(b'))}$};
\node (I) at (6,-1) {$\scriptscriptstyle{E(F(b+b'))}$};
\node (J) at (8,-1) {$\scriptscriptstyle{F'(H(b+b'))}$};
\path[->,font=\scriptsize,>=angle 90]
(K) edge node [left] {$\lambda$} (K2)
(K) edge node [above] {$1 \times E_0$} (L1)
(K2) edge node [below] {$E_0$} (L2)
(H) edge node [right] {$F'(H_{b,b'})$.} (J)
(E) edge node [above] {$\theta_b \times \theta_{b'}$} (G)
(G) edge node [above] {$\phi'_{H(b),H(b')}$} (H)
(F) edge node [below] {$E(\phi_{b,b'})$} (I)
(I) edge node [below] {$\theta_{b+b'}$} (J)
(B) edge node [above] {$E(s) \times E(s')$} (E)
(D) edge node [below] {$E(s \times s')$} (F)
(E) edge node [right] {$E_{F(b),F(b')}$} (F)
(L1) edge node[above]{$E_0 \times 1$} (B)
(L1) edge node [left] {$\lambda$} (L2)
(L2) edge node[below]{$E(\lambda^{-1})$} (D)
(B) edge node[right]{$E_{1,1}$} (D);
\end{tikzpicture}
\]
We also need a natural transformation expressing lax preservation of the unit object by $G$, which
is defined similarly.   One can show that this laxator and unit morphism satisfy the coherence laws of a lax monoidal functor, and they are invertible, so $G$ is strong monoidal functor.   If furthermore $F, F'$ and $E$ are symmetric, one can check that $G$ is symmetric monoidal, and that $\lH$ is a symmetric monoidal double functor.
\end{proof}

\section{Structured cospans} \label{Structured}

Structured cospans are an alternative approach to equipping the apex of a cospan with
extra data \cite{BC}.    Here we recall this formalism and give a new construction of the double category of structured cospans that clarifies their
relation to decorated cospans.   This new construction again uses Shulman's \cref{lem:shulman}. For concrete examples of this formalism, see
\cref{Applications}.

\begin{thm}\label{thm:structured_cospans}
Given categories $\A$ and $\X$ with finite colimits and $L \maps \A \to \X$ a functor preserving finite colimits, there is a symmetric monoidal double category ${}_L\lCsp(\X)$ in which
\begin{itemize}
\item an object is an object of $\A$,
\item a vertical 1-morphism is a morphism of $\A$,
\item a horizontal 1-cell from $a$ to $b$ is an $L$-\define{structured cospan}, that is,
a diagram in $\X$ of the form
\begin{displaymath}
\begin{tikzpicture}[scale=1.5]
\node (A) at (0,0) {$L(a)$};
\node (B) at (1,0) {$x$};
\node (C) at (2,0) {$L(b)$,};
\path[->,font=\scriptsize,>=angle 90]
(A) edge node[above]{$i$} (B)
(C) edge node[above]{$o$} (B);
\end{tikzpicture}
\end{displaymath}
\item a 2-morphism is a \define{map of} $L$-\define{structured cospans}, that is, a commutative diagram in $\X$ of the form
\begin{displaymath}
\begin{tikzpicture}[scale=1.5]
\node (A) at (0,0) {$L(a)$};
\node (B) at (1,0) {$x$};
\node (C) at (2,0) {$L(b)$};
\node (A') at (0,-1) {$L(a')$};
\node (B') at (1,-1) {$x'$};
\node (C') at (2,-1) {$L(b')$.};
\path[->,font=\scriptsize,>=angle 90]
(A) edge node[above]{$i$} (B)
(C) edge node[above]{$o$} (B)
(A') edge node[below]{$i'$} (B')
(C') edge node[below]{$o'$} (B')
(A) edge node [left]{$L(f)$} (A')
(B) edge node [left]{$\alpha$} (B')
(C) edge node [right]{$L(g)$} (C');
\end{tikzpicture}
\end{displaymath}
\end{itemize}
Vertical composition is done using composition in $\A$, while horizontal composition
is done using pushouts in $\X$. The tensor product of two horizontal 1-cells is
\[
\begin{tikzcd}[column sep={.4in,between origins}, row sep=.08in]
& x &&&& x' &&&& x+x' & \\
&&&\ot&&&& = &&&& \\
L(a)\ar[uur,"i" pos=0.3]&& L(b)\ar[uul,"o"' pos=0.3] && L(a')\ar[uur,"i'" pos=0.3] && L(b')\ar[uul,"o'"' pos=0.3] && L(a+a')\ar[uur,"i+i'" pos=0.3] && L(b+b')\ar[uul,"o+o'"' pos=0.3]
\end{tikzcd}
\]
where $i + i'$ and $o + o'$ are defined using the fact that $L$ preserves binary coproducts, and tensor product of two 2-morphisms is given by:
\[
\begin{tikzpicture}[scale=1.5]
\node (E) at (3,0) {$L(a_1)$};
\node (F) at (5,0) {$L(b_1)$};
\node (G) at (4,0) {$x_1$};
\node (E') at (3,-1) {$L(a_2)$};
\node (F') at (5,-1) {$L(b_2)$};
\node (G') at (4,-1) {$x_2$};
\node (E'') at (6.5,0) {$L(a_1')$};
\node (F'') at (8.5,0) {$L(b_1')$};
\node (G'') at (7.5,0) {$x_1'$};
\node (E''') at (6.5,-1) {$L(a_2')$};
\node (F''') at (8.5,-1) {$L(b_2')$};
\node (G''') at (7.5,-1) {$x_2'$};
\node (X) at (5.75,-0.5) {$\otimes$};
\node (E'''') at (4,-2) {$L(a_1 + a_1')$};
\node (F'''') at (7.5,-2) {$L(b_1 + b_1')$};
\node (G'''') at (5.75,-2) {$x_1 + x_1'$};
\node (E''''') at (4,-3) {$L(a_2 + a_2')$};
\node (F''''') at (7.5,-3) {$L(b_2 + b_2').$};
\node (G''''') at (5.75,-3) {$x_2 + x_2'$};
\node (Y) at (3,-2.5) {$=$};
\path[->,font=\scriptsize,>=angle 90]
(F) edge node[above]{$o_1$} (G)
(E) edge node[left]{$L(f)$} (E')
(F) edge node[right]{$L(g)$} (F')
(G) edge node[left]{$\alpha$} (G')
(E) edge node[above]{$i_1$} (G)
(E') edge node[below]{$i_2$} (G')
(F') edge node[below]{$o_2$} (G')
(F'') edge node[above]{$o_1'$} (G'')
(E'') edge node[left]{$L(f')$} (E''')
(F'') edge node[right]{$L(g')$} (F''')
(G'') edge node[left]{$\alpha'$} (G''')
(E'') edge node[above]{$i_1'$} (G'')
(E''') edge node[below]{$i_2'$} (G''')
(F''') edge node[below]{$o_2'$} (G''')
(F'''') edge node[above]{$o_1 + o_1'$} (G'''')
(E'''') edge node[left]{$L(f + f')$} (E''''')
(F'''') edge node[right]{$L(g + g')$} (F''''')
(G'''') edge node[left]{$\alpha + \alpha'$} (G''''')
(E'''') edge node[above]{$i_1 + i_1'$} (G'''')
(E''''') edge node[below]{$i_2 + i_2'$} (G''''')
(F''''') edge node[below]{$o_2 + o_2'$} (G''''');
\end{tikzpicture}
\]
\end{thm}

\begin{proof}
This was proved in \cite[Theorems~2.3 \& 3.9]{BC}, where all the structures are specified in detail.  In fact, the double category structure only requires that $\X$ have pushouts, whereas the symmetric monoidal structure also requires that $\X$ and $\A$ have finite coproducts and that $L$ preserve these \cite[Theorem~3.2.3]{CourserThesis}.

Our new proof is analogous to that of \cref{thm:decorated_cospans_1} in that we apply \cref{lem:shulman} to a Beck--Chevalley symmetric monoidal bifibration $\Psi \maps L/\X \to \A$,
and define ${}_L \lCsp(\X)$ to be $\Fr(\Psi)$.    First, note that in the comma category $L/\X$:
\begin{itemize}
\item an object is a pair $(a, L(a) \xrightarrow{} x)$ consisting of an object $a \in \A$ and a morphism $L(a) \xrightarrow{i} x$ in $\X$;
\item a morphism from $(a,L(a) \xrightarrow{i} x)$ to $(a',L(a') \xrightarrow{i'} x')$ is a pair $(f,g)$ of morphisms $f \maps a \to a'$ and $g \maps x \to x'$ such that this square commutes:
\[
\begin{tikzpicture}[scale=1.5]
\node (A) at (0,0) {$L(a)$};
\node (B) at (1.5,0) {$L(a')$};
\node (C) at (0,-1) {$x$};
\node (D) at (1.5,-1) {$x'$.};
\path[->,font=\scriptsize,>=angle 90]
(A) edge node[above]{$L(f)$} (B)
(A) edge node[left]{$i$} (C)
(B) edge node[right]{$i'$} (D)
(C) edge node[below]{$g$} (D);
\end{tikzpicture}
\]
\end{itemize}
Next, we define the functor $\Psi \maps L/\X \to \A$ on objects by $\Psi(a,L(a) \xrightarrow{i} x)=a$ and on morphisms by $\Psi(f,g)=f$.

The comma category $L/\X$ has finite colimits because $\A$ and $\X$ have finite colimits and $L$
preserves them.   It thus becomes symmetric monoidal with the chosen finite coproducts providing the monoidal structure.   The tensor product of two objects $(a,L(a) \xrightarrow{i} x)$ and $(a,L(a) \xrightarrow{i'} x')$ is given by
\[  (a,L(a) \xrightarrow{i} x) + (a',L(a') \xrightarrow{i'} x')=(a+a',L(a+a')  \xrightarrow{\sim} L(a)+L(a') \xrightarrow{i+i'} x+x'). \]
The monoidal unit is $(0_\A, L(0_\A) \xrightarrow{!} 0_\X)$.  It is clear that $\Psi \maps (L/\X, +) \to (\A,+)$ is a strict monoidal functor.

The functor $\Psi$ is both a fibration and an opfibration. Suppose we are given a morphism $f \maps a \to b$ in $\A$.  For any object $(b,L(b) \xrightarrow{i} x) \in L/\X$ over $b$, a cartesian lifting of $f$ is given by
\[ (f,1_x) \maps (a,L(a) \xrightarrow{i L(f)} x) \; \to \; (b,L(b) \xrightarrow{i} x) . \]
For any object $(a,L(a) \xrightarrow{i} x)$ in $L/\X$ over $a$, a cocartesian lifting of $f$ is given by
\[ (f,\kappa) \maps (a,L(a) \xrightarrow{i} x) \to (b,L(b) \xrightarrow{\theta} x+_{L(a)} L(b)) \]
where $\kappa$ and $\theta$ arise from the pushout in this diagram:
\[
\begin{tikzpicture}[scale=1.5]
\node (A) at (0,0) {$L(a)$};
\node (B) at (2.25,0) {$L(c)$};
\node (C) at (1.5,1) {$L(b)$};
\node (D) at (0,-2) {$x$};
\node (E) at (2.25,-2) {$y$};
\node (F) at (1.5,-1) {$x+_{L(a)} L(b)$};
\path[->,font=\scriptsize,>=angle 90]
(D) edge [dashed] node [above] {$\kappa$} (F)
(F) edge [dashed] node [left] {$\exists ! \psi$} (E)
(D) edge node [above] {$g'$} (E)
(A) edge node [left] {$i$} (D)
(C) edge [dashed] node [left,near end] {$\theta$} (F)
(B) edge node {$$} (E)
(A) edge node[above]{$L(f')$} (B)
(C) edge node [right] {$L(p)$} (B)
(A) edge node[left]{$L(f)$} (C);
\end{tikzpicture}
\]
This diagram also gives the proof that $(f,\kappa)$ is a cocartesian lifting of $f$.
Both cartesian and cocartesian liftings are clearly preserved by the tensor product of $L/\X$, so
$\Psi$ is a monoidal bifibration.

Lastly, we show that $\Psi \maps L/\X \to \A$ satisfies the Beck--Chevalley condition.
For this, given a pushout square in $\A$:
\[
\begin{tikzpicture}[scale=1.5]
\node (A) at (0,0) {$a$};
\node (B) at (1,0) {$b$};
\node (C) at (0,-1) {$c$};
\node (D) at (1,-1) {$d$};
\path[->,font=\scriptsize,>=angle 90]
(A) edge node[above]{$h$} (B)
(A) edge node[left]{$k$} (C)
(B) edge node[right]{$g$} (D)
(C) edge node[below]{$f$} (D);
\end{tikzpicture}
\]
we need to show that the natural transformation
\[
\begin{tikzpicture}[scale=1.5]
\node (A) at (0,0) {$(L/\X)_a$};
\node (B) at (1.5,0) {$(L/\X)_b$};
\node (C) at (0,-1) {$(L/\X)_c$};
\node (D) at (1.5,-1) {$(L/\X)_d$};
\node (E) at (0.75,-0.5) {$\scriptstyle\Downarrow \theta$};
\path[->,font=\scriptsize,>=angle 90]
(B) edge node[above]{$h^*$} (A)
(A) edge node[left]{$k_!$} (C)
(B) edge node[right]{$g_!$} (D)
(D) edge node[below]{$f^*$} (C);
\end{tikzpicture}
\]
defined in \eqref{theta} is an isomorphism.  Let $(b,L(b) \xrightarrow{i} x)$ be an object of $(L/\X)_b$.  Going left, we precompose with $L(h) \maps L(a) \to L(b)$ to obtain the object $(a,L(a) \xrightarrow{i  L(h)} x)$ in $(L/\X)_a$.  Then going down, we push forward along $L(k) \maps L(a) \to L(c)$:
\[
\begin{tikzpicture}[scale=1.5]
\node (A) at (0,0) {$L(a)$};
\node (B) at (1.5,0) {$x$};
\node (C) at (0,-1) {$L(c)$};
\node (D) at (1.5,-1) {$L(c)+_{L(a)}x$};
\node at (1.1,-0.7) {\Large{$\ulcorner$}};
\path[->,font=\scriptsize,>=angle 90]
(A) edge node[above]{$i L(h)$} (B)
(A) edge node[left]{$L(k)$} (C)
(B) edge node[right]{$$} (D)
(C) edge node[below]{$\psi$} (D);
\end{tikzpicture}
\]
to obtain the object $(c,L(c) \xrightarrow{\psi} L(c)+_{L(a)} x)$ of $(L/ \X)_c$. For the other route, first going down, we push the object $(b,L(b) \xrightarrow{i} x)$ forward along $L(g) \maps L(b) \to L(d)$ by taking the following pushout:
\[
\begin{tikzpicture}[scale=1.5]
\node (A) at (0,0) {$L(b)$};
\node (B) at (1.5,0) {$x$};
\node (C) at (0,-1) {$L(d)$};
\node (D) at (1.5,-1) {$L(d)+_{L(b)}x$};
\node at (1.1,-0.7) {\Large{$\ulcorner$}};
\path[->,font=\scriptsize,>=angle 90]
(A) edge node[above]{$i$} (B)
(A) edge node[left]{$L(g)$} (C)
(B) edge node[right]{$$} (D)
(C) edge node[below]{$\phi$} (D);
\end{tikzpicture}
\]
which yields the object $(d,L(d) \xrightarrow{\phi} L(d)+_{L(b)} x)$ of $(L/\X)_d$.  Then going left, we precompose with $L(f) \maps L(c) \to L(d)$ to obtain the object $(c,L(c) \xrightarrow{\phi L(f)} L(d)+_{L(b)} x)$ of $(L/\X)_c$.

The natural transformation $\theta$ gives a morphism
\[   \theta \maps (c,L(c) \xrightarrow{\psi} L(c)+_{L(a)} x) \; \to \; (c,L(c) \xrightarrow{\phi  L(f)} L(d)+_{L(b)} x) \]
in $(L/\X)_c$.   As in the proof of \cref{thm:decorated_cospans_1}, this arises from the universal property of the pushout:
\[
\begin{tikzpicture}[scale=1.5]
\node (A) at (0,0) {$L(b)$};
\node (B) at (1.5,0) {$x$};
\node (C) at (0,-1) {$L(d)$};
\node (D) at (1.5,-1) {$L(d)+_{L(b)}x$};
\node (E) at (-1.5,0) {$L(a)$};
\node (F) at (-1.5,-1) {$L(c)$};
\node (G) at (2.5,-2) {$L(c)+_{L(a)}x$};
\node at (1.1,-0.7) {\Large{$\ulcorner$}};
\path[->,font=\scriptsize,>=angle 90]
(E) edge node [left] {$L(k)$} (F)
(E) edge node [above] {$L(h)$} (A)
(F) edge node [below] {$L(f)$} (C)
(A) edge node[above]{$i$} (B)
(A) edge node[left]{$L(g)$} (C)
(B) edge node[right]{$$} (D)
(C) edge node[below]{$\phi$} (D)
(F) edge [bend right] node[below]{$\psi$} (G)
(G) edge [dashed] node[below left]{$\exists ! \;\zeta $} (D)
(B) edge [bend left] (G);
\end{tikzpicture}
\]
Namely, we have $\theta = (1,\zeta)$.   But when the original square in $A$ is a pushout,
the square at left above is also a pushout, since $L$ preserves finite colimits.  Since the pasting
of pushout squares is a pushout, $\zeta$ and hence $\theta$ is an isomorphism, so $\Psi$ satisfies the Beck--Chevalley condition.
\end{proof}

\section{Structured versus decorated cospans} \label{EquivDoubleCats}

We now compare structured and decorated cospans.   In \cref{thm:decorated_cospans_1} we built a double category of decorated cospans from a bifibration $\Phi \maps \A/U \to \A$.    In \cref{thm:structured_cospans}, we built a double category of structured cospans from a bifibration $\Psi \maps L/\X \to \A$ in a very similar way.   We now show that under certain conditions $L$ is left adjoint to $U$.   Whenever this happens, $\A/U$ is isomorphic to $L/\X$, by a simple and general fact about arrow categories.    The bifibrations $\Phi$ and $\Psi$ are also isomorphic.   Because  Shulman's construction in \cref{lem:shulman} is actually functorial, it follows that the double category of decorated cospans is isomorphic to the double category of structured cospans.

The key issue is thus to determine when $L$ is left adjoint to $U$.  For this, let $\Rex$ be the 2-category of categories with finite colimits, functors preserving finite colimits, and natural transformations.  Also let $\SMC$ be the 2-category of symmetric monoidal categories, symmetric strong monoidal functors and natural transformations.   Recall that for us a category $\C \in \Rex$ comes with a choice of finite colimits, so it gives a specific cocartesian monoidal category $(\C,+)$, and this induces a 2-functor $\Rex \to \SMC$.

Our main result is this:

\begin{thm} \label{thm:equiv}
Suppose $\A$ has finite colimits and $F \maps (\A,+) \to (\Cat,\times)$ is a symmetric lax monoidal pseudofunctor. If the corresponding pseudofunctor $F \maps \A \to \SMC$ from \cref{lem:MonGroth}
 factors through the above 2-functor $\Rex\to\SMC$, then the symmetric monoidal double categories $F\lCsp$ of decorated cospans and ${}_L\lCsp(\inta F)$ of structured
cospans are isomorphic, where $L \maps \A \to \inta F$ is a left adjoint of the induced Grothendieck opfibration $U \maps \inta F \to \A$.
\end{thm}

The hypothesis of this theorem is essentially a way of asking that the fibers of the induced opfibration $U \maps \inta F \to \A$ have finite colimits which are preserved by the reindexing functors, and that the induced fiberwise monoidal structure is cocartesian.   Expanding on this a bit:
the fibers of $U$ are the same as the categories $F(a)$ for all $a \in \A$.   The lax monoidal pseudofunctor structure of $F$ gives rise to a specific symmetric monoidal structure on each category $F(a)$, as in \cref{eq:explicitstructure1}.  The hypothesis of the theorem asks that the resulting pseudofunctor $F \maps \A \to \SMC$ factors through $\Rex$.  This implies
that the symmetric monoidal structure on each category $F(a)$ is cocartesian.

It will be important to know that under the hypotheses of
\cref{thm:equiv} the opfibration $U$ preserves all finite colimits.   For this we need the following lemma.

\begin{lem}[\textbf{Hermida}] \label{lem:fibrewiselimits}
Suppose $\J$ is a small category and $U \maps \X \to \A$ is an opfibration where the base $\A$ has $\J$-colimits.  Then the following are equivalent:
\begin{enumerate}
 \item all fibers have $\J$-colimits, and the reindexing functors preserve them;
 \item the total category $\X$ has $\J$-colimits, and $U$ preserves them.
\end{enumerate}
Moreover, if $\X$ has $\J$-colimits and $U$ preserves them, for any choice of $\J$-colimits in $\A$, they can be chosen in $\X$ in such a way that $U$ strictly preserves them.
\end{lem}

\begin{proof}
See \cite[Corollary~4.9]{Hermida1999}, and for the final statement \cite[Remark~4.11]{Hermida1999}.
\end{proof}

The first part of this lemma asserts the existence of colimits \emph{locally} in each fiber, and if we let $\J$ range over all finite categories it says that the corresponding pseudofunctor $F \maps \A \to \Cat$ lands in the sub-2-category $\Rex$. The second part formulates the existence of colimits \emph{globally} in the total category $\inta F$, and if we let $\J$ range over all finite categories it says that $\X$ has finite colimits and $U$ preserves all finite colimits.

\begin{cor} \label{cor:fcocMonGroth}
 Suppose $\A$ has finite colimits and $F \maps (\A,+) \to (\Cat,\times)$ is a symmetric lax monoidal pseudofunctor for which the corresponding
pseudofunctor $F  \maps \A \to \SMC$ from \cref{lem:MonGroth} factors through $\Rex\to\SMC$.  Then $\inta F$ has all finite colimits and  the induced
opfibration $U\maps \inta F \to \A$ preserves them.  Moreover we can choose finite colimits in $\inta F$ so that $U$ preserves them strictly.
\end{cor}

In what follows we also need a left adjoint $L$ to the opfibration $U$.  The following result provides sufficient conditions for that.  Following Gray \cite{Gray}, we say a functor has a `left adjoint right inverse' or \define{lari} if it has a left adjoint where the unit of the adjunction is the identity.

\begin{lem}[\textbf{Gray}] \label{prop:opfibtolari}
Let $U \maps\X \to \A$ be an opfibration.   Then $U$ has a lari if its fibers have initial objects
that are preserved by the reindexing functors.
\end{lem}

\begin{proof}
 This is \cite[Proposition~4.4]{Gray}.    Suppose each fiber $\X_a$ of the opfibration $U$ has an initial object $\bot_a$ and these objects are preserved (up to isomorphism) by the reindexing functors.   Define $L \maps \A \to \X$ on objects $a \in \A$ by $L(a) = \bot_a$.   Given a morphism $f\maps a\to a'$ in $\A$, define $L(f)$ to be the composite
\[
  \bot_a\xrightarrow{\mathrm{Cocart}(f,\bot_a)}f_!(\bot_a)\xrightarrow{\;\chi_a\;}\bot_{a'}
\]
where $\mathrm{Cocart}(f,\bot_a)$ is the cocartesian lifting of $f$ to $\bot_a$ and $\chi_a$ is the unique isomorphism between two initial objects in
the fiber above $a'$.    The functor $L$ then becomes left adjoint to $U$ with unit $\iota_a \maps a \to U(L(a))$ being the identity, using the fact
that $U(L(a)) = U(\bot_a) = a$.
\end{proof}

We now have all the necessary background to construct an isomorphism between the double category of decorated cospans and the double category of structured cospans, starting from a symmetric lax monoidal pseudofunctor $F \maps (\A, +) \to (\Cat, \times)$ whose corresponding pseudofunctor $ F\maps \A \to \SMC$ factors through $\Rex$.

\begin{proof}[Proof of \cref{thm:equiv}]
Since we are assuming the pseudofunctor $F \maps \A \to \SMC$ factors through $\Rex$,
\cref{cor:fcocMonGroth} implies that the Gro\-the\-ndieck construction gives rise to a category $\inta F$ with finite colimits, and we can choose these in such a way that the corresponding opfibration $U \maps \inta F \to \A$ strictly preserves them.  We do this in what follows.

By \cref{prop:opfibtolari}, $U$ has a left adjoint $L\maps\A\to\inta F$ with $UL=1_\A$. Diagrammatically,
\[
 F\maps\A\to\Cat\quad\mapsto\quad\begin{tikzcd}[baseline=.3]\inta F\ar[d,"U"'] \\ \A \end{tikzcd}\quad\mapsto\quad\begin{tikzcd}\A\ar[r,bend left,pos=.55,"L"]\ar[r,phantom,"\bot"description] & \inta F\ar[l,bend left,pos=.45,"U"]\end{tikzcd}
\]
describes the construction of the adjunction from the original $F$. Explicitly, the left adjoint maps each object $a \in \A$ to the initial object in its fiber $\X_a$, namely $L(a)=(a,I_a)$, where $I_a$ is defined as in \cref{eq:explicitstructure1}.

As a left adjoint, $L$ preserves all colimits that exist between the categories $\A$ and $\inta F$, which have finite colimits.   Thus, we can construct the symmetric monoidal double category of structured cospans $_L\lCsp(\inta F)$ as in \cref{thm:structured_cospans}.  To show that this is isomorphic to the symmetric monoidal double category of decorated cospans $F\lCsp$ given by  \cref{thm:decorated_cospans_2}, we use a result of Shulman \cite[Theorem~14.11]{Shulman2010}.    Namely, the $\mathbb{F}\mathbf{r}$ construction of \cref{lem:shulman} extends to a 2-functor
\[ \mathbb{F}\mathbf{r} \maps \mathbf{BCMF} \to \mathbf{FDbl} \]
where $\mathbf{BCMF}$ is the 2-category consisting of
\begin{itemize}
\item Beck--Chevalley monoidal bifibrations,
\item strong monoidal morphisms of bifibrations, and
\item monoidal transformations of bifibrations
\end{itemize}
while $\mathbf{FDbl}$ is the 2-category of
\begin{itemize}
\item fibrant double categories,
\item double functors, and
\item double transformations.
\end{itemize}
Given two monoidal bifibrations $\Phi \maps \A \to \mathsf{B}$ and $\Phi' \maps \A' \to \mathsf{B}'$, a `strong monoidal morphism of bifibrations' from $\Phi$ to $\Phi'$ consists of a pair of strong monoidal functors $F_1$ and $F_2$ making the following square commute
\[
\begin{tikzpicture}[scale=1.5]
\node (A) at (0,0) {$\A$};
\node (B) at (1,0) {$\A'$};
\node (C) at (0,-1) {$\mathsf{B}$};
\node (D) at (1,-1) {$\mathsf{B}'$};
\path[->,font=\scriptsize,>=angle 90]
(A) edge node[above]{$F_1$} (B)
(A) edge node[left]{$\Phi$} (C)
(B) edge node[right]{$\Phi'$} (D)
(C) edge node[below]{$F_2$} (D);
\end{tikzpicture}
\]
not just as functors but as strong monoidal functors.   In our case the two monoidal bifibrations are the functors $\Phi$ and $\Psi$ of Theorems \ref{thm:decorated_cospans_2} and \ref{thm:structured_cospans}, respectively.  These share $\A$ as a common base, so we only need a single strong monoidal functor $F \maps \A/U \to L/\X$ making this diagram of strong monoidal functors commute:
\[
\begin{tikzpicture}[scale=1.0]
\node (A) at (0,0) {$\A/U$};
\node (B) at (2,0) {$L/ \X$};
\node (C) at (1,-1.5) {$\A$.};
\path[->,font=\scriptsize,>=angle 90]
(A) edge node[above]{$F$} (B)
(A) edge node[left]{$\Phi$} (C)
(B) edge node[right]{$\Psi$} (C);
\end{tikzpicture}
\]
There is an isomorphism $F \maps \A/U \to L/\X$ arising from the fact that $U$ is right adjoint to $L$.   It is clear that $F$ is strong monoidal via componentwise binary coproducts which are preserved by both $L$ and $U$, and that $\Psi F=\Phi$ as strong monoidal functors.  As $F$ is an isomorphism, we obtain an isomorphism
\[  \mathbb{F}\mathbf{r}(F) \maps \mathbb{F}\mathbf{r}(\Phi) \to \mathbb{F}\mathbf{r}(\Psi)\] between the two monoidal double categories $\mathbb{F}\mathbf{r}(\Phi)=F\lCsp$ and $\mathbb{F}\mathbf{r}(\Psi)= {}_L \lCsp(\inta F)$ of Theorems \ref{thm:decorated_cospans_2} and \ref{thm:structured_cospans}, respectively.   One can check by hand that this isomorphism
is symmetric monoidal.
\end{proof}

Given the hypothesis of \cref{thm:equiv}, the isomorphism $\Fr(F)$ between decorated and
structured cospans works concretely as follows.  First recall that $\X = \inta F$, so that an object of $\X$ is a pair $(a,s)$ with $a \in \A$ and $s \in F(a)$.   The functor $U \maps \X \to \A$ maps $(a,s)$ to $a$.    Its left adjoint $L \maps \A \to \X$ maps $a \in \A$ to $(a,I_a)$, where $I_a \in F(a)$ is the \define{trivial decoration} given by the composite
\[  \one\xrightarrow{\phi_0}F(0)\xrightarrow{F(!_a)}F(a). \]
The object $L(a) = (a,I_a)$ is the initial object in the fiber of $\inta F$ over $a$.

The isomorphism $\Fr(F) \maps F\lCsp \to {}_L\lCsp(\inta F)$  is the identity on objects and vertical morphisms, which are just objects and morphisms of $\A$.   It maps the decorated cospan
\[
\begin{tikzpicture}[scale=1.5]
\node (A) at (0,0) {$a$};
\node (B) at (1,0) {$m$};
\node (C) at (2,0) {$b$,};
\node (D) at (3,0) {$s \in F(m)$};
\path[->,font=\scriptsize,>=angle 90]
(A) edge node[above]{$i$} (B)
(C) edge node[above]{$o$} (B);
\end{tikzpicture}
\]
to the structured cospan
\[
\begin{tikzpicture}[scale=1.5]
\node (A) at (0,0) {$(a,I_a)$};
\node (B) at (1.5,0) {$(m,s)$};
\node (C) at (3,0) {$(b,I_b)$.};
\path[->,font=\scriptsize,>=angle 90]
(A) edge node[above]{$(i,!)$} (B)
(C) edge node[above]{$(o,!)$} (B);
\end{tikzpicture}
\]
where $! \maps F(i)(I_a) \to s$ is the unique morphism with this domain and
codomain (recall that $F(i)$ preserves initial objects), and similarly for
$! \maps F(o)(I_b) \to s$.   Finally, $\Fr(\Phi)$ sends a map of decorated cospans:
\[
\begin{tikzpicture}[scale=1.5,baseline=(current bounding box.center)]
\node (A) at (0,0.5) {$a$};
\node (A') at (0,-0.5) {$a'$};
\node (B) at (1,0.5) {$m$};
\node (C) at (2,0.5) {$b$};
\node (C') at (2,-0.5) {$b'$};
\node (D) at (1,-0.5) {$m'$};
\node (E) at (3,0.5) {$s \in F(m)$};
\node (F) at (3,-0.5) {$s' \in F(m')$};
\node (G) at (1,-1) {$\tau \maps F(h)(s) \to s'$};
\path[->,font=\scriptsize,>=angle 90]
(A) edge node[above]{$i$} (B)
(C) edge node[above]{$o$} (B)
(A) edge node[left]{$f$} (A')
(C) edge node[right]{$g$} (C')
(A') edge node[below] {$i'$} (D)
(C') edge node[below] {$o'$} (D)
(B) edge node [left] {$h$} (D);
\end{tikzpicture}
\]
to this map of structured cospans:
\[
\begin{tikzpicture}[scale=1.5,baseline=(current bounding box.center)]
\node (A) at (0,0.5) {$(a,I_a)$};
\node (A') at (0,-0.5) {$(a',I_{a'})$};
\node (B) at (1.5,0.5) {$(m,s)$};
\node (C) at (3,0.5) {$(b,I_b)$};
\node (C') at (3,-0.5) {$(b',I_{b'})$.};
\node (D) at (1.5,-0.5) {$(m',s')$};
\path[->,font=\scriptsize,>=angle 90]
(A) edge node[above]{$(i,!)$} (B)
(C) edge node[above]{$(o,!)$} (B)
(A) edge node[left]{$L(f)$} (A')
(C) edge node[right]{$L(g)$} (C')
(A') edge node[below] {$(i',!)$} (D)
(C') edge node[below] {$(o',!)$} (D)
(B) edge node [left] {$(h,\tau)$} (D);
\end{tikzpicture}
\]

\section{Bicategorical and categorical aspects}
\label{spinoffs}

While double categories are a natural context for studying cospans, bicategories are more
familiar---and of course, \emph{categories} are even more so!   Luckily, all our results
phrased in the language of double categories have analogues for bicategories and categories.
We explain those here.

As discussed for example by Shulman \cite{Shulman2010}, any double category $\lD$ has a
\define{horizontal bicategory}, denoted $\bD$, in which:
\begin{itemize}
\item objects are objects of $\lD$,
\item morphisms are horizontal 1-cells of $\lD$,
\item 2-morphisms are \define{globular} 2-morphisms of $\lD$, meaning 2-morphisms whose source and target vertical 1-morphisms are identities,
\item composition of morphisms is given by composition of horizontal 1-cells in $\lD$,
\item vertical and horizontal composition of 2-morphisms are given by vertical and horizontal
composition of 2-morphisms in $\lD$.
\end{itemize}
The bicategory $\bD$ has a \define{decategorification}, a category $\D$ in which:
\begin{itemize}
\item objects are objects of $\bD$,
\item morphisms are isomorphism classes of morphisms of $\bD$.
\end{itemize}
Thus, the double category $F\lCsp$ of structured cospans constructed in \cref{thm:decorated_cospans_1} automatically gives rise to a bicategory $F\bCsp$, and a category $F\Csp$.   In \cref{thm:decorated_cospans_2} we gave conditions under which the double category $F\lCsp$ becomes symmetric monoidal.   We would like the bicategory $F\bCsp$ and the category $F\Csp$ to become symmetric monoidal under the same conditions, and indeed this is true.

A double category is `fibrant' if every vertical 1-morphism has a `companion' and a `conjoint'---concepts explained in \cref{def:companion}. Shulman (\cref{Shulhorizontalbicat}) proved that when a double category $\lD$ is fibrant, any symmetrical monoidal structure on $\lD$ gives one on $\bD$.    We can apply this to decorated cospans as follows:

\begin{lem}
Given a category $\A$ with finite colimits and a lax monoidal pseudofunctor $(F,\phi,\phi_0) \maps (\A,+)\to(\Cat,\times)$, the double category $F\lCsp$ is fibrant.
\end{lem}

\begin{proof}
We show that any vertical 1-morphism $f \maps a \to b$ in $F\lCsp$ has a companion and a conjoint.  First, we can make this horizontal 1-cell $\hat{f}$:
\[
\begin{tikzpicture}[scale=1.5]
\node (A) at (0,0) {$a$};
\node (B) at (1,0) {$b$};
\node (C) at (2,0) {$b$};
\node (D) at (4,0) {$I_b \in F(b),$};
\path[->,font=\scriptsize,>=angle 90]
(A) edge node[above]{$f$} (B)
(C) edge node[above]{$1$} (B);
\end{tikzpicture}
\]
where $I_b$ is the trivial decoration given by
\[    \one \xrightarrow{\phi_0} F(0_\A) \xrightarrow{F(!_b)} F(b), \]
into a companion of $f$ using the following 2-morphisms:
\[
\begin{tikzpicture}[scale=1.5]
\node (A) at (0,0.5) {$a$};
\node (A') at (0,-0.5) {$b$};
\node (B) at (1,0.5) {$b$};
\node (C) at (2,0.5) {$b$};
\node (C') at (2,-0.5) {$b$};
\node (D) at (1,-0.5) {$b$};
\node (E) at (3,0.5) {$I_b \in F(b)$};
\node (F) at (3,-0.5) {$I_b \in F(b)$};
\node (G) at (5,0.5) {$a$};
\node (H) at (6,0.5) {$a$};
\node (I) at (7,0.5) {$a$};
\node (G') at (5,-0.5) {$a$};
\node (H') at (6,-0.5) {$b$};
\node (I') at (7,-0.5) {$b$};
\node (J) at (8,0.5) {$I_a \in F(a)$};
\node (K) at (8,-0.5) {$I_b \in F(b)$};
\node (L) at (1,-1) {$\tau_{1_b} = 1_{I_b}$};
\node (M) at (6,-1) {$\tau_f \maps F(f)(I_a) \to I_b$};
\path[->,font=\scriptsize,>=angle 90]
(A) edge node[above]{$f$} (B)
(C) edge node[above]{$1$} (B)
(A) edge node[left]{$f$} (A')
(C) edge node[right]{$1$} (C')
(A') edge node[below] {$1$} (D)
(C') edge node[below] {$1$} (D)
(B) edge node [left] {$1$} (D)
(G) edge node [above] {$1$} (H)
(G) edge node [left] {$1$} (G')
(H) edge node [left] {$f$} (H')
(G') edge node [below] {$f$} (H')
(I) edge node [above] {$1$} (H)
(I) edge node [right] {$f$} (I')
(I') edge node [below] {$1$} (H');
\end{tikzpicture}
\]
where the decoration morphism $\tau_f$ is the isomorphism given by pseudofunctoriality of $F$:
\[
\begin{tikzpicture}[scale=1.5]
\node (A) at (0,0) {$\one$};
\node (B) at (1,0) {$F(0)$};
\node (C) at (2.2,0.5) {$F(a)$};
\node (C') at (2.2,-0.5) {$F(b)$};
\node (D) at (1.8,0) {$\cong$};
\path[->,font=\scriptsize,>=angle 90]
(A) edge node[above]{$\phi_0$} (B)
(B) edge node[above]{$F(!_a)\;$} (C)
(B) edge node[below]{$F(!_b)\;\;$} (C')
(C) edge node[right]{$F(f)$} (C');
\end{tikzpicture}
\]
These 2-morphisms satisfy the equations \cref{eq:CompanionEq} required of a companion, involving vertical and horizontal composition of 2-morphisms in this double category:
\[
\begin{tikzpicture}[scale=1.5]
\node (N) at (0,1.5) {$a$};
\node (O) at (1,1.5) {$a$};
\node (P) at (2,1.5) {$a$};
\node (Q) at (-1,1.5) {$I_a \in F(a)$};
\node (A) at (0,0.5) {$a$};
\node (A') at (0,-0.5) {$b$};
\node (B) at (1,0.5) {$b$};
\node (C) at (2,0.5) {$b$};
\node (C') at (2,-0.5) {$b$};
\node (D) at (1,-0.5) {$b$};
\node (E) at (-1,0.5) {$I_b \in F(b)$};
\node (F) at (-1,-0.5) {$I_b \in F(b)$};
\node (G) at (4,1) {$a$};
\node (H) at (5,1) {$a$};
\node (I) at (6,1) {$a$};
\node (G') at (4,0) {$b$};
\node (H') at (5,0) {$b$};
\node (I') at (6,0) {$b$};
\node (J) at (7,1) {$I_a \in F(a)$};
\node (K) at (7,0) {$I_b \in F(b)$};
\node (Q) at (1,-1) {$\tau_f \maps F(f)(I_a) \to I_b$};
\node (L) at (1,-1.5) {$\tau_{1_b} = 1_{I_b}$};
\node (M) at (5,-0.5) {$\tau_f \maps F(f)(I_a) \to I_b$};
\node (R) at (3,0.5) {$=$};
\path[->,font=\scriptsize,>=angle 90]
(N) edge node[above]{$1$} (O)
(P) edge node[above]{$1$} (O)
(N) edge node[left]{$1$} (A)
(O) edge node[left]{$f$} (B)
(P) edge node[right]{$f$} (C)
(A) edge node[above]{$f$} (B)
(C) edge node[above]{$1$} (B)
(A) edge node[left]{$f$} (A')
(C) edge node[right]{$1$} (C')
(A') edge node[below] {$1$} (D)
(C') edge node[below] {$1$} (D)
(B) edge node [left] {$1$} (D)
(G) edge node [above] {$1$} (H)
(G) edge node [left] {$f$} (G')
(H) edge node [left] {$f$} (H')
(G') edge node [below] {$1$} (H')
(I) edge node [above] {$1$} (H)
(I) edge node [right] {$f$} (I')
(I') edge node [below] {$1$} (H');
\end{tikzpicture}
\]

\[
\begin{tikzpicture}[scale=1.5]
\node (G) at (-0.75,0.5) {$a$};
\node (H) at (-0.75,-0.5)  {$b$};
\node (I) at (-1.5,0.5) {$a$};
\node (J) at (-1.5,-0.5) {$a$};
\node (A) at (0,0.5) {$a$};
\node (A') at (0,-0.5) {$b$};
\node (B) at (0.75,0.5) {$b$};
\node (C) at (1.5,0.5) {$b$};
\node (C') at (1.5,-0.5) {$b$};
\node (D) at (0.75,-0.5) {$b$};
\node (E) at (0.75,1) {$\scriptstyle{I_b \in F(b)}$};
\node (F) at (2,-0.5) {$\scriptstyle{I_b \in F(b)}$};
\node (L) at (2,0) {$\scriptstyle{\tau_{b} = 1_{I_b}}$};
\node (E') at (-0.75,1) {$\scriptstyle{I_a \in F(a)}$};
\node (F') at (-2.25,-0.5) {$\scriptstyle{I_b \in F(b)}$};
\node (L') at (-2.5,0) {$\scriptstyle{\tau_{f} \maps F(f)(I_a) \to I_b}$};
\node (M) at (2.5,-0.5) {$\scriptstyle{=}$};
\node (N) at (4.5,0) {$a$};
\node (O) at (4.5,-1) {$b$};
\node (P) at (5.25,0) {$b$};
\node (R) at (6,0) {$b$};
\node (S) at (6,-1) {$b$};
\node (T) at (5.25,0.5) {$\scriptstyle{I_b \in F(b)}$};
\node (T') at (3.75,0.5) {$\scriptstyle{I_a \in F(a)}$};
\node (U) at (4.5,-1.5) {$\scriptstyle{I_b \in F(b)}$};
\node (V) at (4.5,-1.75) {$\scriptstyle{\tau_{\rho_{\hat{f}}} \maps (I_b \odot I_a) \to I_b}$};
\node (K) at (-1.5,-1.5) {$a$};
\node (L') at (0,-1.5) {$b$};
\node (M') at (1.5,-1.5) {$b$};
\node (N') at (0,-2) {$\scriptstyle{I_b \in F(b)}$};
\node (O') at (0,-2.25) {$\scriptstyle{\tau_{\lambda_{\hat{f}}} \maps (I_b \odot I_b) \to I_b}$};
\node (W) at (3.75,0) {$a$};
\node (Y) at (3,0) {$a$};
\node (Z) at (3,-1) {$a$};
\path[->,font=\scriptsize,>=angle 90]
(Y) edge node [left] {$1$} (Z)
(Y) edge node [above] {$1$} (W)
(Z) edge node [below] {$f$} (O)
(N) edge node [above] {$1$} (W)
(J) edge node [left] {$1$} (K)
(A') edge node [left] {$1$} (L')
(C') edge node [right] {$1$} (M')
(K) edge node[below] {$f$} (L')
(M') edge node [below] {$1$} (L')
(N) edge node[left]{$f$} (O)
(R) edge node[right]{$1$} (S)
(N) edge node[above]{$f$} (P)
(R) edge node[above]{$1$} (P)
(S) edge node[below]{$1$} (O)
(A) edge node[above]{$f$} (B)
(C) edge node[above]{$1$} (B)
(A) edge node[left]{$f$} (A')
(C) edge node[right]{$1$} (C')
(A') edge node[above] {$1$} (D)
(C') edge node[above] {$1$} (D)
(B) edge node [left] {$1$} (D)
(A) edge node[above]{$1$} (G)
(G) edge node[left]{$f$} (H)
(A') edge node[above]{$1$} (H)
(J) edge node[above] {$f$} (H)
(I) edge node[left] {$1$} (J)
(I) edge node [above] {$1$} (G);
\end{tikzpicture}
\]
Note that the right hand side of the first equation is $U_f$, while the second equation involves the left and right unitors for $\odot$: these are maps from a horizontal composite of two decorated cospans to a single decorated cospan.  The conjoint of $f$ is given by this horizontal 1-cell $\check{f}$, which is just the opposite of the companion above:
\[
\begin{tikzpicture}[scale=1.5]
\node (A) at (0,0) {$b$};
\node (B) at (1,0) {$b$};
\node (C) at (2,0) {$a$};
\node (D) at (4,0) {$I_b \in F(b)$.};
\path[->,font=\scriptsize,>=angle 90]
(A) edge node[above]{$1$} (B)
(C) edge node[above]{$f$} (B);
\end{tikzpicture}
\]
Just as $\hat{f}$ obeys the equations required of a companion, $\check{f}$ obeys the equations required of a conjoint with similar structure 2-morphisms to those of a companion above.
\end{proof}

\begin{thm}
\label{thm:bicat}
Let $\A$ be a category with finite colimits and $F \maps (\A, +) \to (\Cat,\times)$ a symmetric lax monoidal pseudofunctor. Then there exists a symmetric monoidal bicategory $F \mathbf{Csp}$ in which:
\begin{enumerate}
\item objects are those of $\A$,
\item morphisms are $F$-decorated cospans:
\[
\begin{tikzpicture}[scale=1.5]
\node (A) at (0,0) {$a$};
\node (B) at (1,0) {$m$};
\node (C) at (2,0) {$b$};
\node (D) at (4,0) {$s \in F(m)$,};
\path[->,font=\scriptsize,>=angle 90]
(A) edge node[above]{$i$} (B)
(C) edge node[above]{$o$} (B);
\end{tikzpicture}
\]
\item a 2-morphism is a map of cospans in $\A$
\[
\begin{tikzpicture}[scale=1.5]
\node (A) at (0,0) {$a$};
\node (B) at (1,0.5) {$m$};
\node (C) at (2,0) {$b$};
\node (E) at (1,-0.5) {$m'$};
\node (D) at (3,0.5) {$s \in F(m)$};
\node (F) at (3,-0.5) {$s' \in F(m')$};
\path[->,font=\scriptsize,>=angle 90]
(A) edge node[above]{$i$} (B)
(C) edge node[above]{$o$} (B)
(A) edge node[below]{$i'$} (E)
(B) edge node[left]{$h$} (E)
(C) edge node[below]{$o'$} (E);
\end{tikzpicture}
\]
together with a morphism $\tau \maps F(h)(s) \to s'$ in $F(m')$.
\end{enumerate}
\end{thm}

\begin{proof}
This follows by applying \cref{Shulhorizontalbicat} to the fibrant symmetric monoidal double category $F\lCsp$.
\end{proof}

This symmetric monoidal bicategory $F\bCsp$ generalizes one constructed by the second author \cite{Courser}.    We can decategorify $F\bCsp$ to obtain a symmetric monoidal category of
decorated cospans generalizing those considered by Fong \cite{Fong}:

\begin{cor}
Let $\A$ be a category with finite colimits and $F \maps (\A, +) \to (\Cat, \times)$ a symmetric lax monoidal pseudofunctor.  Then there exists a symmetric monoidal category $F\Csp$ in which:
\begin{enumerate}
\item objects are those of $\A$
\item morphisms are isomorphism classes of $F$-decorated cospans of $\A$, where two
$F$-decorated cospans
\[
\begin{tikzpicture}[scale=1.5]
\node (A) at (0,0) {$a$};
\node (B) at (1,0) {$m$};
\node (C) at (2,0) {$b$};
\node (D) at (4,0) {$s \in F(m)$};
\path[->,font=\scriptsize,>=angle 90]
(A) edge node[above]{$i$} (B)
(C) edge node[above]{$o$} (B);
\end{tikzpicture}
\]
\[
\begin{tikzpicture}[scale=1.5]
\node (A) at (0,0) {$a$};
\node (B) at (1,0) {$m'$};
\node (C) at (2,0) {$b$};
\node (D) at (4,0) {$s' \in F(m')$};
\path[->,font=\scriptsize,>=angle 90]
(A) edge node[below]{$i'$} (B)
(C) edge node[below]{$o'$} (B);
\end{tikzpicture}
\]
are isomorphic if and only if there exists an isomorphism $f \maps m \to m'$ in $\A$ such that following diagram commutes:
\[
\begin{tikzpicture}[scale=1.5]
\node (A) at (0,0) {$a$};
\node (B') at (1,0.5) {$m$};
\node (B) at (1,-0.5) {$m'$};
\node (C) at (2,0) {$b$};
\path[->,font=\scriptsize,>=angle 90]
(A) edge node[below]{$i'$} (B)
(C) edge node[below]{$o'$} (B)
(A) edge node[above]{$i$} (B')
(C) edge node[above]{$o$} (B')
(B') edge node[left]{$f$} (B);
\end{tikzpicture}
\]
and there exists an isomorphism $\tau \maps F(f)(s) \to s'$ in $F(m')$.
\end{enumerate}
\end{cor}

In \cref{thm:equiv} we gave conditions under which the symmetric monoidal double category of \emph{decorated} cospans $F\lCsp$ is isomorphic to the  symmetric monoidal double category of \emph{structured} cospans ${}_L \lCsp(\inta F)$.   We now show that under the same conditions we get an isomorphism of symmetric monoidal bicategories, and of categories.

\begin{thm} \label{thm:bicat_equiv}
Suppose $\A$ has finite colimits and $F \maps(\A,+) \to (\Cat,\times)$ is a symmetric lax monoidal pseudofunctor that factors through $\Rex$ as an ordinary pseudofunctor.    Define the symmetric monoidal bicategory $_L\bCsp(\inta F)$ as in \cref{thm:equiv}.   Then there is an isomorphism of symmetric monoidal bicategories
\[      F\bCsp \cong {}_L \bCsp(\inta F)   \]
and of symmetric monoidal categories
\[      F\Csp \cong {}_L \Csp(\inta F)  . \]
\end{thm}

\begin{proof}
Hansen and Shulman \cite{HS} showed that the passage from symmetric monoidal double categories to symmetric monoidal bicategories is functorial in a suitable sense.  This implies that an isomorphism of symmetric monoidal double categories $\lD \cong \lD'$ gives an isomorphism of symmetric monoidal bicategories $\bD \cong \bD'$.    Since the process of decategorifying a bicategory merely discards 2-morphisms and takes isomorphism classes of 1-morphisms, the isomorphism of symmetric monoidal bicategories $\bD \cong \bD'$ in turn induces an isomorphism of symmetric monoidal categories $\D \cong \D'$.   Thus, the theorem follows from \cref{thm:equiv}. \end{proof}

\section{Applications}\label{Applications}

Thinking about systems and processes categorically dates back to early works by Lawvere \cite{Lawvere}, Bunge--Fiore \cite{BungeFiore}, Joyal--Nielsen--Winskel \cite{JNW}, Katis--Sabadini--Walters \cite{KSW} and others.   Spivak and others have used wiring diagrams and sheaves to capture compositional features of dynamical systems \cite{BFV,SSV,VSL}.  Another approach uses signal flow diagrams and other string diagrams \cite{BE,BSZ,FRS} to understand systems behaviorally, following ideas of Willems \cite{Willems}.

Decorated cospans were introduced by Fong \cite{Fong,FongThesis} to describe open systems as cospans equipped with extra data.  They were then applied to open electrical circuits \cite{BF}, Markov processes \cite{BFP}, and chemical reaction networks \cite{BP}.  Unfortunately, some of these applications were marred by technical flaws, which were later fixed using structured cospans \cite{BC}. Here we explain how they can also be fixed using our new decorated cospans.  We compare the two approaches in applications to graphs, electrical circuits, Petri nets, reaction networks and dynamical systems.

In some cases, \cref{thm:equiv} shows that the structured and decorated cospan approaches are equivalent: \cref{subsec:graphs,subsec:circuits,subsec:petri} illustrate this.  However, in some cases decorated cospans appear to be necessary, and in \cref{subsec:petrirates} we explain why \cref{thm:equiv} does not apply to open dynamical systems.

\subsection{Graphs}
\label{subsec:graphs}

One of the simplest kinds of network is a graph.  For us a \define{graph} will be a pair of functions $s,t\maps E \to N$ where $E$ and $N$ are finite
sets.   We call elements of $E$ \define{edges} and elements of $N$ \define{nodes}.  There is a category $\Graph$ where the objects are graphs and a
morphism from the graph $s,t\maps E \to N$ to the graph $s',t' \maps E' \to N'$ is a pair of functions $f \maps N \to N', g \maps E \to E'$ such that
these diagrams commute:
\[
\begin{tikzpicture}[scale=1.5]
\node (A) at (0,0) {$E$};
\node (A') at (0,-1) {$E'$};
\node (B) at (1,0) {$N$};
\node (B') at (1,-1) {$N'$};
\path[->,font=\scriptsize,>=angle 90]
(A) edge node[above]{$s$} (B)
(A') edge node[below]{$s'$} (B')
(A) edge node[left]{$g$} (A')
(B) edge node[right]{$f$} (B');

\node (C) at (2,0) {$E$};
\node (C') at (2,-1) {$E'$};
\node (D) at (3,0) {$N$};
\node (D') at (3,-1) {$N'$.};
\path[->,font=\scriptsize,>=angle 90]
(C) edge node[above]{$t$} (D)
(C') edge node[below]{$t'$} (D')
(C) edge node[left]{$g$} (C')
(D) edge node[right]{$f$} (D');
\end{tikzpicture}
\]

We can easily build a double category with `open graphs' as horizontal 1-cells using the machinery of structured cospans \cite[Section 5]{BC}.  Let $L \maps \Fin\Set \to \Graph$ be the functor that assigns to a finite set $N$ the \define{discrete graph} on $N$: the graph with no edges and $N$ as its set of vertices. Both $\Fin\Set$ and $\Graph$ have finite colimits, and the functor $L \maps \Fin\Set \to \Graph$ is left adjoint to the forgetful functor $R \maps \Graph \to \Fin\Set$ that assigns to a graph $G$ its underlying set of vertices $R(G)$. Thus, using structured cospans and appealing to \cref{thm:structured_cospans}, we get a symmetric monoidal double category $_L \lCsp(\Graph)$ in which:
\begin{itemize}
\item objects are finite sets,
\item a vertical 1-morphism from $X$ to $Y$ is a function $f \maps X \to Y$,
\item a horizontal 1-cell from $X$ to $Y$ is an \define{open graph} from $X$ to $Y$, meaning a cospan in $\Graph$ of this form:
\[
\begin{tikzpicture}[scale=1.5]
\node (A) at (0,0) {$L(X)$};
\node (B) at (1,0) {$G$};
\node (C) at (2,0) {$L(Y)$,};
\path[->,font=\scriptsize,>=angle 90]
(A) edge node[above]{$i$} (B)
(C) edge node[above]{$o$} (B);
\end{tikzpicture}
\]
\item a 2-morphism is a commuting diagram in $\Graph$ of this form:
\[
\begin{tikzpicture}[scale=1.5]
\node (A) at (0,0) {$L(X)$};
\node (B) at (1,0) {$G$};
\node (C) at (2,0) {$L(Y)$};
\node (A') at (0,-1) {$L(X')$};
\node (B') at (1,-1) {$G'$};
\node (C') at (2,-1) {$L(Y')$.};
\path[->,font=\scriptsize,>=angle 90]
(A) edge node[above]{$i$} (B)
(C) edge node[above]{$o$} (B)
(A') edge node[below]{$i'$} (B')
(C') edge node[below]{$o'$} (B')
(A) edge node [left]{$L(f)$} (A')
(B) edge node [left]{$\alpha$} (B')
(C) edge node [right]{$L(g)$} (C');
\end{tikzpicture}
\]
\end{itemize}
Here is an example of an open graph:
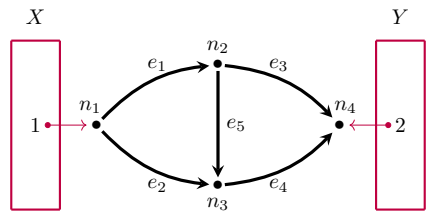
\[
\scalebox{0.8}{
\begin{tikzpicture}
	\begin{pgfonlayer}{nodelayer}
		\node [contact] (n1) at (-2,0) {$\bullet$};
		\node [style = none] at (-2.1,0.3) {$n_1$};
		\node [contact] (n2) at (0,1) {$\bullet$};
		\node [style = none] at (0,1.3) {$n_2$};
		\node [contact] (n3) at (0,-1) {$\bullet$};
		\node [style = none] at (0,-1.3) {$n_3$};
		\node [contact] (n4) at (2,0) {$\bullet$};
		\node [style = none] at (2.1,0.3) {$n_4$};

		\node [style = none] at (-1,1) {$e_1$};
		\node [style = none] at (-1,-1) {$e_2$};
		\node [style = none] at (1,1) {$e_3$};
		\node [style = none] at (1,-1) {$e_4$};
	    \node [style = none] at (0.3,0) {$e_5$};

		\node [style=none] (1) at (-3,0) {1};
		\node [style=none] (4) at (3,0) {2};

		\node [style=none] (ATL) at (-3.4,1.4) {};
		\node [style=none] (ATR) at (-2.6,1.4) {};
		\node [style=none] (ABR) at (-2.6,-1.4) {};
		\node [style=none] (ABL) at (-3.4,-1.4) {};

		\node [style=none] (X) at (-3,1.8) {$X$};
		\node [style=inputdot] (inI) at (-2.8,0) {};

		\node [style=none] (Z) at (3,1.8) {$Y$};
	 \node [style=inputdot] (outI') at (2.8,0) {};

		\node [style=none] (MTL) at (2.6,1.4) {};
		\node [style=none] (MTR) at (3.4,1.4) {};
		\node [style=none] (MBR) at (3.4,-1.4) {};
		\node [style=none] (MBL) at (2.6,-1.4) {};

	\end{pgfonlayer}
	\begin{pgfonlayer}{edgelayer}
		\draw [style=inarrow, bend left=20, looseness=1.00] (n1) to (n2);
		\draw [style=inarrow, bend right=20, looseness=1.00] (n1) to (n3);
		\draw [style=inarrow, bend left=20, looseness=1.00] (n2) to (n4);
		\draw [style=inarrow, bend right=20, looseness=1.00] (n3) to (n4);
		\draw [style=inarrow] (n2) to (n3);
		\draw [style=simple] (ATL.center) to (ATR.center);
		\draw [style=simple] (ATR.center) to (ABR.center);
		\draw [style=simple] (ABR.center) to (ABL.center);
		\draw [style=simple] (ABL.center) to (ATL.center);
		\draw [style=simple] (MTL.center) to (MTR.center);
		\draw [style=simple] (MTR.center) to (MBR.center);
		\draw [style=simple] (MBR.center) to (MBL.center);
		\draw [style=simple] (MBL.center) to (MTL.center);
		\draw [style=inputarrow] (inI) to (n1);
		\draw [style=inputarrow] (outI') to (n4);
	\end{pgfonlayer}
\end{tikzpicture}
}
\]

We can also build a double category with open graphs as horizontal 1-cells using decorated cospans.    For any finite set $N$, there is a category $F(N)$ where:
\begin{itemize}
\item an object is a \define{graph structure} on $N$: that is, a graph $s,t \maps E \to N$,
\item a morphism from $s,t \maps E \to N$ to $s',t' \maps E' \to N$ is a morphism of graphs that is the identity on $N$: that is, a function $g \maps E \to E'$ such that these diagrams commute:
\[
\begin{tikzpicture}[scale=1.5]
\node (A) at (0,0) {$E$};
\node (A') at (0,-1) {$E'$};
\node (B) at (1,-0.5) {$N$};
\path[->,font=\scriptsize,>=angle 90]
(A) edge node[above]{$s$} (B)
(A') edge node[below]{$s'$} (B)
(A) edge node[left]{$g$} (A');
\node (C) at (2,0) {$E$};
\node (C') at (2,-1) {$E'$};
\node (D) at (3,-0.5) {$N$.};
\path[->,font=\scriptsize,>=angle 90]
(C) edge node[above]{$t$} (D)
(C') edge node[below]{$t'$} (D)
(C) edge node[left]{$g$} (C');
\end{tikzpicture}
\]
\end{itemize}

In general, decorated cospans involve a pseudofunctor to $\Cat$, but in this example there is actually an honest functor $F \maps \Set \to \Cat$ that assigns to a set $N$ the above category $F(N)$.   Given a function $f \maps M \to N$, we define $F(f) \maps F(M) \to F(N)$ as the functor that maps any graph structure $s,t \maps E \to M$ to the graph structure $f s, f t \maps E \to N$.

We can make $F$ into a symmetric lax monoidal pseudofunctor $F \maps (\Fin\Set, +) \to (\Cat,\times)$ by equipping it with suitable functors
\[   \phi_{N,N'} \maps F(N) \times F(N') \to F(N+N'), \qquad \phi_0 \maps  \one \to F(\emptyset) . \]

The functor $\phi_0$ is uniquely determined since $F(\emptyset)$ is the terminal category.   More interesting is $\phi_{N,N'}$.   This functor maps a pair of graph structures $s, t \maps E \to N$ and $s',t' \maps E' \to N'$ to the graph structure $s+s', t+t' \maps E+E' \to N+N'$.  In other words, it sends a pair of graph structures to their `disjoint union'.   Surprisingly, though $F$ is a functor, this choice of $\phi_{N,N'}$ does not make $F$ into a symmetric lax monoidal functor, but only a symmetric lax monoidal pseudofunctor, since it obeys the required laws only up to natural isomorphism, as in \cref{eq:omega}.   See \cite[Section 5]{BC} for a proof that these laws fail to hold on the nose.   This fact is what necessitated a generalization of Fong's original approach to decorated cospans.

It is well known, and easy to check, that the Grothendieck category $\inta F$ is isomorphic to the category $\Graph$.   The other side of this observation is that the opfibration $U \maps \inta F \to \Fin\Set$ is isomorphic to the forgetful functor $R \maps \Graph \to \Fin\Set$.    In fact one can check that $U \maps \inta F \to \Fin\Set$ and $R \maps \Graph \to \Fin\Set$ are isomorphic as symmetric monoidal opfibrations, where all the categories involved are given cocartesian monoidal structures.

Starting from the symmetric lax monoidal pseudofunctor $F \maps (\Fin\Set, +) \to (\Cat,\times)$, \cref{thm:decorated_cospans_2} gives us a symmetric monoidal double category $F\lCsp$ in which:
\begin{itemize}
\item objects are finite sets,
\item a vertical 1-morphism from $X$ to $X'$ is a function $f \maps X \to X'$,
\item a horizontal 1-cell from $X$ to $Y$ is a pair
\[
\begin{tikzpicture}[scale=1.5]
\node (A) at (0,0) {$X$};
\node (B) at (1,0) {$N$};
\node (C) at (2,0) {$Y$};
\node (D) at (3.25,0) {$G \in F(N)$};
\path[->,font=\scriptsize,>=angle 90]
(A) edge node[above]{$i$} (B)
(C) edge node[above]{$o$} (B);
\end{tikzpicture}
\]
which can also be thought of as an open graph from $X$ to $Y$,
\item a 2-morphism
\[
\begin{tikzpicture}[scale=1.5]
\node (A) at (0,0) {$X$};
\node (B) at (1,0) {$N$};
\node (C) at (2,0) {$Y$};
\node (A') at (0,-1) {$X'$};
\node (C') at (2,-1) {$Y'$};
\node (D) at (1,-1) {$N'$};
\node (E) at (3,0) {$G \in F(N)$};
\node (F) at (3,-1) {$G' \in F(N')$};
\path[->,font=\scriptsize,>=angle 90]
(A) edge node[above]{$i$} (B)
(C) edge node[above]{$o$} (B)
(A) edge node[left]{$f$} (A')
(C) edge node[right]{$g$} (C')
(C') edge node [below] {$o'$} (D)
(A') edge node [below] {$i'$} (D)
(B) edge node [left] {$h$} (D);
\end{tikzpicture}
\]
is a commuting diagram in $\Fin\Set$ together with a morphism $\tau \maps F(h)(G) \to G'$ in $F(N')$.
\end{itemize}

We thus have two symmetric monoidal double categories: ${}_L \lCsp(\Graph)$ obtained from structured cospans and $F\lCsp$ obtained from decorated cospans. Each of these double categories has $\Fin\Set$ as its category of objects, open graphs as horizontal 1-cells, and maps of open graphs as 2-morphisms.   This suggests that ${}_L \lCsp(\Graph)$  and $F\lCsp$ are isomorphic as symmetric monoidal double categories---and indeed this follows from \cref{thm:equiv}.

\subsection{Circuits}
\label{subsec:circuits}

Structured and decorated cospans are a powerful tool for studying categories where the morphisms are electrical circuits---see \cite[Section 6.1]{BC} and \cite{BCR,BF}.  The key idea is to use open graphs with labeled edges to describe circuits, where the labels can stand for resistors with any chosen resistance, capacitors with any chosen capacitance, or other circuit elements.   The whole theory of open graphs discussed in the previous section can be recapitulated for labeled graphs.  Since the abstract formalism works the same way, we can be brief.   Concrete applications of this formalism are discussed in the above references, and in \cite{BFP} a class of Markov processes were
also handled using this formalism, by reducing them to circuits of resistors.

Fix a set $\La$ to serve as edge labels.  Define an $\La$-\define{graph} to be a graph $s,t\maps E\to N$ equipped with a function $\ell \maps E \to \La$.  There is a category $\Graph_\La$ where the objects are $\La$-graphs and a morphism from the $\La$-graph
 \[ \xymatrix{\La & E \ar@<2.5pt>[r]^{s} \ar@<-2.5pt>[r]_{t} \ar[l]_{\ell} & N} \]
 to the $\La$-graph
\[ \xymatrix{\La & E' \ar@<2.5pt>[r]^{s'} \ar@<-2.5pt>[r]_{t'} \ar[l]_{\ell'} & N'} \]
is a pair of functions $f \maps N \to N', g \maps E \to E'$ such that these diagrams commute:
\[
\begin{tikzpicture}[scale=1.5]
\node (A) at (0,0) {$E$};
\node (A') at (0,-1) {$E'$};
\node (B) at (1,0) {$N$};
\node (B') at (1,-1) {$N'$};
\path[->,font=\scriptsize,>=angle 90]
(A) edge node[above]{$s$} (B)
(A') edge node[below]{$s'$} (B')
(A) edge node[left]{$g$} (A')
(B) edge node[right]{$f$} (B');

\node (C) at (2,0) {$E$};
\node (C') at (2,-1) {$E'$};
\node (D) at (3,0) {$N$};
\node (D') at (3,-1) {$N'$};
\path[->,font=\scriptsize,>=angle 90]
(C) edge node[above]{$t$} (D)
(C') edge node[below]{$t'$} (D')
(C) edge node[left]{$g$} (C')
(D) edge node[right]{$f$} (D');

\node (E) at (4,-0.5) {$\La$};
\node (G) at (5,0) {$E$};
\node (G') at (5,-1) {$E'$.};
\path[->,font=\scriptsize,>=angle 90]
(G) edge node[above]{$\ell$} (E)
(G) edge node[right]{$g$} (G')
(G') edge node[below]{$\ell'$} (E);
\end{tikzpicture}
\]
There is a functor $U \maps \Graph_\La \to \Fin\Set$ that takes an $\La$-graph to its underlying set of nodes. This has a left adjoint $L \maps \Fin\Set \to \Graph_\La$ sending any set to the $\La$-graph with that set of nodes and no edges.  Both $\Fin\Set$ and $\Graph_\La$ have colimits, and $L$ preserves them.

This sets the stage for structured cospans: \cref{thm:structured_cospans} gives us a symmetric monoidal double category ${}_L \lCsp(\Graph_\La)$ where a horizontal 1-cell is an \define{open} $\La$-\define{graph}, also called an $\La$-\define{circuit}: that is, a cospan in $\Graph_\La$ of this form:
\[
\begin{tikzpicture}[scale=1.5]
\node (A) at (0,0) {$L(X)$};
\node (B) at (1,0) {$G$};
\node (C) at (2,0) {$L(Y)$.};
\path[->,font=\scriptsize,>=angle 90]
(A) edge node[above]{$i$} (B)
(C) edge node[above]{$o$} (B);
\end{tikzpicture}
\]
For example, here is a $\La$-circuit with $\La = (0,\infty)$:
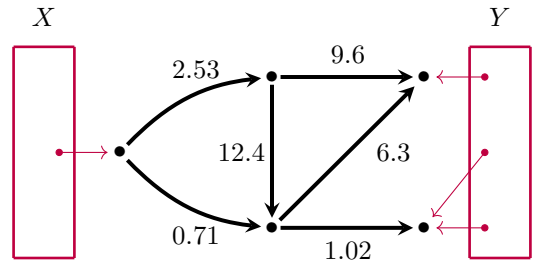
\[
\scalebox{1}{
\begin{tikzpicture}
	\begin{pgfonlayer}{nodelayer}
		\node [contact] (n1) at (-2,0) {$\bullet$};
		\node [style = none] at (-2.1,0.3) {$$};
		\node [contact] (n2) at (0,1) {$\bullet$};
		\node [style = none] at (0,1.3) {$$};
		\node [contact] (n3) at (0,-1) {$\bullet$};
		\node [style = none] at (0,-1.3) {$$};
		\node [contact] (n4) at (2,1) {$\bullet$};
		\node [style = none] at (2.1,0.3) {$$};
		\node [contact] (n5) at (2,-1) {$\bullet$};
		\node [style = none] at (2.1,0.3) {$$};

		\node [style = none] at (-1,1.1) {$2.53$};
		\node [style = none] at (-1,-1.1) {$0.71$};
		\node [style = none] at (1,1.3) {$9.6$};
		\node [style = none] at (1,-1.3) {$1.02$};
	     \node [style = none] at (-0.4,0) {$12.4$};
	     \node [style = none] at (1.6,0) {$6.3$};

		\node [style=none] (1) at (-3,0) {};
		\node [style=none] (4) at (3,0) {};

		\node [style=none] (ATL) at (-3.4,1.4) {};
		\node [style=none] (ATR) at (-2.6,1.4) {};
		\node [style=none] (ABR) at (-2.6,-1.4) {};
		\node [style=none] (ABL) at (-3.4,-1.4) {};

		\node [style=none] (X) at (-3,1.8) {$X$};
		\node [style=inputdot] (inI) at (-2.8,0) {};

		\node [style=none] (Z) at (3,1.8) {$Y$};
	 \node [style=inputdot] (outI') at (2.8,1) {};
	 \node [style=inputdot] (outI'') at (2.8,0) {};
	 \node [style=inputdot] (outI''') at (2.8,-1) {};

		\node [style=none] (MTL) at (2.6,1.4) {};
		\node [style=none] (MTR) at (3.4,1.4) {};
		\node [style=none] (MBR) at (3.4,-1.4) {};
		\node [style=none] (MBL) at (2.6,-1.4) {};

	\end{pgfonlayer}
	\begin{pgfonlayer}{edgelayer}
		\draw [style=inarrow, bend left=20, looseness=1.00] (n1) to (n2);
		\draw [style=inarrow, bend right=20, looseness=1.00] (n1) to (n3);
		\draw [style=inarrow, bend left=0, looseness=1.00] (n2) to (n4);
		\draw [style=inarrow, bend right=0, looseness=1.00] (n3) to (n4);
		\draw [style=inarrow, bend right=0, looseness=1.00] (n3) to (n5);
		\draw [style=inarrow] (n2) to (n3);
		\draw [style=simple] (ATL.center) to (ATR.center);
		\draw [style=simple] (ATR.center) to (ABR.center);
		\draw [style=simple] (ABR.center) to (ABL.center);
		\draw [style=simple] (ABL.center) to (ATL.center);
		\draw [style=simple] (MTL.center) to (MTR.center);
		\draw [style=simple] (MTR.center) to (MBR.center);
		\draw [style=simple] (MBR.center) to (MBL.center);
		\draw [style=simple] (MBL.center) to (MTL.center);
		\draw [style=inputarrow] (inI) to (n1);
		\draw [style=inputarrow] (outI') to (n4);
		\draw [style=inputarrow] (outI'') to (n5);
		\draw [style=inputarrow] (outI''') to (n5);
	\end{pgfonlayer}
\end{tikzpicture}
}
\]
The edges here represent wires, with the positive real numbers labeling them serving to describe the resistance of resistors on the wires.  The elements of the sets $X$ and $Y$ represent `terminals': that is, points where we allow ourselves to attach a wire from another circuit.

We can now also describe $\La$-circuits using our new approach to decorated cospans.   There is a symmetric lax monoidal pseudofunctor $F \maps (\Fin\Set, +) \to (\Cat, \times)$ such that for any finite set $N$, the category $F(N)$ has:
\begin{itemize}
\item objects being $\La$-\define{graph structures} on $N$: that is, $\La$-graphs where the set of nodes is $N$,
\item morphisms being morphisms of $\La$-graphs that are the identity on the set of nodes.
\end{itemize}
This gives a symmetric monoidal double category $F \lCsp$, and using \cref{thm:equiv} we can show that this is isomorphic, as a symmetric monoidal double category, to ${}_L \lCsp(\Graph_\La)$.

\subsection{Petri nets}
\label{subsec:petri}

Petri nets are widely used as models of systems in engineering and computer science \cite{GiraultValk, Peterson}.   Structured cospans have been used to define a symmetric monoidal double category of `open Petri nets' \cite{BM}, which lets us build large Petri nets out of smaller pieces.  We can also use decorated cospans to create a double category of open Petri nets.  Again this example is very similar to the example of open graphs.

A \define{Petri net} is a pair of finite sets $S$ and $T$ and functions $s,t \maps T \to \N[S]$.  Here $S$ is the set of \define{places}, $T$ is the set of \define{transitions}, and $\N[S]$ is the underlying set of the free commutative monoid on $S$.  Each transition thus has a formal sum of places as its source and target as prescribed by the functions $s$ and $t$, respectively.  Here is an example:
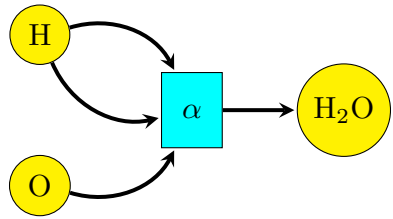
\[
\begin{tikzpicture}
	\begin{pgfonlayer}{nodelayer}
		\node [style=species] (I) at (0,1) {H};
		\node [style=species] (T) at (0,-1) {O};
		\node [style=transition] (W) at (2,0) {$\phantom{\Big|}\alpha$\phantom{\Big|}};
		\node [style=species] (Water) at (4,0) {$\textnormal{H}_2$O};
	\end{pgfonlayer}
	\begin{pgfonlayer}{edgelayer}
		\draw [style=inarrow, bend right=40, looseness=1.00] (I) to (W);
		\draw [style=inarrow, bend left=40, looseness=1.00] (I) to (W);
		\draw [style=inarrow, bend right=40, looseness=1.00] (T) to (W);
		\draw [style=inarrow] (W) to (Water);
	\end{pgfonlayer}
\end{tikzpicture}
\]
This Petri net has a single transition $\alpha$ with $2\textnormal{H}+\textnormal{O}$ as its source and $\textnormal{H}_2 \textnormal{O}$ as its target.

There is a category $\Petri$ with Petri nets as objects, where a morphism from the Petri net
$s, t \maps T \to \N[S]$ to the Petri net $s', t' \maps T' \to \N[S']$ is a pair of functions $f \maps S \to S', g \maps T \to T'$ such that the following diagrams commute:
	\[
	\xymatrix{
		T \ar[d]_g  \ar[r]^-{s} & \N[S] \ar[d]^-{\N[f]} \\
		T' \ar[r]_-{s'} & \N[S']
	}
	\qquad
	\xymatrix{
		T \ar[d]_g  \ar[r]^-{t} & \N[S] \ar[d]^-{\N[f]} \\
		T' \ar[r]_-{t'} & \N[S'] .
	}
	\]
There is a functor $R \maps \Petri \to \Fin\Set$ sending any Petri net to its set of places, and this has a left adjoint $L \maps \Fin\Set \to \Petri$ sending any finite set $S$ to the Petri net with $S$ as its set of places and no transitions \cite[Lemma 11]{BM}.   Since both $\Fin\Set$ and $\Petri$ have finite colimits and $L$ preserves them, \cref{thm:structured_cospans} yields a symmetric monoidal double category ${}_L \lCsp(\Petri)$ in which:
\begin{itemize}
\item objects are finite sets,
\item vertical 1-morphisms are functions,
\item horizontal 1-cells are \define{open Petri nets}, which are cospans in $\Petri$ of the form:
\[
\begin{tikzpicture}[scale=1.5]
\node (D) at (-3,0) {$L(X)$};
\node (E) at (-2,0) {$P$};
\node (F) at (-1,0) {$L(Y)$};
\path[->,font=\scriptsize,>=angle 90]
(D) edge node [above] {$i$} (E)
(F) edge node [above] {$o$} (E);
\end{tikzpicture}
\]
\item 2-morphisms are diagrams in $\Petri$ of the form:
\[
\begin{tikzpicture}[scale=1.5]
\node (E) at (3,0) {$L(X)$};
\node (F) at (5,0) {$L(Y)$};
\node (G) at (4,0) {$P$};
\node (E') at (3,-1) {$L(X')$};
\node (F') at (5,-1) {$L(Y')$.};
\node (G') at (4,-1) {$P'$};
\path[->,font=\scriptsize,>=angle 90]
(F) edge node[above]{$o$} (G)
(E) edge node[left]{$L(f)$} (E')
(F) edge node[right]{$L(g)$} (F')
(G) edge node[left]{$\alpha$} (G')
(E) edge node[above]{$i$} (G)
(E') edge node[below]{$i'$} (G')
(F') edge node[below]{$o'$} (G');
\end{tikzpicture}
\]
\end{itemize}

We can equivalently describe open Petri nets using decorated cospans.  This works very much like the previous examples.  There is a symmetric lax monoidal pseudofunctor $F \maps (\Fin\Set, +) \to (\Cat, \times)$ such that for any finite set $S$, the category $F(S)$ has:
\begin{itemize}
\item objects given by Petri nets whose set of places is $S$,
\item morphisms given by morphisms of Petri nets that are the identity on the set of places.
\end{itemize}
This gives a symmetric monoidal double category $F \lCsp$, and using \cref{thm:equiv} we can show that this is isomorphic, as a symmetric monoidal double category, to ${}_L \lCsp(\Petri)$.

The machinery of structured cospans has been used to provide a semantics for open Petri nets \cite{BM}: a symmetric monoidal double functor from ${}_L \lCsp(\Petri)$ to  a symmetric monoidal double category of `open commutative monoidal categories'.  Presumably this double functor can equivalently be obtained using the machinery of decorated cospans, with the help of \cref{thm:functoriality}.  However, it should be clear by now that so far, in cases where either structured or decorated cospans can be used, structured cospans are simpler.   We next turn to an example where decorated cospans are necessary.

\subsection{Petri nets with rates}
\label{subsec:petrirates}

In chemistry, population biology, epidemiology and other fields, modelers use `Petri nets with rates', where the transitions are labeled with nonnegative real numbers called `rate constants' \cite{Haas,Koch,Wilkinson}.   From any Petri net with rates one can systematically construct a dynamical system.  Mathematical chemists have proved deep theorems relating the topology of Petri nets with rates to the qualitative behavior of their dynamical systems \cite{CTF}.

Pollard and the first author showed how to construct an open dynamical system from any open Petri net with rates,  thus defining a functor from a category with open Petri nets with rates as morphisms to one with open dynamical systems as morphisms \cite{BP}.  They used Fong's original decorated cospans to do this.  Here we show we show how to promote these categories to double categories using our new approach to decorated cospans.

First, to briefly illustrate these ideas, here is an open Petri net with rates:
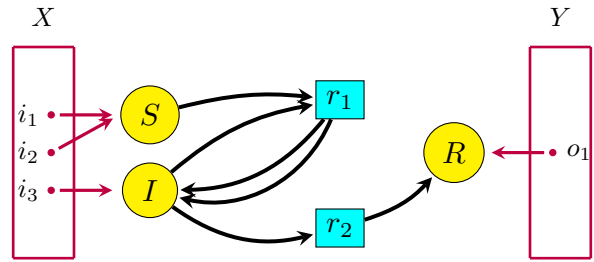
\[
\begin{tikzpicture}
	\begin{pgfonlayer}{nodelayer}
\node [style=inputdot] (1a) at (-5.3, 0.5) {};
\node [style=inputdot] (1b) at (-5.3, 0) {};
\node [style=inputdot] (2) at (-5.3, -0.5) {};
\node [style=inputdot] (3) at (1.3, 0) {};
\node [style=none] (1'a) at (-4.4, 0.5) {};
\node [style=none] (1'b) at (-4.4, 0.5) {};
\node [style=none] (2') at (-4.4, -0.5) {};
\node [style=none] (3') at (0.4, 0) {};
\node [style=species] (S) at (-4, 0.5) {$S$};
\node [style=species] (I) at (-4, -0.5) {$I$};
\node [style=species] (R) at (0, 0) {$R$};
		\node [style=transition] (tau1) at (-1.5, 0.7) {$r_1$};
            \node [style=transition] (tau2) at (-1.5,-1) {$r_2$};
		\node [style=none] (I1a) at (-5.6, 0.5) {$i_1$};
		 \node [style=none] (I1b) at (-5.6, 0) {$i_2$};
           \node [style=none] (I2) at (-5.6, -0.5) {$i_3$};
		\node [style=none] (O3) at (1.65, 0) {$o_1$};
		\node [style=none] (ATL) at (-5.8,1.4) {};
		\node [style=none] (ATR) at (-5,1.4) {};
		\node [style=none] (ABR) at (-5,-1.4) {};
		\node [style=none] (ABL) at (-5.8,-1.4) {};
		\node [style=none] (MTL) at (1,1.4) {};
		\node [style=none] (MTR) at (1.8,1.4) {};
		\node [style=none] (MBR) at (1.8,-1.4) {};
		\node [style=none] (MBL) at (1,-1.4) {};
		\node [style=none] (X) at (-5.4,1.8) {$X$};
		\node [style=none] (Z) at (1.4,1.8) {$Y$};

	\end{pgfonlayer}
	\begin{pgfonlayer}{edgelayer}
		\draw [style=inarrow, bend left =10, looseness=1.00] (S) to (tau1);
		\draw [style=inarrow, bend left=15, looseness=1.00] (I) to (tau1);
		\draw [style=inarrow, bend left=25, looseness=1.00] (tau1) to (I);
		\draw [style=inarrow, bend left=40, looseness=1.00] (tau1) to (I);
	       \draw [style=inarrow, bend right=25, looseness=1.00] (I) to (tau2);
             \draw [style=inarrow, bend right=15, looseness=1.00] (tau2) to (R);
\path[color=purple, very thick, shorten >=2pt, shorten <=2pt, ->, >=stealth] (1a) edge (1'a);
\path[color=purple, very thick, shorten >=2pt, shorten <=2pt, ->, >=stealth] (1b) edge (1'b);
\path[color=purple, very thick, shorten >=2pt, shorten <=2pt, ->, >=stealth] (2) edge (2');
\path[color=purple, very thick, shorten >=2pt, shorten <=2pt, ->, >=stealth] (3) edge (3');
		\draw [style=simple] (ATL.center) to (ATR.center);
		\draw [style=simple] (ATR.center) to (ABR.center);
		\draw [style=simple] (ABR.center) to (ABL.center);
		\draw [style=simple] (ABL.center) to (ATL.center);
		\draw [style=simple] (MTL.center) to (MTR.center);
		\draw [style=simple] (MTR.center) to (MBR.center);
		\draw [style=simple] (MBR.center) to (MBL.center);
		\draw [style=simple] (MBL.center) to (MTL.center);
	\end{pgfonlayer}
\end{tikzpicture}
\]
It is an open Petri net where the transitions are labeled with rate constants $r_1, r_2 \geq 0$.
Here is the corresponding open dynamical system:
\begin{equation}
\label{eq:openPetrir}
  \begin{array}{ccl} \displaystyle{\frac{dS(t)}{dt}} &=& -r_1 \, S(t)I(t)  + I_1(t) + I_2(t) \\ \\
\displaystyle{\frac{dI(t)}{dt}}  &=& r_1\, S(t)I(t) - r_2 \, I(t) + I_3(t)  \\  \\
\displaystyle{\frac{dR(t)}{dt}}  &=&   r_2 \, I(t)  - O_1(t).
\end{array}
\end{equation}
Here $I_1(t),I_2(t),I_3(t)$ and $O_1(t)$ are arbitrary smooth functions of time, which describe inflows and outflows at the points $i_1,i_2,i_3 \in X$ and $o_1 \in Y$.
If we drop these inflow and outflow terms, we obtain a dynamical system: an autonomous system of coupled nonlinear first-order ordinary differential equations.   In fact these equations are a famous model of infectious disease, the `SIR model', where $S(t)$, $I(t)$ and $R(t)$ describe the populations of susceptible, infected and recovered individuals, respectively.   The inflow and outflow terms allow individuals to enter or leave the population.   This in turn lets us couple the SIR model to other models, and build larger models from smaller pieces.    Indeed, a group of researchers has recently used open Petri nets and the mathematics of structured cospans in their software tool for building and manipulating epidemiological models \cite{AP,BFMLP}.

Now we turn to the details.  A \define{Petri net with rates} is a Petri net $s,t \maps T \to \N[S]$ together with a function $r \maps T \to [0,\infty)$ assigning to each transition $\tau \in T$ a nonnegative real number called its \define{rate constant}.  There is a category $\Petri_r$ whose objects are Petri nets with rates, where a morphism from
\[   \xymatrix{ [0,\infty) & T \ar[l]_-r \ar@<-.5ex>[r]_-t \ar@<.5ex>[r]^-s & \N[S] }\]
 to
 \[   \xymatrix{ [0,\infty) & T' \ar[l]_-{r'} \ar@<-.5ex>[r]_-{t'} \ar@<.5ex>[r]^-{s'} & \N[S'] }\]
is a morphism of the underlying Petri nets whose map $g \maps T \to T'$ obeys
\[    r'(\tau') = \sum_{\{\tau \in T: g(\tau) = \tau'\}} r(\tau) \]
for all $\tau' \in T'$.   This definition was suggested by Sophie Libkind; it agrees with the earlier definition \cite{BP} in the case of isomorphisms, but not in general, and the difference is important here.

We can describe open Petri nets with rates using decorated cospans.  There is a symmetric lax monoidal pseudofunctor $F \maps (\Fin\Set, +) \to (\Cat, \times)$ such that for any finite set $S$, the category $F(S)$ has:
\begin{itemize}
\item objects given by Petri nets with rates whose set of places is $S$,
\item morphisms given by morphisms of Petri nets with rates that are the identity on the set of places.
\end{itemize}
This gives a symmetric monoidal double category $F \lCsp$ where the horizontal 1-cells are called \define{open Petri nets with rates}.

There is also a symmetric monoidal double category of open dynamical systems.   A dynamical system is a vector field, thought of as giving a system of first-order ordinary differential equations.  A Petri net with rates gives a special sort of dynamical system: an \define{algebraic} vector field on $\R^S$ for some finite set $S$, meaning a vector field whose components are polynomials in the coordinates. We shall think of such a vector field as a special sort of function $v \maps \R^S \to \R^S$.

Using Fong's original approach to decorated cospans, Pollard and the first author constructed a symmetric monoidal category for which the morphisms are open dynamical systems \cite[Theorem 17]{BP}.  This category is constructed from a symmetric lax monoidal functor
$D \maps \Fin\Set \to \Set$ such that:
\begin{itemize}
\item $D$ maps any finite set $S$ to
\[ D(S) = \{ v \maps \R^S \to \R^S | \; v \textrm{ is algebraic}  \}. \]
\item $D$ maps any function $f \maps S \to S'$ between finite sets to the function $D(f) \maps D(S) \to D(S')$ given as follows:
\[ D(f)(v) = f_* \circ v \circ f^* \]
where the pullback $ f^* \maps \R^{S'} \to \R^S $ is given by
\[ f^*(c)(\sigma) = c(f(\sigma)) \]
while the pushforward $ f_* \maps \R^{S} \to \R^{S'} $ is given by
\[ f_*(c)(\sigma') = \sum_{ \{ \sigma \in S : f(\sigma) = \sigma' \} } c(\sigma). \]
\end{itemize}
The functorality of $D$ is proved in \cite[Lemma 15]{BP} while the symmetric lax monoidal stucture is given in Lemma 16 of that paper.

Since every set gives a discrete category with that set of objects, we can reinterpret $D$ as a symmetric lax monoidal pseudofunctor $D \maps (\Fin\Set, +) \to (\Cat, \times)$ which happens to actually be a functor.  Applying \cref{thm:decorated_cospans_2} we obtain a symmetric monoidal double category $D \lCsp$ where:
\begin{itemize}
\item objects are finite sets,
\item vertical 1-morphisms are functions,
\item a horizontal 1-cell from $X$ to $Y$ is an \define{open dynamical system}, that is, a cospan
\[
\begin{tikzpicture}[scale=1.5]
\node (D) at (-3,0) {$X$};
\node (E) at (-2,0) {$S$};
\node (F) at (-1,0) {$Y$};
\path[->,font=\scriptsize,>=angle 90]
(D) edge node [above] {$i$} (E)
(F) edge node [above] {$o$} (E);
\end{tikzpicture}
\]
in $\Fin\Set$ together with an algebraic vector field $v \in D(S)$,
\item a 2-morphism from
\[
\begin{tikzpicture}[scale=1.5]
\node (D) at (-3,0) {$X$};
\node (E) at (-2,0) {$S$};
\node (F) at (-1,0) {$Y,$};
\node (G) at (0,0) {$v \in D(S)$};
\path[->,font=\scriptsize,>=angle 90]
(D) edge node [above] {$i$} (E)
(F) edge node [above] {$o$} (E);
\end{tikzpicture}
\]
to
\[
\begin{tikzpicture}[scale=1.5]
\node (D) at (-3,0) {$X'$};
\node (E) at (-2,0) {$S'$};
\node (F) at (-1,0) {$Y'$};
\node (G) at (0,0) {$v' \in D(S')$};
\path[->,font=\scriptsize,>=angle 90]
(D) edge node [above] {$i'$} (E)
(F) edge node [above] {$o'$} (E);
\end{tikzpicture}
\]
is a diagram
\[
\begin{tikzpicture}[scale=1.5]
\node (E) at (3,0) {$X$};
\node (F) at (5,0) {$Y$};
\node (G) at (4,0) {$S$};
\node (E') at (3,-1) {$X'$};
\node (F') at (5,-1) {$Y'$};
\node (G') at (4,-1) {$S'$};
\path[->,font=\scriptsize,>=angle 90]
(F) edge node[above]{$o$} (G)
(E) edge node[left]{$f$} (E')
(F) edge node[right]{$g$} (F')
(G) edge node[left]{$h$} (G')
(E) edge node[above]{$i$} (G)
(E') edge node[below]{$i'$} (G')
(F') edge node[below]{$o'$} (G');
\end{tikzpicture}
\]
in $\Fin\Set$ such that $D(h)(v) = v'$.
\end{itemize}

Next, we can define a symmetric monoidal double functor
\[      \graysquare \maps F \lCsp \to D \lCsp \]
sending any open Petri net with rates to its corresponding open dynamical system.  This was already defined at the level of categories by Pollard and the first author \cite[Section 7]{BP}, who called it `gray-boxing'.   To boost this result to the double category level we use \cref{thm:functoriality}, taking the square in that theorem to be
\[
\begin{tikzpicture}[scale=1.5]
\node (A) at (0,0) {$\Fin\Set$};
\node (B) at (1.5,0) {$\Cat$};
\node (C) at (0,-1) {$\Fin\Set$};
\node (D) at (1.5,-1) {$\Cat$.};
\node (E) at (0.75,-0.5) {$\Downarrow \theta$};
\path[->,font=\scriptsize,>=angle 90]
(A) edge node[above]{$F$} (B)
(A) edge node[left]{$1$} (C)
(B) edge node[right]{$1$} (D)
(C) edge node[below]{$D$} (D);
\end{tikzpicture}
\]
Here $\theta$ is given as follows.  For any finite set $S$, $\theta_S \maps F(S) \to D(S)$ maps any Petri net with rates
\[   \xymatrix{ [0,\infty) & T \ar[l]_-r \ar@<-.5ex>[r]_-t \ar@<.5ex>[r]^-s & \N[S] }\]
to an algebraic vector field on $\R^S$, say $v$.   This vector field is defined using a standard prescription taken from chemistry, called `the law of mass action'.   Namely, for any $c \in \R^s$, we set
\[
v(c) = \sum_{\tau \in T} r(\tau) \, ( t(\tau) - s(\tau) ) c^{s(\tau)}
\]
where
\[     c^{s(\tau)} = \prod_{i \in S} {c_i}^{s(\tau)_i}  \]
and we think of $t(\tau), s(\tau) \in \N[S]$ as vectors in $\R^S$.   This formula is explained in the paper with Pollard \cite{BP}.   Using the new definition of morphisms in $F(S)$, one can check that $\theta$ extends to a monoidal natural transformation between the functors $F, D \maps (\Fin\Set,+) \to (\Cat,\times)$.  Thus, it defines a symmetric monoidal double functor $\graysquare \maps F \lCsp \to D \lCsp$.

In applications, this double functor lets us turn an open Petri net with rates into an open dynamical system as follows.  Given a Petri net with rates and defining $v$ as above, we obtain a system of first-order ordinary differential equations for a function $c \maps \R \to \R^S$
called the \define{rate equation}:
\[    \frac{d}{dt} c(t) = v(c(t)) .  \]
More generally, given an open Petri net with rates
\[
\begin{tikzpicture}[scale=1.5]
\node (D) at (-3,0) {$X$};
\node (E) at (-2,0) {$S$};
\node (F) at (-1,0) {$Y,$};
\node (G) at (0,0) {$P \in F(S)$};
\path[->,font=\scriptsize,>=angle 90]
(D) edge node [above] {$i$} (E)
(F) edge node [above] {$o$} (E);
\end{tikzpicture}
\]
we get an equation called the \define{open rate equation}:
\[     \frac{d}{dt} c(t) = v(c(t))  + i_*(I(t)) - o_*(O(t)) \]
where $v$ is defined as above and $I \maps \R \to \R^X$ and $O \maps \R \to \R^Y$ are arbitrary smooth functions describing \define{inflows} and \define{outflows}, respectively.   Applying this prescription to the open Petri net with rates shown at the start of this section one gets the differential equations \cref{eq:openPetrir}.  Other examples are worked out in \cite{BP}.

We now show that the decorated cospan double category $D \lCsp$ of open dynamical systems is not isomorphic to a structured cospan double category via \cref{thm:equiv}. Recall that in this theorem we start with the data required to build a decorated cospan category, namely a symmetric lax monoidal pseudofunctor $F \maps (\A,+) \to (\Cat,\times)$, and show that if the resulting pseudofunctor $F \maps \A \to \SMC$ factors through $\Rex$, then the opfibration $U \maps \X = \inta F \to \A$ has a left adjoint $L \maps \A \to \X$. We then obtain an isomorphism between decorated and structured cospan double categories, $F \lCsp \cong {}_L \lCsp(\X)$. We now show that in the case at hand, where $F = D$ is the functor sending each finite set $S$ to the set of dynamical systems on $\R^S$, the opfibration $U$ does \emph{not} have a left adjoint. Thus, the conditions of \cref{thm:equiv} cannot hold in this case: $F$ does not factor through $\Rex$.

Taking $D$ as above, it is easy to see that in the category $\inta D$
\begin{itemize}
\item an object is a pair $(S,v)$ where $S$ is a finite set and $v$ is an algebraic vector field $v \maps \R^S \to \R^S$,
\item a morphism from $(S,v)$ to $(S',v')$ is a function $f \maps S \to S'$ such that $v' = f_* \circ v \circ f^*$
\end{itemize}
with the usual composition of functions.   The forgetful functor $U \maps \inta D \to \Fin\Set$ acts as follows:
\begin{itemize}
\item on objects, $D(S,v) = S$,
\item on morphisms, $D(f) = f$.
\end{itemize}

To show that $U$ does not have a left adjoint, we use a known result \cite[Lemma 4.6.1]{Riehl}:
\begin{lem} \label{lem:initial}
A functor $U \maps \A \to \X$ admits a left adjoint if and only if for every $x \in \X$, the comma category $x \downarrow U$ has an initial object.
\end{lem}
Because the empty set is initial in $\Fin\Set$, the comma category
$\emptyset \downarrow U$ is just $\inta D$.  This contains an object $(\emptyset, v_\emptyset)$, where $v_\emptyset$ is the only possible vector field on $\R^\emptyset$, namely, the zero vector field.   The only object in $\inta D$ with any morphisms to $(\emptyset, v_\emptyset)$ is $(\emptyset, v_\emptyset)$ itself, so no other object can be initial.  However $(\emptyset, v_\emptyset)$ is not initial either, because it has no morphisms to an object $(S,v)$ unless $v$ is the zero vector field on $\R^S$.  Thus by \cref{lem:initial}, $U$ does not have a left adjoint.

\section{Conclusions}\label{sec:conclusions}

We have given conditions under which a decorated cospan double category is isomorphic to a structured cospan double category, in \cref{thm:equiv}. The converse question is also interesting: is every structured cospan double category isomorphic to a decorated cospan double category? The answer is similar to the previous one: yes, under certain conditions that let us pass from an appropriate functor $L \maps \A \to \X$ to an appropriate pseudofunctor $F \maps \A \to \Cat$.

Let us now sketch the story; details will appear in a forthcoming paper \cite{CV}.   Suppose the conditions hold for constructing the double category of structured cospans ${}_L \lCsp(\X)$ as in \cref{thm:structured_cospans}.  That is, suppose $\A$ and $\X$ have finite colimits and $L \maps \A \to \X$ preserves them.    If $L$ also has a right adjoint `left inverse' (meaning the unit is the identity) $U \maps \X \to \A$, which moreover strictly preserves the chosen pushouts, it can be shown that $U$ is an opfibration.  Consequently, $U$ corresponds to a pseudofunctor $F \maps \A \to \Cat$ by the inverse Grothendieck construction, as in the first part of \cref{thm:Grothendieck}. Furthermore, if $U$ preserves finite coproducts, $F$ acquires the structure of a symmetric lax monoidal pseudofunctor $F \maps (\A,+) \to (\Cat,\times)$ by the special case of the cocartesian monoidal Grothendieck construction discussed under \cref{lem:MonGroth}.  As a result, $F$ now has enough structure to induce a double category of decorated cospans $F\lCsp$ as in \cref{thm:decorated_cospans_2}.  Finally, it can be shown that the structured and decorated cospan double categories are isomorphic as symmetric monoidal double categories: ${}_L \lCsp(\X) \cong F\lCsp$.

To give a better sense of how the pseudofunctor $F \maps \A \to\Cat$ is constructed: for each object $a \in A$, $F(a)$ is defined to be the fiber of $U$ over $a$, namely the category of all objects in $x \in \X$ such that $U(x)=a$ and morphisms $k\maps x\to y$ such that $U(k)=1_a$.  Given a morphism $f \maps a \to b$, there is a functor $F(f) \maps F(a)\to F(b)$ that maps $x \in F(a)$ to the following pushout:
\begin{displaymath}
 \begin{tikzcd}
La\ar[r,"Lf"]\ar[d,"\varepsilon_x"']\ar[dr,phantom,near end,"\ulcorner"] & Lb\ar[d] & \\
x\ar[d,dotted]\ar[r] & x+_{La}Lb\ar[d,dotted] & \textrm{in }\X \\
a\ar[r,"f"'] & b & \textrm{in }\A
 \end{tikzcd}
\end{displaymath}
where $\varepsilon_x \maps LU(x)=L(a)\to x$ is the counit of the adjunction $L\dashv U$. The fact that $U$ strictly preserves pushouts is necessary to show that the pushout is mapped, via $U$, directly down to $b$.

Even though for both \cref{thm:equiv} and the above result the conditions stated are only sufficient, they suggest that with work we could establish this functorial picture:
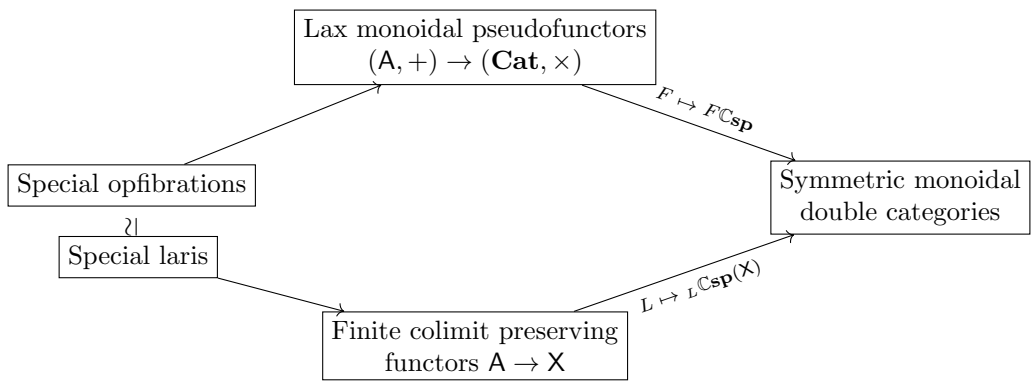
\begin{displaymath}
\begin{tikzpicture}
\node [align=center,draw] (A) { Lax monoidal pseudofunctors \\ $(\A,+)\to(\Cat,\times)$};
\node[below right=1 and 1.5 of A,align=center,draw] (C)  { Symmetric monoidal \\double categories };
\node[below left=1.05 and .5  of A,draw] (D)  {Special opfibrations};
\node[below left=2 and 1 of A,draw] (E)  {Special laris};
\node[below=3 of A,align=center,draw] (B)  {Finite colimit preserving \\ functors $\A\to\X$};
\node[below=0 of D] {\rotatebox{90}{$\simeq$}};
\path[->,font=\scriptsize]
(A) edge node[above,sloped, midway]{$\quad F\mapsto F\lCsp$}  (C)
(D) edge (A)
(E) edge (B)
(B) edge node[below,sloped, midway]{$\quad L\mapsto {}_{L}\lCsp(\X)$} (C);
\end{tikzpicture}
\end{displaymath}
with a natural isomorphism in the middle.  The connection between opfibrations and laris goes back to Gray's \cref{prop:opfibtolari}, but we need to specialize it to a class suitable for both the structured and decorated cospan constructions.  This would imply that starting from an appropriate middle ground, these two constructions are essentially the same.   We leave such considerations for future work.

Finally, it is worth mentioning a structured cospan double category to which the argument sketched above does \emph{not} apply.   For any functor $\phi \maps \C \to \D$ between small categories, precomposition with $\phi$ gives a functor $R \maps \widehat{\D} \to \widehat{\C}$ between presheaf categories which has a left adjoint $L \maps \widehat{\C} \to \widehat{\D}$.   Since presheaf categories have all small colimits and $L$ preserves them, the conditions of \cref{thm:structured_cospans} apply and we obtain a symmetric monoidal double category of structured cospans, ${}_L\lCsp(\widehat{\D})$.  However, the right adjoint $R$ is not always an opfibration---and when it is not, the above arguments cannot be used to show that ${}_L\lCsp(\widehat{\D})$ is a decorated cospan double category.  A simple example where $R$ is not an opfibration was  provided to us by Morgan Rogers.  Take $\phi \maps \mathsf{1} + \mathsf{1} \to \mathsf{1}$ to be the unique functor where $\mathsf{1}$ is the terminal category.  Then $R \maps \Set \to \Set^2$ is the diagonal, and the morphism $(0,0) \to (0,1)$ in $\Set^2$ admits no lift at all to $0 \in \Set$.   It will be interesting to find conditions on $\phi \maps \C \to \D$ that
guarantee ${}_L\lCsp(\widehat{\D})$ is isomorphic to a decorated double cospan category.

\appendix

\section{Definitions}
In this appendix, we gather some well-known concepts required to make the material self-contained, as well as references to more detailed expositions.

\subsection{Bicategories}
\label{subsec:bicats}

For standard 2-categorical material, we refer the reader to \cite{KS}.  For monoidal 2-categories see \cite{DS}, and for detailed definitions concerning monoidal bicategories see \cite{GPS,McCrudden,Stay}.  Briefly, a \define{monoidal} bicategory $\bA$ comes with a pseudofunctor $\otimes\maps\bA\times\bA\to\bA$ and a unit object $I$ that are associative and unital up to coherent equivalence. A \define{braided} monoidal bicategory also comes with a pseudonatural equivalence $\beta_{a,b}\maps a\ot b\to b\ot a$ and appropriate invertible modifications obeying certain equations; it is \define{sylleptic} if there is an invertible modification $1_{a\ot b}\Rrightarrow\beta_{b,a}\circ\beta_{a,b}$ obeying its own equation, and \define{symmetric} if one further axiom holds.

A \define{lax monoidal} pseudofunctor (called \emph{weak monoidal homomorphism} in some earlier references) between monoidal bicategories $F\maps\bA\to\bB$ is a pseudofunctor equipped with pseudonatural transformations with components $\phi_{a,b}\maps Fa\otimes Fb\to F(a\otimes b)$ and $\phi_0\maps I\to FI$ along with invertible modifications for associativity and unitality with components
\begin{equation}\label{eq:omega}
\begin{tikzcd}[column sep=.8in, row sep=.3in]
(F a\ot F b)\ot F c\ar[ddr,phantom,"{\cong}"]\ar[r,"{\phi_{a,b}\ot1}"]\ar[d,anchor=center,sloped,swap,"\sim"] & F (a\ot b)\ot F
c\ar[d,"{\phi_{a\ot b,c}}"] \\
F a\ot( F b\ot F c)\ar[d,"{1\ot\phi_{b,c}}"'] &
F ((a\ot b)\ot c)\ar[d,anchor=center,sloped,"\sim"] \\
F a\ot F(b\ot c)\ar[r,"{\phi_{a,b\ot c}}"'] &
F (a\ot(b\ot c))
\end{tikzcd}
\end{equation}
\begin{displaymath}
\begin{tikzcd}[column sep=.3in,row sep=.3in]
F a\ar[r,"\sim"]\ar[drr,sloped,"\sim"'] & F a\otimes I\ar[r,"1\otimes\phi_0"]\ar[dr,phantom,"\cong"] & F
a\otimes F I\ar[d,"\phi_{a,I}"] \\
&& F(a\ot I)
\end{tikzcd}\qquad
\begin{tikzcd}[column sep=.3in,row sep=.3in]
F a\ar[r,"\sim"]\ar[drr,sloped,"\sim"'] & I\otimes F a \ar[r,"\phi_0\otimes1"]\ar[dr,phantom,"\cong"] & F
I\otimes F a\ar[d,"\phi_{I,a}"] \\
&& F(I\ot a)
\end{tikzcd}
\end{displaymath}
subject to coherence conditions listed in \cite[Definition 2]{DS}.
In particular, pseudonaturality of the monoidal structure means that it comes with isomorphisms of this form:
\begin{equation}\label{eq:pseudonaturality}
\begin{tikzcd}[column sep=.5in]
 Fa\ot Fb\ar[r,"Ff\ot Fg"]\ar[d,"\phi_{a,b}"']\ar[dr,phantom,"\stackrel{\phi_{f,g}}{\cong}"] & Fa'\ot Fb'\ar[d,"\phi_{a',b'}"] \\
 F(a\ot b)\ar[r,"F(f\ot g)"'] & F(a'\ot b')
 \end{tikzcd}
\end{equation}
 natural in $f$ and $g$.
A \define{braided lax monoidal} pseudofunctor between braided monoidal bicategories comes with an invertible modification with components
\begin{equation}\label{eq:braidedpseudofun}
 \begin{tikzcd}
Fa\otimes Fb\ar[r,"\phi_{a,b}"]\ar[d,"\beta_{Fa,Fb}"']\ar[dr,phantom,"\stackrel{u_{a,b}}{\cong}"] & F(a\ot b)\ar[d,"F(\beta_{a,b})"] \\
Fb\otimes Fa\ar[r,"\phi_{b,a}"'] & F(b\ot a)
 \end{tikzcd}
\end{equation}
subject to two axioms found e.g. in \cite[Definition 14]{DS}.
A \define{sylleptic lax monoidal} pseudofunctor satisfies one extra condition and a \define{symmetric lax monoidal} pseudofunctor between symmetric monoidal bicategories is just a sylleptic one.

\subsection{Fibrations and opfibrations}\label{subsec:fibrations}

Basic material regarding the theory of fibrations can be found, for example, in \cite{Borc,Gray}. Recall that a functor $U \maps \X \to \A$ is an \textbf{opfibration} if for every $x\in\X$ with $U(x)=a$ and $f \maps a \to b$ in $\A$, there exists a \textbf{cocartesian lifting} of $f$ to $x$, namely a morphism $\beta$ in $\X$ with domain $x$ with $U(\beta) = f$ and the following universal property: for any $g\maps b\to b'$ in $\A$ and $\gamma\maps x\to y'$ in $\X$ above the composite $g\circ f$, there exists a unique $\delta\maps y\to y'$ such that $U(\delta)=g$ and $\gamma=\delta\circ\beta$ as shown below.
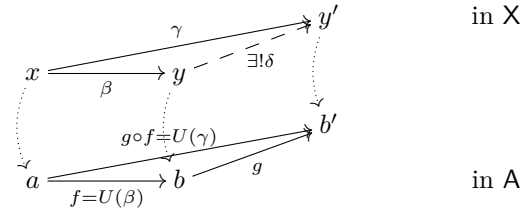
\begin{displaymath}
\xymatrix @R=.1in @C=.6in
{&& y'\ar @{.>}@/_/[dd] & \textrm{in }\X\\
x\ar[r]_-{\beta} \ar @{.>}@/_/[dd]
\ar[urr]^-{\gamma} &
y \ar @{.>}@/_/[dd] \ar @{-->}[ur]_-{\exists! \delta}
  \\
&& b' &&\\
a\ar[r]_-{f=U(\beta)} \ar[urr]^-{g\circ f=U(\gamma)}
 & b \ar[ur]_-g && \textrm{in }\A}
\end{displaymath}
The category $\X$ is called the \textbf{total} category and $\A$ is called the \textbf{base} category of the opfibration. For any $a\in\A$, the \textbf{fiber} above $a$ is the category $\X_a$ consisting of all objects that map to $a$ and \define{vertical} morphisms between them, i.e., morphisms mapping to $1_a$.

Assuming the axiom of choice, we may select a cocartesian lifting of each morphism $f\maps a\to b$ in $\A$ to each $x\in\X _a$, denoted by $\mathrm{Cocart}(f,x)\maps x\to f_!(x)$, rendering $U$ a so-called \textbf{cloven} opfibration. This choice induces \textbf{reindexing functors} $f_!\maps\X _a\to\X _b$ between the fibers, which by the lifting's universal property come equipped with natural isomorphisms $(1_a)_!\cong 1_{\X _a}$ and $(f\circ g)_!\cong f_!\circ g_!$.   With the help of these, any cloven opfibration $U \maps \X \to \A$ gives a pseudofunctor $F \maps \A \to \Cat$, where $\A$ is viewed as a 2-category with trivial 2-morphisms, $F(a) = \X_a$ for each object $a \in \A$, and $F(f) = f_!$ for each morphism $f$ in $\A$.

In fact, there is a 2-equivalence between opfibrations and pseudofunctors induced by the so-called `Grothendieck construction', or more specifically the `covariant'  Grothendieck construction, since there is also a version of this construction for fibrations.  Let $\OpFib(\A)$ denote the 2-subcategory of the slice 2-category $\Cat/ \A$ of opfibrations over $\A$, functors that
preserve cocartesian liftings, and natural transformations with vertical components.
\begin{defn}\label{def:GrothCat}
For any pseudofunctor $F\maps\A\to\Cat$ where $\A$ is a category viewed as a 2-category with trivial 2-morphisms, the \textbf{Grothendieck category}
$\inta F$ has
\begin{itemize}
\item objects pairs $(a, s \in F(a))$ and
\item a morphism from $(a, s \in F(a))$ to $(b, t\in F(b))$ is a pair $(f \maps a \to b,k \maps F(f)(s) \to t)$.
\end{itemize}
The identity morphism of $(a, s\in F(a))$ is $(1_a\maps a\to a,F(1_a)(s)\xrightarrow{\sim} s)$ and the composite of $(f,k) \maps (a,s) \to (b,t)$ and $(g,\ell) \maps (b,t) \to (c,u)$ is
\[
\left(a\xrightarrow{f}b\xrightarrow{g}c, \; F(g\circ f)(x)\xrightarrow{\sim} (Fg)((Ff)(s))\xrightarrow{(Fg)(k)}(Fg)(t)\xrightarrow{\ell}u\right)
\]
This is an opfibered category over $\A$ via the obvious forgetful functor, with fibers $(\inta F)_a=F(a)$ and reindexing functors $f_!=F(f)$.
\end{defn}
The constructions sketched so far---the Grothendieck construction and the construction of a pseudofunctor into $\Cat$ from a cloven opfibration---are the two halves of the following equivalence.

\begin{thm}\label{thm:Grothendieck}\hfill
\begin{enumerate}
\item Every opfibration $\X \to \A$ gives rise to a pseudofunctor $\A \to \Cat$.
\item Every pseudofunctor $\A \to \Cat$ gives rise to an opfibration $\inta F \to\A$.
\item The above correspondences yield an equivalence of 2-categories
\begin{displaymath}
[\A,\Cat]_\pse \simeq \OpFib(\A)
\end{displaymath}
where $[\A,\Cat]_\pse$ is the 2-category of pseudofunctors from $\A$ to $\Cat$, pseudonatural transformations, and modifications.
\end{enumerate}
\end{thm}

\begin{proof}
The idea goes back to Grothendieck; a proof can be found in, for example, \cite[Section 1.10]{Jacobs}.
\end{proof}

All the above concepts and results have analogues for fibrations.   A functor $\Phi \maps \X \to \A$ a \define{fibration} if and only if $\Phi\op \maps \X\op \to \A\op$ is an opfibration.  Equivalently, $\Phi$ is a fibration if and only if for every $y \in \X$ with $U(y) = b$ and $f \maps a \to b$ in $\A$ there exists a cartesian lifting of $f$ to $y$, where this concept is defined dually to cocartesian lifting.     Furthermore, if $\Phi$ is a fibration there is a contravariant reindexing functor
\[    f^\ast \maps \X_b \to \X_a \]
for each morphism $f \maps a \to b$.  Moreover, $\Phi\colon\X\to\A$ is a \define{bifibration} if it is both a fibration and opfibration.

If $\X$ and $\A$ are (symmetric) monoidal categories, a \define{(symmetric) monoidal fibration} $\Phi \maps \X \to \A$ is a fibration that is also a (symmetric) strict monoidal functor, such that the tensor product preserves cartesian liftings.  Similarly, a \define{(symmetric) monoidal opfibration} is an opfibration that is also a (symmetric) strict monoidal functor, such that the tensor product preserves cocartesian liftings.   Finally, a \define{(symmetric) monoidal bifibration} is a bifibration that is also a (symmetric) strict monoidal functor such that the tensor product preserves both cartesian and cocartesian liftings.

In a bifibration we have both covariant and contravariant reindexing functors, and in
fact $f_!$ is left adjoint to $f^*$ \cite[Proposition 3.9]{Shulman2008}.   Using this, one can easily show that for any commutative square in $\A$
\[
\begin{tikzpicture}[scale=1.5]
\node (A) at (0,0) {$a$};
\node (B) at (1.5,0) {$b$};
\node (C) at (0,-1) {$c$};
\node (D) at (1.5,-1) {$d$};
\path[->,font=\scriptsize,>=angle 90]
(A) edge node[above]{$h$} (B)
(A) edge node[left]{$k$} (C)
(B) edge node[right]{$g$} (D)
(C) edge node[below]{$f$} (D);
\end{tikzpicture}
\]
the following square commutes up to a specified natural transformation:
\[
\begin{tikzpicture}[scale=1.5]
\node (A) at (0,0) {$\X_a$};
\node (B) at (1.5,0) {$\X_b$};
\node (C) at (0,-1) {$\X_c$};
\node (D) at (1.5,-1) {$\X_d$};
\node (E) at (0.75,-0.5) {$\scriptstyle\Downarrow \theta$};
\path[->,font=\scriptsize,>=angle 90]
(B) edge node[above]{$h^\ast$} (A)
(A) edge node[left]{$k_!$} (C)
(B) edge node[right]{$g_!$} (D)
(D) edge node[below]{$f^\ast$} (C);
\end{tikzpicture}
\]
where $\theta$ is built as a composite involving the unit of the adjunction between
$g_!$ and $g^\ast$ and the counit of the adjunction between $k_!$ and $k^\ast$:
\begin{equation}
\label{theta}
   k_! h^\ast \To k_! h^\ast g^\ast g_! \cong k_! k^\ast f^\ast g_! \To f^\ast g_! .
\end{equation}
If $\theta$ is a natural isomorphism whenever the original square in $\A$ is
a pushout, we say that the bifibration $\Phi$ is \define{Beck--Chevalley}.  (Shulman  uses the term `strongly co-BC' \cite[Definition 13.21] {Shulman2008}.)

\subsection{Double categories}\label{sec:doublecats}

For double categories we follow the notation of our paper on structured cospans \cite{BC}, which in turn follows that of Hansen and Shulman \cite{HS,Shulman2010}.  Our double categories are always `pseudo' double categories, where composition of horizontal 1-cells is unital and associative only up to coherent isomorphism \cite{GP1,GP2,Shulman2008}.

\begin{defn}\label{defn:double_category}
A \define{double category} $\lD$ consists of a \define{category of objects}
$\lD_0$, a \define{category of arrows} $\lD_1$, functors
\[   S,T \maps \lD_1 \to \lD_0, \;  U\maps \lD_0 \to \lD_1, \; \textrm{ and }
   \odot \maps \lD_1 \times_{\lD_0} \lD_1 \to \lD_1\]
called the \define{source} and \define{target}, \define{unit} and \define{composition}
functors, respectively, such that
\[  S(U_{A})=A=T(U_{A}),  \quad S(M \odot N)=S(N), \quad T(M \odot N)=T(M), \]
and natural isomorphisms called the \define{associator}
\[ \alpha_{L,M,N} \maps (L \odot M) \odot N \to L \odot (M \odot N)  \]
and \define{left and right unitors}
\[		\lambda_N \maps U_{T(N)} \odot N \to N, \qquad
     \rho_N \maps N \odot U_{S(N)} \to N \]
such that $S(\alpha), S(\lambda), S(\rho), T(\alpha), T(\lambda)$ and $T(\rho)$ are all identities,
such that the standard coherence laws hold: the pentagon identity for the
associator and the triangle identity for the left and right unitor.
\end{defn}

Objects of $\lD_0$ are called \define{objects} and morphisms of $\lD_0$ are called \define{vertical 1-morphisms}. Objects of $\lD_1$ are called \define{horizontal 1-cells} and morphisms of $\lD_1$ are called \define{2-morphisms}.   We can draw a 2-morphism $a \maps M \to N$ with $S(a)=f,T(a)=g$ as follows:
\[
\begin{tikzcd}
A\ar[tick,r,"M"]\ar[d,"f"']\ar[dr,phantom,"\scriptstyle\Downarrow\alpha"] & B\ar[d,"g"] \\
C\ar[tick,r,"N"'] & D
\end{tikzcd}
\]
We call $M$ and $N$ the \define{horizontal source and target} of $a$ respectively, and call $f$ and $g$ its \define{vertical source and target}.   A 2-morphism where $f$ and $g$ are identities is called \textbf{globular}.   For example, the associator and unitors in a double category are globular 2-morphisms.

\begin{defn}\label{def:doublefun}
Given double categories $\lD$ and $\lE$, a \define{double functor} $\lF \maps \lD \to \lE$ consists of:
\begin{itemize}
\item{functors $\lF_0 \maps \lD_0 \to \lE_0$ and $\lF_1 \maps \lD_1 \to \lE_1$ such that $S \lF_1 = \lF_0 S$ and $T \lF_1 = \lF_0 T$, and}
\item{for every composable pair of horizontal 1-cells $M$ and $N$ in $\lD$, a natural transformation $\mathbb{F}_\odot \maps \lF(N) \odot \lF(M) \to \lF(N \odot M)$ called the \define{composite comparison} and for every object $a$ in $\lD$, a natural transformation $\mathbb{F}_U \maps U_{\lF_0(a)} \to \lF_1(U_a)$ called the \define{unit comparison}. The components of each of these natural transformations are globular isomorphisms that must obey coherence laws analogous to those of a monoidal functor.}
\end{itemize}
\end{defn}

\begin{defn}
Given double functors $\lF,\mathbb{G} \maps \lD \to \lE$, a \define{double natural transformation} $\alpha \maps \lF \Rightarrow \mathbb{G}$ consists of natural transformations $\alpha_0 \maps \lF_0 \Rightarrow \mathbb{G}_0$ and $\alpha_1 \maps \lF_1 \Rightarrow \mathbb{G}_1$ such that:
\begin{itemize}
\item{$S(\alpha_M) = \alpha_{S(M)}$ and $T(\alpha_M) = \alpha_{T(M)}$ for all horizontal 1-cells $M$ of $\lD$,}
\item{$\alpha \circ \lF_\odot = \mathbb{G}_\odot \circ (\alpha_M \odot \alpha_N)$ for all composable pairs $M$ and $N$ of horizontal 1-cells in $\lD$, and}
\item{$\alpha \circ \lF_U = \mathbb{G}_U \circ \alpha$ for all objects $a$ of $\lD$.}
\end{itemize}
The double natural transformation $\alpha$ is a \define{double natural isomorphism} if both $\alpha_0$ and $\alpha_1$ are natural isomorphisms.
\end{defn}

Let $\Dbl$ denote the 2-category of double categories, double functors and double transformations. One can check that $\Dbl$ has finite products, and in any 2-category with finite products we can define a `pseudomonoid', which is a categorified analogue of a monoid \cite{DS}. For example, a pseudomonoid in $\Cat$ is a monoidal category.   We can also define symmetric pseudomonoids, which in $\Cat$ are symmetric monoidal categories.

\begin{defn}
\label{defn:monoidal_double_category}
A \define{monoidal double category} is a pseudomonoid in $\Dbl$, namely it is equipped with double functors $\otimes\maps\lD\times\lD\to\lD$, $I\maps\one\to\lD$ and invertible double transformations $\otimes\circ(1\times\otimes)\cong\otimes\circ(\otimes\times1)$, $\otimes\circ(1\times I)\cong1\cong\otimes\circ(I\times1)$ satisfying standard axioms.
\end{defn}
\noindent
Explicitly, a monoidal double category is a double category $\lD$ with:
\begin{itemize}
\item monoidal structures on both $\lD_0$ and $\lD_1$ (each with tensor product denoted $\otimes$, associator $a$, left unitor $\ell$ and right unitor $r$ and unit object $I$), such that $U \maps \lD_0 \to \lD_1$ strictly preserves the unit objects and $S,T \maps \lD_1 \to \lD_0$ are strict monoidal,
\item the structure of a double functor on $\otimes$:
that is, invertible globular 2-morphisms
\[ \chi \maps (M_2\otimes N_2)\odot(M_1\otimes N_1)  \simrightarrow
(M_2\odot M_1)\otimes (N_2\odot N_1)\]
\[ \mu \maps U_{A\otimes B} \simrightarrow U_A \otimes U_B\]
obeying a list of equations that can be found after \cite[Definition 2.10]{HS} and also \cite[Definition A.5]{BC}.
\end{itemize}

\begin{defn}
\label{defn:symmetric_monoidal_double_category}
A \define{symmetric monoidal double category} is a symmetric pseudomonoid in $\Dbl$.
\end{defn}
\noindent
Explicitly, a symmetric monoidal double category is a monoidal double category $\lD$ such that:
\begin{itemize}
		\item $\lD_0$ and $\lD_1$ are symmetric monoidal categories, with braidings both denoted $\beta$.
		\item The functors $S$ and $T$ are symmetric strict monoidal functors.
		\item The following diagrams commute, expressing that the braiding is a transformation of double categories:
		\begin{displaymath}
		\begin{tikzpicture}
			\node (A) at (0,1.5) {\footnotesize{$ (M_2\otimes N_2)\odot (M_1\otimes N_1)$}};
			\node (A') at (0,0) {\footnotesize{$(M_2 \odot M_1) \otimes (N_2 \odot N_1)$}};
			\node (B) at (5,1.5) {\footnotesize{$ (N_2 \otimes M_2)\odot (N_1 \otimes M_1)$}};
			\node (B') at (5,0) {\footnotesize{$(N_2\odot N_1) \otimes (M_2\odot M_1)$}};
			\path[->,font=\scriptsize]
				(A) edge node[left]{$\chi$} (A')
				(A) edge node[above]{$\beta \odot \beta$} (B)
				(B) edge node[right]{$\chi$} (B')
				(A') edge node[above]{$\beta$} (B');
		\end{tikzpicture}
		\quad
		\begin{tikzpicture}
			\node (A) at (0,1.5) {\footnotesize{$U_A \otimes U_B$}};
			\node (A') at (0,0) {\footnotesize{$U_B\otimes U_A$}};
			\node (B) at (2,1.5) {\footnotesize{$U_{A\otimes B} $}};
			\node (B') at (2,0) {\footnotesize{$U_{B\otimes A}$}};
			\path[->,font=\scriptsize]
				(A) edge node[left]{$\beta$} (A')
				(B) edge node[above]{$\mu$} (A)
				(B) edge node[right]{$U_\beta$} (B')
				(B') edge node[above]{$\mu$} (A');
		\end{tikzpicture}
		\end{displaymath}
\end{itemize}

\begin{defn}\label{defn:monoidal_double_functor}
Given symmetric monoidal double categories $\lD$ and $\lE$, a \define{symmetric monoidal double functor} $\lF \maps \lD \to \lE$ is a double functor $\lF$ together with invertible transformations $\lF_\otimes\maps {\otimes \circ (\lF,\lF)} \to \lF \circ \otimes$ and $I_\lE \to \lF \circ I_\lD$ that satisfy the usual coherence axioms for a symmetric monoidal functor.
\end{defn}
\noindent
Explicitly, a symmetric monoidal double functor is a double functor $\lF \maps \lD \to \lE$ such that:
\begin{itemize}
\item{$\lF_0$ and $\lF_1$ are symmetric monoidal functors,}
\item{we have equalities $\lF_0 S_{\lD} = S_{\lE} \lF_1$ and $\lF_0 T_{\lD} = T_{\lE} \lF_1$ of monoidal functors, and}
\item{the following diagrams commute, expressing that $\phi$ is a transformation of double categories:
		\begin{displaymath}
		\begin{tikzpicture}
			\node (A) at (0,1.5) {\footnotesize{$ (\lF(M_2)\otimes \lF(N_2))\odot (\lF(M_1)\otimes \lF(N_1))$}};
			\node (A') at (0,0) {\footnotesize{$(\lF(M_2) \odot \lF(M_1)) \otimes (\lF(N_2) \odot \lF(N_1))$}};
			\node (B) at (5,1.5) {\footnotesize{$\lF(M_2 \otimes N_2)\odot \lF(M_1 \otimes N_1)$}};
			\node (B') at (5,0) {\footnotesize{$\lF((M_2 \otimes N_2)\odot(M_1 \otimes N_1))$}};
			\node (A'') at (0,-1.5) {\footnotesize{$\lF(M_2 \odot M_1) \otimes \lF(N_2 \odot N_1)$}};
			\node (B'') at (5,-1.5) {\footnotesize{$\lF((M_2 \odot M_1)\otimes(N_2 \odot N_1))$}};
			\path[->,font=\scriptsize]
				(A) edge node[left]{$\chi$} (A')
				(A) edge node[above]{$\lF_\otimes \odot \lF_\otimes$} (B)
				(B) edge node[right]{$\lF_\odot$} (B')
				(A') edge node [left] {$\lF_\odot \otimes \lF_\odot$} (A'')
				(B') edge node [right] {$\lF(\chi)$} (B'')
				(A'') edge node[above]{$\lF_\otimes$} (B'');
		\end{tikzpicture}
		\quad
		\begin{tikzpicture}
			\node (A) at (0,1.5) {\footnotesize{$U_{\lF(a) \otimes \lF(b)}$}};
			\node (A') at (0,0) {\footnotesize{$U_{\lF(a)} \otimes U_{\lF(b)}$}};
			\node (B) at (2,1.5) {\footnotesize{$U_{\lF(a \otimes b)} $}};
			\node (B') at (2,0) {\footnotesize{$\lF(U_{a \otimes b})$}};
			\node (A'') at (0,-1.5) {\footnotesize{$\lF(U_a) \otimes \lF(U_b)$}};
			\node (B'') at (2,-1.5) {\footnotesize{$\lF(U_a \otimes U_b) $}};
			\path[->,font=\scriptsize]
				(A) edge node[left]{$\mu$} (A')
				(A') edge node [left] {$\lF_U \otimes \lF_U$} (A'')
				(B') edge node [right] {$\lF(\mu)$} (B'')
				(A) edge node[above]{$U_{\lF_\otimes}$} (B)
				(B) edge node[right]{$\lF_U$} (B')
				(A'') edge node[above]{$\lF_\otimes$} (B'');
		\end{tikzpicture}
		\end{displaymath}
}
\end{itemize}

\begin{defn}\label{def:isomorphism}
An \define{isomorphism} of symmetric monoidal double categories is a symmetric monoidal
double functor $\lF \maps \lD \to \lE$ that has an inverse.
\end{defn}

A symmetric monoidal double functor is an isomorphism if
it is bijective on objects, vertical 1-morphisms, horizontal 1-cells and 2-morphisms.

\begin{defn}\label{def:companion}
  Let $\lD$ be a double category and $f\maps A\to B$ a vertical
  1-morphism.  A \define{companion} of $f$ is a horizontal 1-cell
  $\fhat\maps A\to B$ together with 2-morphisms
	\[
	\raisebox{-0.5\height}{
	\begin{tikzpicture}
		\node (A) at (0,1) {$A$};
		\node (B) at (1,1) {$B$};
		\node (A') at (0,0) {$B$};
		\node (B') at (1,0) {$B$};
		\path[->,font=\scriptsize,>=angle 90]
			(A) edge node[above]{$\hat{f}$} (B)
			(A) edge node[left]{$f$} (A')
			(B) edge node[right]{$1$} (B')
			(A') edge node[below]{$U_B$} (B');
		%
		\node () at (0.5,0.5) {\scriptsize{$\Downarrow$}};
	\end{tikzpicture}
	}
	\quad \text{ and } \quad
	\raisebox{-0.5\height}{
	\begin{tikzpicture}
		\node (A) at (0,1) {$A$};
		\node (B) at (1,1) {$A$};
		\node (A') at (0,0) {$A$};
		\node (B') at (1,0) {$B$};
		\path[->,font=\scriptsize,>=angle 90]
			(A) edge node[above]{$U_A$} (B)
			(A) edge node[left]{$1$} (A')
			(B) edge node[right]{$f$} (B')
			(A') edge node[below]{ $\hat{f}$} (B');
		%
		\node () at (0.5,0.5) {\scriptsize{$\Downarrow$}};
	\end{tikzpicture}
	}
	\]
  such that the following equations hold.
	\begin{equation}
	\label{eq:CompanionEq}
	\raisebox{-0.5\height}{
	\begin{tikzpicture}
		\node (A) at (0,2) {$A$};
		\node (B) at (1.1,2) {$A$};
		\node (A') at (0,1) {$A$};
		\node (B') at (1.1,1) {$B$};
		\node (A'') at (0,0) {$B$};
		\node (B'') at (1.1,0) {$B$};
		\path[->,font=\scriptsize,>=angle 90]
			(A) edge node[left]{$1$} (A')
			(A') edge node[left]{$f$} (A'')
			(B) edge node[right]{$f$} (B')
			(B') edge node[right]{$1$} (B'')
			(A) edge node[above]{$U_A$} (B)
			(A') edge  (B')
			(A'') edge node[below]{$U_B$} (B'');
		%
		\draw[line width=2mm,white] (0.5,.925) -- (0.5,1.075);
		\node () at (0.5,0.5) {\scriptsize{$\Downarrow$}};
		\node () at (0.5,1.5) {\scriptsize{$\Downarrow$}};
		\node () at (0.5,1) {\scriptsize $\widehat{f}$};
	\end{tikzpicture}
	}
	\raisebox{-0.5\height}{=}
	\raisebox{-0.5\height}{
	\begin{tikzpicture}
		\node (A) at (0,1) {$A$};
		\node (B) at (1,1) {$A$};
		\node (A') at (0,0) {$B$};
		\node (B') at (1,0) {$B$};
		\path[->,font=\scriptsize,>=angle 90]
		(A) edge node[left]{$f$} (A')
		(B) edge node[right]{$f$} (B')
		(A) edge node[above]{$U_A$} (B)
		(A') edge node[below]{$U_B$} (B');
		%
		\node () at (0.5,0.5) {\scriptsize{$\Downarrow U_f$}};
	\end{tikzpicture}
	}
	\raisebox{-0.5\height}{\text{   and   }}
	\raisebox{-0.5\height}{
	\begin{tikzpicture}
		\node (A) at (0,1) {$A$};
		\node (A') at (0,0) {$A$};
		\node (B) at (1.4,1) {$A$};
		\node (B') at (1.4,0) {$B$};
		\node (C) at (2.8,1) {$B$};
		\node (C') at (2.8,0) {$B$};
		\node (A'') at (0,-1) {$A$};
		\node (C'') at (2.8,-1) {$B$};
		\path[->,font=\scriptsize,>=angle 90]
			(A) edge node[left]{$1$} (A')
			(B) edge node [left] {$f$} (B')
			(C) edge node[right]{$1$} (C')
			(A) edge node[above]{$U_A$} (B)
			(B) edge node[above]{$\hat{f}$} (C)
			(A') edge (B')
			(B') edge (C')
			(A'') edge node[below]{$\hat{f}$} (C'')
			(A') edge node[left]{$1$} (A'')
			(C') edge node[right]{$1$} (C'');
		%
	\draw[line width=2mm,white] (0.7,-.05) -- (0.7,.05);
	\draw[line width=4.4mm,white] (2.1,-.05) -- (2.1,.05);
		\node () at (0.7,0.5) {\scriptsize{$\Downarrow$}};
		\node () at (2.1,0.5) {\scriptsize{$\Downarrow$}};
		\node () at (1.4,-0.6) {\scriptsize{$\Downarrow \lambda_{\hat{f}}$}};
		\node () at (0.7,0) {$\hat{f}$};
		\node () at (2.1,0) {\scriptsize{$U_B$}};

	\end{tikzpicture}
	}
	\raisebox{-0.5\height}{=}
	\raisebox{-0.5\height}{
	\begin{tikzpicture}
	     \node (A0) at (0,2) {$A$};
	     \node (B0) at (1,2) {$A$};
		\node (C0) at (2,2) {$B$};
		\node (A) at (0,1) {$A$};
		\node (C) at (2,1) {$B$};
		%
		\path[->,font=\scriptsize,>=angle 90]
			(A0) edge node[above]{$U_A$} (B0)
			(B0) edge node[above]{$\hat{f}$} (C0)
			(A0) edge node[left]{$1$} (A)
			(C0) edge node[right]{$1$} (C)
			(A) edge node[below]{$\hat{f}$} (C);
		%
		\node () at (1,1.4) {\scriptsize{$\Downarrow \rho_{\hat{f}}$}};
	\end{tikzpicture}
	}
	\end{equation}
  A \define{conjoint} of $f$, denoted $\fchk \maps B \to A$, is a
  companion of $f$ in the double category
  obtained by reversing the horizontal 1-cells, but not the vertical
  1-morphisms, of $\lD$.
\end{defn}
\noindent

\begin{defn}
\label{defn:fibrant}
We say that a double category is \define{fibrant} if every vertical
1-morphism has both a companion and a conjoint.
\end{defn}

\begin{thm}{\cite[Theorem 1.1]{HS}}
\label{Shulhorizontalbicat}
If $\lD$ is a fibrant monoidal double category, then its horizontal bicategory $\bD$ is a monoidal bicategory. If $\lD$ is braided or symmetric, then so is $\bD$.
\end{thm}

\end{document}